\documentclass[a4paper,12pt]{amsart}
\usepackage[a4paper,hscale=0.8,vscale=0.75,centering]{geometry}
\usepackage{fullpage}
\usepackage{url}
\usepackage{xcolor}
\usepackage{bbm}
\usepackage{hyperref}       % hyperlinks
\usepackage[numbers]{natbib}
\usepackage{mathtools}
\usepackage{amssymb}
\usepackage{amsmath}
\usepackage{algorithm}
\usepackage{algorithmic}
\usepackage{enumerate}

\DeclareMathOperator*{\argmin}{argmin}

\newtheorem{theorem}{Theorem}[section]

\newtheorem{lemma}[theorem]{Lemma}
\newtheorem{corollary}[theorem]{Corollary}
\newtheorem{proposition}[theorem]{Proposition}

\theoremstyle{definition}
\newtheorem{definition}[theorem]{Definition}

\newtheorem{assumption}[theorem]{Assumption}

\theoremstyle{remark}
\newtheorem{remark}[theorem]{Remark}

\numberwithin{equation}{section}

\begin{document}

\title{Linear convergence of proximal descent schemes on the Wasserstein space}
\author{Razvan-Andrei Lascu}
\address{Center for Advanced Intelligence Project, RIKEN, Tokyo, Japan}
\email{razvan-andrei.lascu@riken.jp}

\author{Mateusz B. Majka}
\address{School of Mathematical and Computer Sciences, Heriot-Watt University, Edinburgh, UK, and Maxwell Institute for Mathematical Sciences, Edinburgh, UK}
\email{m.majka@hw.ac.uk}

\author{David \v{S}i\v{s}ka}
\address{School of Mathematics, University of Edinburgh, UK, and Simtopia, UK}
\email{d.siska@ed.ac.uk}

\author{\L ukasz Szpruch}
\address{School of Mathematics, University of Edinburgh, UK, The Alan Turing Institute, UK and Simtopia, UK}
\email{l.szpruch@ed.ac.uk}


\dedicatory{}

\begin{abstract}
We investigate proximal descent methods, inspired by the minimizing movement scheme introduced by Jordan, Kinderlehrer and Otto, for optimizing entropy-regularized functionals on the Wasserstein space.
We establish linear convergence under flat convexity assumptions, thereby relaxing the common reliance on geodesic convexity. Our analysis circumvents the need for discrete-time adaptations of the Evolution Variational Inequality (EVI). Instead, we leverage a uniform logarithmic Sobolev inequality (LSI) and the entropy ``sandwich" lemma, extending the analysis from \cite{Nitanda2022ConvexAO, chizat2022meanfield}. The major challenge in the proof via LSI is to show that the relative Fisher information is well-defined at every step of the scheme. Since the relative entropy is not Wasserstein differentiable, we prove that along the scheme the iterates belong to a certain class of Sobolev regularity, and hence the relative entropy has a unique Wasserstein sub-gradient, and that the relative Fisher information is indeed finite.
\end{abstract}

\maketitle
\section{Introduction}
\label{sec:Introduction}
We consider the problem of minimizing an entropy-regularized flat-convex function
\begin{equation}
\label{eq:mean-field-min-problem}
\min_{\mu \in \mathcal{P}_2(\mathbb R^d)} F^{\sigma}(\mu), \text{ with } F^{\sigma}(\mu)\coloneqq F(\mu) + \sigma \operatorname{KL}(\mu|\rho),
\end{equation}
over the Wasserstein space $\left(\mathcal{P}_2(\mathbb R^d), \mathcal{W}_2\right)$, where $F: \mathcal{P}_2(\mathbb R^d) \to \mathbb R$ is a function bounded from below on $\mathcal{P}_2(\mathbb R^d),$ $\rho \in \mathcal{P}_2(\mathbb{R}^d)$ is a reference probability measure, $\sigma > 0$ is a regularization parameter and $\operatorname{KL}$ is the KL-divergence (relative entropy). Such optimization problems are motivated by many applications in data science and machine learning, including the task of training two-layer neural networks (NNs) in the mean-field regime \cite{nitanda2017stochasticparticlegradientdescent, Chizat2018OnTG, Mei2018AMF, Rotskoff2018TrainabilityAA,sirignano} and reinforcement learning \cite{leahy,ruiyi,yamamoto24a}.

In this work, we tackle \eqref{eq:mean-field-min-problem} from the perspective of discrete-time stepping schemes by proposing the following Jordan--Kinderlehrer--Otto (JKO)-based optimization methods: proximal point, prox-linear and proximal gradient,\footnote{We maintain the terminology used for analogous methods in finite-dimensional optimization; see e.g. \cite{Drusvyatskiy2016EfficiencyOM,parikh}. Indeed, our naming convention is justified since the JKO step \eqref{eq:JKO} can be viewed as a proximal operator on the Wasserstein space $(\mathcal{P}_2(\mathbb R^d), \mathcal{W}_2)$.} for which we prove linear convergence to the minimizer of $F^{\sigma},$ without requiring that $F$ is geodesically convex. For $\rho \propto e^{-U}$ with a sufficiently regular potential $U: \mathbb{R}^d \to \mathbb{R}$ and $F = 0$, given a step-size $\tau > 0$ and starting from $\mu^0 \in \mathcal{P}_2(\mathbb R^d)$, the JKO scheme, also called minimizing movement scheme or proximal descent in the Wasserstein space, was originally introduced in \cite{jko} and constructs a sequence $\left(\mu^n\right)_{n \in \mathbb N} \subset \mathcal{P}_2(\mathbb R^d)$ by the update rule
\begin{equation}
\label{eq:JKO}
\mu^{n+1} = \argmin_{\mu \in \mathcal{P}_2(\mathbb R^d)} \left\{\sigma \operatorname{KL}(\mu|\rho) + \frac{1}{2\tau}\mathcal{W}_2^2(\mu, \mu^n)\right\},
\end{equation}
where $\mathcal{W}_2$ is the $L^2$-Wasserstein distance. 
Note that, by replacing $\rho$ in~\eqref{eq:JKO} with $e^{-\sigma^{-1}f-U}$ for some function $f:\mathbb{R}^d \to \mathbb{R}$, and normalizing appropriately, \eqref{eq:JKO} covers optimization problems for functions of the form $F^{\sigma}(\mu) = \int_{\mathbb{R}^d} f \mathrm{d}\mu + \sigma \operatorname{KL}(\mu|e^{-U})$.

Numerical methods for implementing the JKO scheme \eqref{eq:JKO} were proposed in \cite{Benamou,Benamou2014DiscretizationOF,CarrilloCraigPatacchini2019,CancesGallouetTodeschi2020,natale2020tpfafinitevolumeapproximation,Carrillo2022PrimalDual,hraivoronska2025convergencefullydiscretejko}. A survey of \cite{Benamou,Benamou2014DiscretizationOF} is also included in \cite[Section 4.7]{santambrogio}. Moreover, if the initial measure $\mu^0$ and the target measure $\rho$ are both Gaussian, it is shown in \cite[Section 3; Example 5]{Wibisono2018SamplingAO} that the update step in \eqref{eq:JKO} can be computed in closed form. Note that, however, these works do not cover the case of non-linear $F$.

Another important direction in the numerical computation of the JKO scheme employs neural networks. For instance, \cite{mokrov2021largescale} utilized input-convex neural networks (ICNNs) to compute each update of the JKO scheme. Subsequently, \cite{Yao} reformulated the JKO scheme \eqref{eq:JKO} as an optimization problem over the space of transport maps and further showed that accurate approximations of these maps can be computed using standard neural networks. Building on this idea, \cite{cheng} applied the method in the context of flow-based generative models.

\subsection{JKO-based stepping schemes}\label{sec:JKO1}
A natural approach to solving \eqref{eq:mean-field-min-problem} for a general $F$ is to start with the proximal point scheme
\begin{equation}
\label{eq:implicit-JKO}
\mu^{n+1} = \argmin_{\mu \in \mathcal{P}_2(\mathbb R^d)} \left\{F(\mu) + \sigma\operatorname{KL}(\mu|\rho) + \frac{1}{2\tau}\mathcal{W}_2^2(\mu, \mu^n)\right\}.
\end{equation}
However, this scheme requires one to solve, at each step, a convex but nonlinear minimization problem and hence is mostly of theoretical interest. Its advantage on the theoretical level is that it lends itself to a clean convergence proof with fewest regularity assumptions, which is why we include analysis of this scheme.
 
A more practical scheme can be created by linearizing $F$ around $\mu^n$ and leveraging the fact that the Wasserstein penalty term $\mathcal{W}_2^2(\mu, \mu^n)$ ensures that the linearization is accurate enough, provided there is appropriate regularity of $F$.
Thus, we define the prox-linear scheme
\begin{equation}
\label{eq:semi-implicit-JKO}
\mu^{n+1} = \argmin_{\mu \in \mathcal{P}_2(\mathbb R^d)} \left\{\int_{\mathbb R^d} \frac{\delta F}{\delta \mu}(\mu^n, x)(\mu-\mu^n)(\mathrm{d}x) + \sigma \operatorname{KL}(\mu|\rho) + \frac{1}{2\tau}\mathcal{W}_2^2(\mu, \mu^n)\right\}.
\end{equation}
Since the map $\mu \mapsto \int_{\mathbb R^d}\tfrac{\delta F}{\delta \mu}(\mu^n, x)\mu(\mathrm{d}x)$ is linear, by replacing $\rho$ in~\eqref{eq:JKO} with appropriately normalized $e^{-\sigma^{-1}\tfrac{\delta F}{\delta \mu}(\mu^n,\cdot)-U}$, we see that \eqref{eq:JKO} covers \eqref{eq:semi-implicit-JKO} as a special case (cf.\ the remark below \eqref{eq:JKO}).
In other words, one could view \eqref{eq:semi-implicit-JKO} as corresponding to \eqref{eq:JKO} with a relative entropy of the form $\operatorname{KL}(\mu|\Phi_{\sigma}[\mu^n])$, where $\Phi_{\sigma}[\mu^n] \propto e^{-\sigma^{-1}\tfrac{\delta F}{\delta \mu}(\mu^n,\cdot) - U}$. Thus, \eqref{eq:semi-implicit-JKO} can be implemented numerically as discussed in \cite[Section 4.7]{santambrogio}. More recently, \cite{teter2024proximal} proposed an algorithm for solving \eqref{eq:semi-implicit-JKO} in the case where $F(\mu) = \frac{1}{2}\int_{\mathbb R^d}\int_{\mathbb R^d} W(x,x') \mu(\mathrm{d}x)\mu(\mathrm{d}x'),$ for an interaction potential $W:\mathbb R^d \times \mathbb R^d \to \mathbb R$ that satisfies $W(x,x') = W(x',x),$ for all $x,x' \in \mathbb R^d.$ Several numerical experiments were performed but no convergence rates for the algorithm were proved.

Another natural approach is to consider the proximal gradient algorithm, which in our context translates to updating $\mu^n$ by a pushforward, which in fact we will show is an optimal transport map for $F$ regular enough and sufficiently small $\tau,$ and updating the resulting measure via a JKO step. Thus, the proximal gradient scheme is 
\begin{equation}
\label{eq:proximal-JKO}
\begin{split}
&\nu^{n+1} = \left(I_d-\tau\nabla_\mu F(\mu^n)(\cdot)\right)_{\#}\mu^n, \\
&\mu^{n+1} = \argmin_{\mu \in \mathcal{P}_2(\mathbb R^d)} \left\{\sigma \operatorname{KL}(\mu|\rho) + \frac{1}{2\tau}\mathcal{W}_2^2(\mu, \nu^{n+1})\right\},
\end{split}
\end{equation}
where $\nabla_\mu F(\mu^n)$ denotes the Wasserstein gradient of $F$ at $\mu^n$ (cf. Definition \ref{def:wass-differentiability}). A proximal scheme related to~\eqref{eq:proximal-JKO} was recently introduced in \cite{korbaproximal} for the case where $F=0$ in \eqref{eq:mean-field-min-problem}. This method splits $\operatorname{KL}(\mu|\rho)$ into the sum of $\int_{\mathbb{R}^d} U d\mu$ and the entropy $H(\mu)$ (cf. Subsection \ref{subsec:notation}) and then implements a gradient descent step on $U$ and a JKO update step for $H$, see the discussion in Section \ref{sec:RelatedWorks} for more details.

It is worth emphasizing that an ``explicit'' scheme in which both $F$ and $\operatorname{KL}(\cdot|\rho)$ are linearized around $\mu^n$ is not expected to converge due to the non-smoothness of the relative entropy in the Wasserstein space \cite[Subsection 3.1.1]{Wibisono2018SamplingAO}. Recently, \cite{xu2024forwardeulertimediscretizationwassersteingradient} provided two counterexamples for which updating $\mu^n$ by the pushforward
\begin{equation*}
    \mu^{n+1} = \left(I_d - \tau\nabla_\mu \operatorname{KL}(\mu^n|\rho)\right)_{\#}\mu^n,
\end{equation*}
fails to converge for $\mu^0$ and $U$ appropriately chosen.

Utilizing techniques from optimal transport and the theory of gradient flows on the space of probability measures, we prove that the iterates $(\mu^n)_{n \in \mathbb N}$ generated by each of the schemes \eqref{eq:implicit-JKO}, \eqref{eq:semi-implicit-JKO} and \eqref{eq:proximal-JKO} converge linearly to the minimizer of $F^{\sigma}.$ 

\subsection{Connection to the Wasserstein gradient flow}
As $\tau \to 0,$ schemes \eqref{eq:implicit-JKO}, \eqref{eq:semi-implicit-JKO} and \eqref{eq:proximal-JKO} are expected to recover the Wasserstein gradient flow of $F^{\sigma},$ given by
\begin{equation}
\label{eq:wass-grad-flow-F-sigma}
\partial_t \mu = \nabla \cdot \left(\left(\nabla \frac{\delta F}{\delta \mu}(\mu, \cdot) + \sigma \nabla U\right)\mu\right) + \sigma\Delta \mu, \quad \mu|_{t=0} \coloneqq \mu^0 \in \mathcal{P}_2(\mathbb R^d).
\end{equation}
In continuous time, there are two potential approaches to show that \eqref{eq:wass-grad-flow-F-sigma} converges with rate $\mathcal{O}(e^{-\kappa t}),$ for $\kappa > 0,$ to the minimizer $\mu_\sigma^*$ of $F^{\sigma}.$ The appropriate approach depends on $F.$

On the one hand, assume that $F$ is geodesically convex and $U$ is $\beta$-strongly-convex for $\beta > 0.$ Then $F^{\sigma}$ is $\sigma\beta$-geodesically convex, which implies that
\begin{equation*}
    F^{\sigma}(\mu_\sigma^*) - F^{\sigma}(\mu_t) \geq \left\langle \nabla \frac{\delta F^{\sigma}}{\delta \mu}(\mu_t, \cdot), T_{\mu_t}^{\mu_\sigma^*} - I_d\right\rangle_{L_{\mu_t}^2(\mathbb R^d)} + \frac{\sigma \beta}{2}\mathcal{W}_2^2(\mu_t, \mu_\sigma^*),
\end{equation*}
where $T_{\mu_t}^{\mu_\sigma^*}:\mathbb R^d \to \mathbb R^d$ is the optimal transport map from $\mu_t$ to $\mu_\sigma^*,$ provided that it exists. Furthermore, by \cite[Lemma 8.4.7]{ambrosio2008gradient} applied to \eqref{eq:wass-grad-flow-F-sigma}, it holds that
\begin{equation*}
    \frac{1}{2}\frac{\mathrm{d}}{\mathrm{d}t}\mathcal{W}_2^2(\mu_t, \mu_\sigma^*) = \left\langle \nabla \frac{\delta F^{\sigma}}{\delta \mu}(\mu_t, \cdot), T_{\mu_t}^{\mu_\sigma^*} - I_d\right\rangle_{L_{\mu_t}^2(\mathbb R^d)}.
\end{equation*}
Hence one obtains the following Evolution Variational Inequality (EVI, cf.\ \cite[Theorem 11.1.4]{ambrosio2008gradient}) 
\begin{equation*}
\frac{1}{2}\frac{\mathrm{d}}{\mathrm{d}t}\mathcal{W}_2^2(\mu_t, \mu_\sigma^*) \leq -\left(F^{\sigma}(\mu_t) - F^{\sigma}(\mu_\sigma^*)\right) - \frac{\sigma\beta}{2}\mathcal{W}_2^2(\mu_t, \mu_\sigma^*),
\end{equation*}
which implies convergence of \eqref{eq:wass-grad-flow-F-sigma} to $\mu_\sigma^*$ in the Wasserstein distance with rate $\mathcal{O}(e^{-\sigma\beta t})$.

On the other hand, assume that $F$ is flat-convex, which implies that
\begin{equation}
\label{eq:flat-conv-grad-flow}
    F(\mu_\sigma^*) - F(\mu_t) \geq \int_{\mathbb R^d} \frac{\delta F}{\delta \mu}(\mu_t, x)(\mu_\sigma^* - \mu_t)(\mathrm{d}x),
\end{equation} 
and assume that the proximal measure $\Phi_{\sigma}[\mu] \propto e^{-\sigma^{-1}\frac{\delta F}{\delta \mu}(\mu,\cdot) - U}$ satisfies the log-Sobolev inequality (LSI) with a constant $\theta > 0$ for any $\mu \in \mathcal{P}_2(\mathbb R^d).$ By \eqref{eq:flat-conv-grad-flow}, we obtain
\begin{equation}
\label{eq:sandwich-rhs}
    F^{\sigma}(\mu_t) - F^{\sigma}(\mu_\sigma^*) \leq \sigma \operatorname{KL}(\mu_t|\Phi_{\sigma}[\mu_t]),
\end{equation}
which is the right-hand side of the entropy ``sandwich'' lemma (Lemma \ref{lemma:sandwich}). Then via the arguments in \cite{Nitanda2022ConvexAO, chizat2022meanfield} using the LSI and \eqref{eq:sandwich-rhs}, we have
\begin{equation*}
\begin{split}
\frac{\mathrm{d}}{\mathrm{d}t}\left(F^{\sigma}(\mu_t) - F^{\sigma}(\mu_\sigma^*)\right) = -\sigma^2 I(\mu_t|\Phi_{\sigma}[\mu_t]) &\leq -2\theta\sigma^2\operatorname{KL}(\mu_t|\Phi_{\sigma}[\mu_t]) \\
&\leq -2\theta\sigma\left(F^{\sigma}(\mu_t) - F^{\sigma}(\mu_\sigma^*)\right).
\end{split}
\end{equation*}
Hence, convergence of \eqref{eq:wass-grad-flow-F-sigma} to $\mu_\sigma^*$ with rate $\mathcal{O}(e^{-2\theta\sigma t})$ is obtained from Gronwall's lemma. In this setting, the same rate of convergence in the Wasserstein distance then follows immediately from Lemma \ref{lemma:sandwich} and Talagrand's inequality since $\mu_\sigma^* = \Phi_{\sigma}[\mu_\sigma^*].$

We stress that the proof strategy via EVI fails if $F$ is not geodesically convex, and that there are examples of applications where the assumption of geodesic convexity is not satisfied, while flat convexity holds (cf.\ Section \ref{appendix:verify-assumptions}). Motivated by this, in the present paper we adapt the proof via LSI and the ``sandwich'' entropy lemma to discrete-time stepping schemes, and prove linear convergence of \eqref{eq:implicit-JKO}, \eqref{eq:semi-implicit-JKO} and \eqref{eq:proximal-JKO} to $\mu_\sigma^*$. In particular, our proof only requires the notion of flat convexity of $F$ instead of geodesic convexity. 

A related line of research focuses on establishing LSIs for particle approximations of the mean-field Langevin dynamics~\eqref{eq:wass-grad-flow-F-sigma}.
This has received significant attention lately with positive results~\cite{chewi2024uniforminnlogsobolevinequalitymeanfield,wang2024uniformlogsobolevinequalitiesmean,pierre}.

Other examples of optimization problems that are flat convex but not geodesically convex include minimizing distances with respect to a fixed target distribution, a setting common in generative modeling \cite{huang2024generativemodelingminimizingwasserstein2, arbel}. One example is the minimization of the squared $L^2$-Wasserstein distance $F(\mu) = \frac{1}{2}\mathcal{W}_2^2(\mu, \mu^*)$, studied in \cite{huang2024generativemodelingminimizingwasserstein2}, which can be interpreted as training a Generative Adversarial Network (GAN) generator to match a target measure $\mu^*$ in $\mathcal{W}_2$. While $F$ in this case fails to be geodesically convex (see, e.g., \cite[Example 9.1.5]{ambrosio2008gradient}), it is flat convex. Another example, given in \cite{arbel}, is the minimization of the Maximum Mean Discrepancy (MMD) $F(\mu) = \frac{1}{2}\operatorname{MMD}^2(\mu, \mu^*)$. Here, again, $F$ is flat convex but not geodesically convex \cite[Section 3.1]{arbel}. However, neither of these examples includes entropy regularization.

While the theory of continuous-time dynamics is by now relatively well-understood, in the present paper we focus on discrete-time algorithms and we address significant challenges that do not arise in the continuous-time setting.

\subsection{Our contribution}
We propose JKO-based methods for solving the mean-field optimization problem \eqref{eq:mean-field-min-problem}. Our contribution can be summarized as follows:
\begin{itemize}
\item In Theorem \ref{thm:well-posedness} we prove existence and uniqueness of the minimizer for each scheme \eqref{eq:implicit-JKO}, \eqref{eq:semi-implicit-JKO} and \eqref{eq:proximal-JKO}, respectively.
\item In Theorem \ref{thm:optim}, we prove that along the iterates generated by these schemes the relative entropy $\operatorname{KL}(\cdot|\rho)$ admits a unique Wasserstein subgradient, and hence we show that the iterates satisfy first-order optimality conditions.
\item In Theorem \ref{thm:monot-decrease}, we prove that the energy function $F^{\sigma}$ monotonically decreases along the iterates $(\mu^n)_{n \geq 0}$ generated by schemes \eqref{eq:implicit-JKO}, \eqref{eq:semi-implicit-JKO} and \eqref{eq:proximal-JKO}.
\item Finally, in Theorem \ref{thm:exp-conv} and Corollary \ref{corollary:conv-schemes-kl-wass} we prove linear convergence of \eqref{eq:implicit-JKO}, \eqref{eq:semi-implicit-JKO} and \eqref{eq:proximal-JKO} to the minimizer $\mu_\sigma^*$ of $F^{\sigma},$ with respect to $F^{\sigma}(\cdot)-F^{\sigma}(\mu_\sigma^*),$ $\operatorname{KL}(\cdot|\mu_\sigma^*)$ and $\mathcal{W}_2^2(\cdot, \mu_\sigma^*).$ In particular, we show that for each of the schemes there exists $\kappa > 1$ such that
\begin{equation*}
0 \leq F^{\sigma}(\mu^n) - F^{\sigma}(\mu_{\sigma}^*) \leq \kappa^{-n}\left(F^{\sigma}(\mu^0) - F^{\sigma}(\mu_{\sigma}^*)\right),
\end{equation*}
for all $n \in \mathbb N.$
\end{itemize}  

\subsection{Related works}\label{sec:RelatedWorks}
As discussed in the first part of the introduction, \cite{korbaproximal} considered the case where $F=0,$ $\sigma = 1$ and $\rho \propto e^{-U}$ in \eqref{eq:mean-field-min-problem}, with $U$ being strongly convex, $\nabla U$ being $L_U$-Lipschitz continuous, for some $L_U > 0,$ and the proximal scheme
\begin{equation}
\begin{split}
\label{eq:proximal-JKO-2}
&\nu^{n+1} = \left(I_d - \tau\nabla U\right)_{\#}\mu^n, \\
&\mu^{n+1} = \argmin_{\mu \in \mathcal{P}_2^{\lambda}(\mathbb R^d)} \left\{H(\mu) + \frac{1}{2\tau}\mathcal{W}_2^2(\mu, \nu^{n+1})\right\},	
\end{split}	
\end{equation}
with step-size $\tau > 0,$ starting from $\mu^0 \in \mathcal{P}_2^{\lambda}(\mathbb R^d)$, where $H$ is the entropy (cf.\ \eqref{eq:defentropy}). Given that $\tau < L_U^{-1},$ it is proved in \cite[Corollary 11]{korbaproximal} that the iterates $(\mu^n)_{n \in \mathbb N}$ generated by \eqref{eq:proximal-JKO-2} converge with linear rate with respect to $\mu \mapsto \mathcal{W}_2^2(\mu,\rho)$ by establishing a discrete-time variant of the EVI (see \cite[Proposition 8]{korbaproximal}). One of the key ingredients in \cite{korbaproximal} for proving the EVI is the geodesic convexity of $\int_{\mathbb R^d} U \mathrm{d}\mu,$ which follows from assuming strong convexity of $U.$ 

As we have already mentioned, for problem \eqref{eq:mean-field-min-problem} the approach via EVI fails since $F$ is not necessarily geodesically convex. Hence, we prove convergence of \eqref{eq:implicit-JKO}, \eqref{eq:semi-implicit-JKO} and \eqref{eq:proximal-JKO} with linear rate via the LSI and the entropy ``sandwich'' lemma. We work under the assumptions that $F$ is flat-convex and $U$ is convex.

The setting of \cite{korbaproximal} was extended in \cite{luu2024nongeodesicallyconvex} by assuming that $U$ is non-convex but expressed as $U \coloneqq U_1 - U_2,$ where $U_1,U_2:\mathbb R^d \to \mathbb R$ are convex, have quadratic growth and $\nabla U_2$ is $L_{U_2}$-Lipschitz continuous, for some $L_{U_2} > 0.$ In this setup, the proximal scheme considered in \cite{luu2024nongeodesicallyconvex} is
\begin{equation}
\begin{split}
\label{eq:proximal-JKO-3}
&\nu^{n+1} = \left(I_d + \tau\nabla U_2\right)_{\#}\mu^n, \\
&\mu^{n+1} = \argmin_{\mu \in \mathcal{P}_2^{\lambda}(\mathbb R^d)} \left\{\int_{\mathbb R^d}U_1(x)\mu(\mathrm{d}x) + H(\mu) + \frac{1}{2\tau}\mathcal{W}_2^2(\mu, \nu^{n+1})\right\},	
\end{split}	
\end{equation}
with step-size $\tau > 0$ and starting from $\mu^0 \in \mathcal{P}_2^{\lambda}(\mathbb R^d)$. Under the assumption that $\rho \propto e^{-(U_1-U_2)}$ satisfies LSI, it is proved in \cite[Theorem 4.5 2.]{luu2024nongeodesicallyconvex} that the iterates $(\mu^n)_{n \in \mathbb N}$ generated by \eqref{eq:proximal-JKO-3} converge with linear rate with respect to $\mu \mapsto \operatorname{KL}(\mu|\rho)$ and $\mu \mapsto \mathcal{W}_2^2(\mu,\rho)$. However, to apply the results from \cite{luu2024nongeodesicallyconvex}, one would need to verify that $\rho \propto e^{-(U_1-U_2)}$ indeed satisfies LSI, which is not immediate, and may require additional conditions on $U_1$ and $U_2$. Although in our setting the potential $\sigma^{-1}\frac{\delta F}{\delta \mu}(\mu,\cdot) + U$ is also non-convex, we provide sufficient conditions under which the proximal measure $\Phi_{\sigma}[\mu] \propto e^{-\sigma^{-1}\frac{\delta F}{\delta \mu}(\mu,\cdot) - U}$ satisfies LSI for any $\mu \in \mathcal{P}_2(\mathbb R^d)$ and we verify these conditions in concrete examples.

Moreover, the existence and uniqueness of the minimizer for the JKO step in \eqref{eq:proximal-JKO-2} and \eqref{eq:proximal-JKO-3} are just postulated as assumptions in both \cite[Assumption B2]{korbaproximal} and \cite[Assumption 1 (iv)]{luu2024nongeodesicallyconvex}, respectively. In contrast, we provide a fully rigorous proof of the existence and uniqueness of the minimizer for each of our schemes \eqref{eq:implicit-JKO}, \eqref{eq:semi-implicit-JKO} and \eqref{eq:proximal-JKO}. 

We also cover the issue of Wasserstein sub-differentiability of the relative entropy $\operatorname{KL}(\cdot|\rho)$ at every step of each scheme \eqref{eq:implicit-JKO}, \eqref{eq:semi-implicit-JKO} and \eqref{eq:proximal-JKO}. Since the relative entropy is not Wasserstein differentiable (but only Wasserstein sub-differentiable), one needs to prove that along each scheme the iterates belong to a certain class of Sobolev regularity (see, e.g., \cite[Theorem 10.4.13]{ambrosio2008gradient}), and hence the relative entropy $\operatorname{KL}(\cdot|\rho)$ has a unique Wasserstein sub-gradient at $\mu$ given by $\nabla \log \frac{\mathrm{d}\mu}{\mathrm{d}\rho}.$ Moreover, within this class we are guaranteed that the relative Fisher information is finite. These two points, namely existence of unique minimizers for scheme steps and Wasserstein differentiability are in fact technically the most challenging steps of all our convergence proofs.

After the first version of this manuscript appeared on arXiv, we became aware of a concurrent work~\cite{zhu2025convergenceanalysiswassersteinproximal} which independently proved linear convergence of the proximal point scheme~\eqref{eq:implicit-JKO}, albeit only for a very specific choice of $F$ (see \cite[Corollary 3.5]{zhu2025convergenceanalysiswassersteinproximal}). In contrast, our proofs for the proximal point scheme~\eqref{eq:implicit-JKO} work in a more general setting, and, most importantly, we also cover the more practical prox-linear~\eqref{eq:semi-implicit-JKO} and proximal gradient~\eqref{eq:proximal-JKO} schemes.

\section{Preliminaries and assumptions}
\label{sec:prelims}
In this section, we introduce the necessary notations and assumptions used throughout the paper.
\subsection{Notation}
\label{subsec:notation}
By $\mathcal{P}_2(\mathbb R^d)$ we denote the space of probability measures with finite second moment. We equip $\mathcal{P}_2(\mathbb R^d)$ with the 2-Wasserstein distance $\mathcal{W}_2,$ and as common in the literature we refer to the metric space $\left(\mathcal{P}_2(\mathbb R^d), \mathcal{W}_2\right)$ as the Wasserstein space. Let $\mathcal{B}(\mathbb R^d)$ denote the Borel $\sigma$-algebra over $\mathbb R^d.$  For any measures $\mu, \nu$ on $\left(\mathbb R^d, \mathcal{B}(\mathbb R^d)\right),$ we write $\mu \ll \nu$ if $\mu$ is absolutely continuous with respect to $\nu$ and $\mu \sim \nu$ if $\mu$ and $\nu$ are equivalent. The set of absolutely continuous measures in $\mathcal{P}_2(\mathbb R^d)$ with respect to $\nu$ is denoted by $\mathcal{P}_2^{\nu}(\mathbb R^d)\coloneqq\left\{\mu \in \mathcal{P}_2(\mathbb R^d): \mu \ll \nu\right\}.$ We denote by $\lambda$ the Lebesgue measure on $\mathbb R^d.$ For any measures $\mu, \nu \in \mathcal{P}_2(\mathbb R^d),$ the map $T_{\mu}^{\nu}:\mathbb R^d \to \mathbb R^d$ denotes the optimal transport map from $\mu$ to $\nu.$ Let $p \in \{1,2\}.$ For any $\mu \in \mathcal{P}_2(\mathbb R^d)$, let $\left(L_{\mu}^p(\mathbb R^d), \|\cdot\|_{L_{\mu}^p(\mathbb R^d)}\right)$ be the space of $\mathcal{B}(\mathbb R^d)$-measurable functions $f :\mathbb R^d \to \mathbb R^d$ such that $\|f\|_{L_{\mu}^p(\mathbb R^d)} \coloneqq \left(\int_{\mathbb R^d} |f(x)|^p \mu(\mathrm{d}x)\right)^{\frac{1}{p}} < \infty.$ Note that the identity map $I_d:\mathbb R^d \to \mathbb R^d,$ given by $I_d(x)= x,$ for all $x \in \mathbb R^d,$ is an element of $L_{\mu}^2(\mathbb R^d).$ For any $\mu \in \mathcal{P}_2(\mathbb R^d),$ we denote by $\langle \cdot,\cdot \rangle_{L_{\mu}^2(\mathbb R^d)}$ the inner product on the space $L_{\mu}^2(\mathbb R^d).$ Let $W_{\mu}^{1,p}(\mathbb R^d)$ be the weighted Sobolev space of $\mathcal{B}(\mathbb R^d)$-measurable functions $f :\mathbb R^d \to \mathbb R^d$ such that $f \in L_{\mu}^p(\mathbb R^d)$ and $\nabla f \in L_{\mu}^p(\mathbb R^d).$ Let $W_{\lambda,\text{loc}}^{1,1}(\mathbb R^d)$ be the Sobolev space of $\mathcal{B}(\mathbb R^d)$-measurable functions $f :\mathbb R^d \to \mathbb R^d$ such that $f \in L_{\lambda,\text{loc}}^1(\mathbb R^d)$ and $\nabla f \in L_{\lambda,\text{loc}}^1(\mathbb R^d).$ For any $f,g : \mathbb R^d \to \mathbb R^d$, the composition $f\circ g: \mathbb R^d \to \mathbb R^d$ is denoted by $f(g).$

For the Lebesgue measure $\lambda$ on $\mathbb R^d,$ the entropy $H:\mathcal{P}_2(\mathbb R^d) \to [0, \infty]$ is given for any $\mu \in \mathcal{P}_2(\mathbb R^d)$ by 
\begin{equation}\label{eq:defentropy}
H(\mu) \coloneqq \begin{cases}
\int_{\mathbb R^d} \log\frac{\mathrm{d}\mu}{\mathrm{d}\lambda}(x) \mu(\mathrm{d}x), & \mu \in \mathcal{P}_2^{\lambda}(\mathbb R^d),\\
+\infty, & \text{ else}.
\end{cases}
\end{equation}

For $\rho \in \mathcal{P}_2(\mathbb R^d),$ the relative entropy $\operatorname{KL}(\cdot | \rho):\mathcal{P}_2(\mathbb R^d) \to [0, \infty]$ with respect to $\rho$ is given for any $\mu \in \mathcal{P}_2(\mathbb R^d)$ by
\begin{equation*}
\operatorname{KL}(\mu|\rho) \coloneqq \begin{cases}
\int_{\mathbb R^d} \log\frac{\mathrm{d}\mu}{\mathrm{d}\rho}(x) \mu(\mathrm{d}x), & \mu \in \mathcal{P}_2^{\rho}(\mathbb R^d),\\
+\infty, & \text{ else},
\end{cases}
\end{equation*}
and the relative Fisher information $I(\cdot|\rho):\mathcal{P}_2(\mathbb R^d) \to [0, \infty]$ with respect to $\rho$ is given for any $\mu \in \mathcal{P}_2(\mathbb R^d)$ by
\begin{equation*}
I(\mu|\rho) \coloneqq \begin{cases}
\int_{\mathbb R^d} \left|\nabla \log\frac{\mathrm{d}\mu}{\mathrm{d}\rho}(x)\right|^2 \mu(\mathrm{d}x), & \mu \in \mathcal{P}_2^{\rho}(\mathbb R^d) \text{ and }\sqrt{\frac{\mathrm{d}\mu}{\mathrm{d}\rho}} \in W_{\rho}^{1,2}(\mathbb R^d),\\
+\infty, & \text{ else}.
\end{cases}
\end{equation*}
\begin{assumption}[Flat-convexity of $F$]
\label{ass:convexity}
Assume $F \in \mathcal{C}^1$ (cf. Definition \ref{def:flat-derivative}) is convex on $\mathcal{P}_2(\mathbb R^d)$, i.e., 
for any $\mu$, $\mu' \in \mathcal{P}_2(\mathbb R^d)$ and any $\varepsilon \in [0,1]$,
\begin{equation}\label{eq:convexity_classical}
    F((1-\varepsilon)\mu + \varepsilon \mu') \leq (1-\varepsilon)F(\mu) + \varepsilon F(\mu').
\end{equation}
\end{assumption}
Note that Assumption \ref{ass:convexity} is standard in the mean-field optimization literature (see e.g. \cite{10.1214/20-AIHP1140,Nitanda2022ConvexAO,chizat2022meanfield}), and that $F \in \mathcal{C}^1$ and \eqref{eq:convexity_classical} together imply that
for any $\mu$, $\mu' \in \mathcal{P}_2(\mathbb R^d)$, 
\begin{equation}\label{eq:convexity_via_flat_derivative}
F(\mu') - F(\mu) \geq \int_{\mathbb R^d} \frac{\delta F}{\delta \mu}(\mu, x)(\mu'-\mu)(\mathrm{d}x),
\end{equation}
cf.\ \cite[Lemma 4.1]{10.1214/20-AIHP1140}. In some of our results, we will work with functions $F$ that are convex but not necessarily flat differentiable, in which case we will refer directly to condition \eqref{eq:convexity_classical}.

\begin{assumption}[Log-concavity of $\rho$]
\label{ass:abs-cty-pi}
Suppose $\rho(\mathrm{d}x) \propto e^{-U(x)}\mathrm{d}x \in \mathcal{P}_2(\mathbb R^d)$, with a convex potential $U:\mathbb{R}^d \to \mathbb{R}$.
\end{assumption}
The following standard assumptions are concerned with the Wasserstein regularity of $F,$ in particular the existence of its Wasserstein gradient understood as the Euclidean gradient of the flat derivative, and the Lipschitz continuity of the Wasserstein gradient.
\begin{assumption}[Wasserstein differentiability of $F$]
\label{assumption:L-derivative}
For any $\mu \in \mathcal{P}_2(\mathbb R^d),$ assume the function $\mathbb R^d \ni x \mapsto \frac{\delta F}{\delta \mu}(\mu, x) \in \mathbb R$ is differentiable, the function $\mathbb R^d \ni x \mapsto \nabla\frac{\delta F}{\delta \mu}(\mu, x) \in \mathbb R^d$ has at most linear growth uniformly in $\mu$ and the function $\mathcal{P}_2(\mathbb R^d) \times \mathbb R^d \ni (\mu,x) \mapsto \nabla \frac{\delta F}{\delta \mu}(\mu, x) \in \mathbb R^d$ is jointly continuous in $(\mu,x).$
\end{assumption}
Under Assumption \ref{assumption:L-derivative}, it follows from \cite[Proposition 5.48; Theorem 5.64]{Carmona2018ProbabilisticTO} that $F:\mathcal{P}_2(\mathbb R^d) \to \mathbb R$ is Wasserstein differentiable at $\mu \in \mathcal{P}_2(\mathbb R^d)$ (cf. Definition \ref{def:wass-differentiability}), with derivative $\nabla_\mu F(\mu)(\cdot) \coloneqq \nabla\frac{\delta F}{\delta \mu}(\mu,\cdot).$

\begin{assumption}[Lipschitzness of the Wasserstein gradient of $F$]
\label{ass:lip-intrinsic}
Assume there exists $L_F' > 0$ such that for all $\mu, \mu' \in \mathcal{P}_2(\mathbb R^d)$ and all $x,x' \in \mathbb R^d,$ we have
\begin{equation*}
\left| \nabla_\mu F(\mu')(x') - \nabla_\mu F(\mu)(x)\right| \leq L_F'\left(|x'-x| + \mathcal{W}_2(\mu', \mu)\right).
\end{equation*}
\end{assumption}
Next, we recall the definition of the so-called proximal Gibbs measure, which is a crucial ingredient in proving convergence via LSI.
\begin{definition}[Proximal Gibbs measure; \cite{Nitanda2022ConvexAO,  chizat2022meanfield}]
\label{def:proximal-gibbs}
Let $\rho \in \mathcal{P}_2(\mathbb R^d).$ For any $\mu \in \mathcal{P}_2(\mathbb R^d)$ and any $\sigma > 0,$ define the operator $\Phi_{\sigma}:\mathcal{P}_2(\mathbb R^d) \to \mathcal{P}_2^{\rho}(\mathbb R^d)$ by
\begin{equation}\label{eq:proximal-gibbs}
\Phi_{\sigma}[\mu](\mathrm{d}x) \coloneqq \frac{1}{Z(\mu)}e^{-\frac{1}{\sigma} \frac{\delta F}{\delta \mu}(\mu, x)}\rho(\mathrm{d}x),
\end{equation}
where, for each $\mu \in \mathcal{P}_2(\mathbb R^d),$ $Z(\mu)$ is a normalization constant. We call $\Phi_{\sigma}[\mu]$ the proximal Gibbs measure.
\end{definition}
We now impose the following assumption, requiring that the proximal measure $\Phi_{\sigma}$ satisfies a uniform LSI.
\begin{assumption}[Uniform log-Sobolev inequality for ${\Phi_\sigma[\mu]}$]
\label{assumption:unif-lsi}
    Let $\rho \in \mathcal{P}_2(\mathbb R^d)$. Suppose there exists $C_{\sigma, F, \rho} > 0$ such that for all $\mu \in \mathcal{P}_2(\mathbb R^d)$ it holds that $e^{-\sigma^{-1}\frac{\delta F}{\delta \mu}(\mu,\cdot)}\rho \in L^1(\mathbb R^d)$ and $\Phi_{\sigma}[\mu]$ satisfies LSI with constant $C_{\sigma, F, \rho}.$ That is, for all $\mu \in \mathcal{P}_2(\mathbb R^d),$
\begin{equation}
\label{eq:sobolev-ineq}
\operatorname{KL}\left(\mu | \Phi_{\sigma}[\mu]\right) \leq \frac{1}{2C_{\sigma, F, \rho}}I\left(\mu \big| \Phi_{\sigma}[\mu]\right).
\end{equation}
\end{assumption}
\begin{remark}[Sufficient conditions for Assumption \ref{assumption:unif-lsi}]
\label{remark:sufficient-criteria-lsi}
Assumption \ref{assumption:unif-lsi} can be checked on a case-by-case basis. In Section \ref{appendix:verify-assumptions}, we verify it in a concrete example. We summarize below a few conditions ensuring that $\Phi_{\sigma}[\mu]$ fulfills Assumption \ref{assumption:unif-lsi}:
\begin{itemize}
    \item Suppose $\rho$ satisfies LSI (see Remark \ref{remark:LSI_for_rho} for an overview of criteria implying LSI for a measure $\rho \propto e^{-U}$). Suppose further that $\frac{\delta F}{\delta \mu}(\mu,\cdot) \in L^{\infty}(\mathbb R^d)$ for all $\mu \in \mathcal{P}_2(\mathbb R^d)$. Then $\Phi_{\sigma}[\mu]$ satisfies LSI by the Holley--Stroock criterion \cite{Holley1987LogarithmicSI}.
    \item If $\rho$ satisfies LSI and $\frac{\delta F}{\delta \mu} (\mu, \cdot)$ is Lipschitz continuous for all $\mu \in \mathcal{P}_2(\mathbb R^d),$ then, by \cite[Theorem 2.7 (2)]{guillin}, $\Phi_{\sigma}[\mu]$ satisfies LSI.
    \item Let Assumption \ref{ass:abs-cty-pi} hold and assume $U \in C^2(\mathbb{R}^d)$ and there exist $c_U', R' > 0$ such that $x \cdot \nabla U(x) \geq c_U'|x|^2$ whenever $|x| > R'$. Suppose also that $F:\mathcal{P}_2(\mathbb R^d) \to \mathbb R$ has a second order Wasserstein derivative and there exist $c_F$, $L_F'$,  such that $x \cdot \nabla_\mu F(\mu)(x) \geq c_F|x|^2$ and $\nabla^2_{\mu}F(\mu)(x) = \nabla^2\frac{\delta F}{\delta \mu}(\mu, x) \succeq -L_F'$ for all $|x| > R'$ and $\mu \in \mathcal{P}_2(\mathbb R^d).$ It follows that $x \cdot \left(\sigma^{-1}\nabla_\mu F(\mu)(x) +\nabla U(x)\right) \geq (\sigma^{-1}c_F+c_U')|x|^2,$ for all $|x| > R'.$ Since $U \in C^2(\mathbb R^d),$ then, by Assumption \ref{ass:abs-cty-pi}, $\nabla^2 U \succeq 0,$ which implies $\sigma^{-1}\nabla^2_{\mu}F(\mu) + \nabla^2 U \succeq -\sigma^{-1}L_F',$ for all $\mu \in \mathcal{P}_2(\mathbb R^d),$ so that, by \cite[Corollary 2.1]{CattiauxGuillinWu2009}, $\Phi_{\sigma}[\mu]$ satisfies LSI.
\end{itemize}
\end{remark}
Note that the constant $C_{\sigma, F, \rho}$ in Assumption \ref{assumption:unif-lsi} has to be uniform in $\mu \in \mathcal{P}_2(\mathbb R^d)$. This can be achieved via any of the criteria in Remark \ref{remark:sufficient-criteria-lsi}, as long as the constant bounding $\frac{\delta F}{\delta \mu}(\mu,\cdot)$, the Lipschitz constant of $\frac{\delta F}{\delta \mu}(\mu,\cdot)$, or the constants involved in bounding $\nabla_\mu F(\mu)(\cdot)$ and $\nabla^2_{\mu}F(\mu)(\cdot)$, respectively, are uniform in $\mu \in \mathcal{P}_2(\mathbb R^d)$. We will verify that the constant is indeed uniform in the example in Section \ref{appendix:verify-assumptions}.

Furthermore, according to \cite{OTTO2000361}, since $\Phi_{\sigma}[\mu]$ satisfies \eqref{eq:sobolev-ineq}, it also satisfies Talagrand's inequality, i.e., for any $\mu \in \mathcal{P}_2(\mathbb R^d),$ it holds
\begin{equation}
\label{eq:Talagrand-ineq}
\mathcal{W}_2^2\left(\mu, \Phi_{\sigma}[\mu]\right) \leq \frac{2}{C_{\sigma, F, \rho}}\operatorname{KL}\left(\mu|\Phi_{\sigma}[\mu]\right).
\end{equation}

It is worth stressing that for proving linear convergence of \eqref{eq:implicit-JKO}, \eqref{eq:semi-implicit-JKO} and \eqref{eq:proximal-JKO}, it suffices to show that for each scheme, there exists $\kappa > 1$ such that
\begin{equation*}
F^{\sigma}(\mu^n) - F^{\sigma}(\mu_{\sigma}^*) \leq \kappa^{-n}\left(F^{\sigma}(\mu^0) - F^{\sigma}(\mu_{\sigma}^*)\right),
\end{equation*}
for all $n \in \mathbb N,$ where $\left(\mu^n\right)_n$ are the iterates generated by \eqref{eq:implicit-JKO}, \eqref{eq:semi-implicit-JKO} and \eqref{eq:proximal-JKO}, respectively. Then convergence with respect to $\mu \mapsto \operatorname{KL}(\mu|\mu_{\sigma}^*)$ and $\mu \mapsto \mathcal{W}_2^2(\mu, \mu_{\sigma}^*)$ will follow immediately from Lemma \ref{lemma:sandwich} and Talagrand's inequality \eqref{eq:Talagrand-ineq}. The dependence of other constants on $\kappa$ will be made explicit in the statement of the convergence result.
\section{Main results}
In this section, we present our main results, which are grouped into four theorems, each addressing all three schemes \eqref{eq:implicit-JKO}, \eqref{eq:semi-implicit-JKO} and \eqref{eq:proximal-JKO}. We begin by outlining the general proof steps applicable to all schemes:
\begin{enumerate}
\item  We prove in Theorem \ref{thm:well-posedness} that given $\mu^n \in \mathcal{P}_2^{\rho}(\mathbb R^d),$ each scheme \eqref{eq:implicit-JKO}, \eqref{eq:semi-implicit-JKO} and \eqref{eq:proximal-JKO} admits a unique minimizer $\mu^{n+1} \in \mathcal{P}_2^{\rho}(\mathbb R^d).$
\item Then, leveraging results that connect the metric slope of the relative entropy and the relative Fisher information, we prove in Theorem \ref{thm:optim} that, 
for each scheme \eqref{eq:implicit-JKO}, \eqref{eq:semi-implicit-JKO} and \eqref{eq:proximal-JKO},
$$\mu^{n+1} \in \mathfrak C \coloneqq \left\{m \in \mathcal{P}^{\rho}_2(\mathbb R^d): \frac{\mathrm{d}m}{\mathrm{d}\rho} \in W_{\lambda,\text{loc}}^{1,1}(\mathbb R^d), \sqrt \frac{\mathrm{d}m}{\mathrm{d}\rho} \in W_{\rho}^{1,2}(\mathbb R^d)\right\}.$$ Hence, by Theorem \ref{thm:subdiff-entropy}, we conclude that the relative entropy $\operatorname{KL}(\cdot|\rho)$ has a unique Wasserstein subgradient at $\mu^{n+1},$ and it is given by $\nabla \log \frac{\mathrm{d}\mu^{n+1}}{\mathrm{d}\rho}.$ As a consequence of the previous result, we prove first-order optimality conditions for each scheme \eqref{eq:implicit-JKO}, \eqref{eq:semi-implicit-JKO} and \eqref{eq:proximal-JKO}. The optimality conditions enable us to connect the Fisher information relative to the proximal measures $\Phi_{\sigma}$ given by \eqref{eq:proximal-gibbs} with the Wasserstein distance $\mathcal{W}_2^2(\mu^{n+1}, \mu^n).$
\item In Theorem \ref{thm:monot-decrease}, we establish that the energy function $F^\sigma$ decreases monotonically along the sequence of iterates $(\mu^n)_{n \geq 0}$ generated by the schemes \eqref{eq:implicit-JKO}, \eqref{eq:semi-implicit-JKO}, and \eqref{eq:proximal-JKO}. In the case of \eqref{eq:implicit-JKO}, this monotonicity is an immediate consequence of the scheme’s definition, whereas for \eqref{eq:semi-implicit-JKO} and \eqref{eq:proximal-JKO} it follows from the optimality conditions established in Theorem \ref{thm:optim} combined with the smoothness of $F$ relative to $\mathcal{W}_2^2(\mu^{n+1}, \mu^n)$ (Lemma \ref{ass:relative-smoothness}), and the convexity along (generalized) geodesics of $\operatorname{KL}(\cdot|\rho)$ (Lemma \ref{lemma:control-geod-KL}).
\item Finally, invoking the LSI (Assumption \ref{assumption:unif-lsi}) together with the entropy sandwich lemma (Lemma \ref{lemma:sandwich}), we show in Theorem \ref{thm:exp-conv} that the schemes converge at a linear rate.
\end{enumerate}

We first state the main existence result.
\begin{theorem}[Existence and uniqueness of minimizers]
\label{thm:well-posedness}
Suppose $\rho \in \mathcal{P}_2(\mathbb R^d)$ and $\mu^n \in \mathcal{P}_2^{\rho}(\mathbb R^d).$ 
\begin{enumerate}
    \item If $F$ is weakly lower semi-continuous and convex in the sense of \eqref{eq:convexity_classical}, then there exists a unique minimizer $\mu^{n+1} \in \mathcal{P}_2^{\rho}(\mathbb R^d)$ of \eqref{eq:implicit-JKO}.
    \item If $F \in \mathcal{C}^1$, then there exists a unique minimizer $\mu^{n+1} \in \mathcal{P}_2^{\rho}(\mathbb R^d)$ of \eqref{eq:semi-implicit-JKO}.    
    \item Suppose $F \in \mathcal{C}^1$ and $\rho \sim \lambda$ (i.e., $\rho \ll \lambda$ and $\lambda \ll \rho$). If Assumptions \ref{assumption:L-derivative} and \ref{ass:lip-intrinsic} hold, and if $\tau < \frac{1}{L_F'},$ then there exists a unique minimizer $\mu^{n+1} \in \mathcal{P}_2^{\rho}(\mathbb R^d)$ of \eqref{eq:proximal-JKO}.
\end{enumerate}
Hence, in each of 1., 2., 3., if $\mu^0 \in \mathcal{P}_2^{\rho}(\mathbb R^d),$ then $\left(\mu^n\right)_{n \in \mathbb N} \subset \mathcal{P}_2^{\rho}(\mathbb R^d)$ along the schemes \eqref{eq:implicit-JKO}, \eqref{eq:semi-implicit-JKO}, \eqref{eq:proximal-JKO}, respectively.
\end{theorem}

 Note that in part 1. we do not require $F$ to have a flat derivative (we need convexity for uniqueness, but we can work with its classical definition \eqref{eq:convexity_classical} without using $\frac{\delta F}{\delta \mu}$). In our convergence result for scheme \eqref{eq:implicit-JKO}, i.e., Theorem \ref{thm:exp-conv} 1., we will need Assumption \ref{ass:convexity} (i.e., both convexity and $F \in \mathcal{C}^1$) due to an application of the entropy sandwich lemma (cf.\ the discussion below the statement of Theorem \ref{thm:exp-conv}) but for the sake of precision we formulate the existence result with weaker assumptions (observe that if Assumption \ref{ass:convexity} holds, then the weak lower semi-continuity of $F$ required in Theorem \ref{thm:well-posedness} 1. is automatically satisfied, cf.\ Remark \ref{remark:weak_lower_semi_continuity}).
 Note also that for the proof of existence in part 1., it would be sufficient to have weak lower semi-continuity of $F$, but without convexity one would not necessarily have uniqueness of the minimizer of \eqref{eq:implicit-JKO}. 

 On the other hand, the definitions of schemes \eqref{eq:semi-implicit-JKO} and \eqref{eq:proximal-JKO} utilize derivatives of $F$ and hence it is natural to assume $F \in \mathcal{C}^1$, however, they do not require the convexity of $F$.
 Note that in scheme \eqref{eq:semi-implicit-JKO} the energy function is automatically convex due to linearization, and in scheme \eqref{eq:proximal-JKO} the function $F$ does not appear in the JKO step. 

Note that parts 1. and 2. do not require the reference measure $\rho$ to be absolutely continuous with respect to the Lebesgue measure $\lambda$, however, part 3. requires $\rho$ and $\lambda$ to be equivalent. This is needed due to the application of \cite[Lemma 5.5.3]{ambrosio2008gradient} in Lemma \ref{lemma:pushforward-is-optimal} (see the proof in Section \ref{app:proximal-scheme} for details) and is naturally satisfied when $\rho \propto e^{-U}$ with a measurable potential $U: \mathbb{R}^d \to \mathbb{R}$  (which will be required in our convergence results in Theorem \ref{thm:exp-conv} anyway).

Assumptions \ref{assumption:L-derivative} and \ref{ass:lip-intrinsic} are concerned with the regularity of the Wasserstein gradient $\nabla_{\mu} F$ and are required in part 3. where we work with the proximal gradient scheme \eqref{eq:proximal-JKO} whose gradient descent step involves $\nabla_{\mu} F$.

The proof of Theorem \ref{thm:well-posedness} is split into three parts that can be found in Sections \ref{app:implicit-scheme}, \ref{app:semi-implicit-scheme} and \ref{app:proximal-scheme}, each corresponding to one of the schemes \eqref{eq:implicit-JKO}, \eqref{eq:semi-implicit-JKO} and \eqref{eq:proximal-JKO}, respectively. We now state the result on the Wasserstein sub-differentiability of the energy function, which includes also first order optimality conditions.

\begin{theorem}[Wasserstein sub-differentiability and optimality conditions]
\label{thm:optim}
    Let $F \in \mathcal{C}^1$, $\rho \in \mathcal{P}_2(\mathbb{R}^d)$, $\mu^n \in \mathcal{P}_2^{\rho}(\mathbb R^d)$ and let Assumptions \ref{ass:abs-cty-pi} and \ref{assumption:L-derivative} hold. 
    \begin{enumerate}
    \item If $F$ is weakly lower semi-continuous and convex in the sense of \eqref{eq:convexity_classical}, 
    then the unique minimizer $\mu^{n+1} \in \mathcal{P}_2^{\rho}(\mathbb R^d)$ of \eqref{eq:implicit-JKO} belongs to $\mathfrak C$ and satisfies
    \begin{equation}
    \label{eq:optim-implicit}
    \nabla_\mu F(\mu^{n+1}) + \sigma \nabla \log \frac{\mathrm{d}\mu^{n+1}}{\mathrm{d}\rho} = \frac{1}{\tau}\left(T_{\mu^{n+1}}^{\mu^n} - I_d\right), \quad \mu^{n+1}\text{-a.e.}
    \end{equation}
    \item The unique minimizer $\mu^{n+1} \in \mathcal{P}_2^{\rho}(\mathbb R^d)$ of \eqref{eq:semi-implicit-JKO} belongs to $\mathfrak C$ and satisfies
    \begin{equation}
    \label{eq:optim-semi-implicit}
    \nabla_\mu F(\mu^{n}) + \sigma \nabla \log \frac{\mathrm{d}\mu^{n+1}}{\mathrm{d}\rho} = \frac{1}{\tau}\left(T_{\mu^{n+1}}^{\mu^n} - I_d\right), \quad \mu^{n+1}\text{-a.e.}
    \end{equation}
    \item If Assumption \ref{ass:lip-intrinsic} holds and if $\tau < \frac{1}{L_F'},$ then the unique minimizer $\mu^{n+1} \in \mathcal{P}_2^{\rho}(\mathbb R^d)$ of \eqref{eq:proximal-JKO} belongs to $\mathfrak C$ and satisfies
    \begin{equation}
    \label{eq:optim-proximal}
    \sigma \nabla \log \frac{\mathrm{d}\mu^{n+1}}{\mathrm{d}\rho} = \frac{1}{\tau}\left(T_{\mu^{n+1}}^{\nu^{n+1}} - I_d\right), \quad \mu^{n+1}\text{-a.e.}
    \end{equation}
    \end{enumerate}
    Hence, in each of 1., 2., 3., if $\mu^0 \in \mathcal{P}_2^{\rho}(\mathbb R^d),$ then $\left(\mu^n\right)_{n \in \mathbb N} \subset \mathfrak C$ along the schemes \eqref{eq:implicit-JKO}, \eqref{eq:semi-implicit-JKO}, \eqref{eq:proximal-JKO}, and moreover \eqref{eq:optim-implicit}, \eqref{eq:optim-semi-implicit}, \eqref{eq:optim-proximal} hold, respectively.
\end{theorem}
The convexity of $U$ in Assumption \ref{ass:abs-cty-pi} is required in order to apply Theorem \ref{thm:subdiff-entropy}, which ensures that the iterates of the schemes \eqref{eq:implicit-JKO}, \eqref{eq:semi-implicit-JKO} and \eqref{eq:proximal-JKO} belong to the Sobolev regularity class $\mathfrak C$ (see the discussion after Theorem \ref{thm:subdiff-entropy} in Appendix \ref{app:OT}). Assumption \ref{ass:abs-cty-pi} also implies that $\rho$ is absolutely continuous with respect to $\lambda$, allowing us to invoke Theorem \ref{thm:existence-optimal-coupling} to guarantee existence and uniqueness of optimal transport maps along the schemes. Assuming that $F \in \mathcal{C}^1$, as well as that Assumption \ref{assumption:L-derivative} holds, is now required for all schemes, since the first order conditions in parts 1. and 2. involve $\nabla_{\mu} F$.

The proof of Theorem \ref{thm:optim} is, similarly to the proof of Theorem \ref{thm:well-posedness}, split into three parts that can be found in Sections \ref{app:implicit-scheme}, \ref{app:semi-implicit-scheme} and \ref{app:proximal-scheme} for the respective schemes. Next, we prove that the energy function decreases along the iterates $(\mu^n)_{n \geq 0}$ produced by the schemes \eqref{eq:implicit-JKO}, \eqref{eq:semi-implicit-JKO} and \eqref{eq:proximal-JKO}.
\begin{theorem}[Monotonic decrease of $F^\sigma$ along the schemes \eqref{eq:implicit-JKO}, \eqref{eq:semi-implicit-JKO}, \eqref{eq:proximal-JKO}]
\label{thm:monot-decrease}
Let $\rho \in \mathcal{P}_2(\mathbb{R}^d)$, $\mu^0 \in \mathcal{P}_2^{\rho}(\mathbb R^d).$
\begin{enumerate}
\item  If $F$ is weakly lower semi-continuous and convex in the sense of \eqref{eq:convexity_classical}, then for the iterates $\left(\mu^n\right)_{n \geq 1}$ of \eqref{eq:implicit-JKO}, we have
$F^{\sigma}(\mu^{n+1}) \leq F^{\sigma}(\mu^{n}), \text{ for all } n \in \mathbb N.$
\item If $F \in \mathcal{C}^1,$ Assumptions \ref{ass:abs-cty-pi}, \ref{assumption:L-derivative} and \ref{ass:lip-intrinsic} hold, and if $\tau < \frac{1}{2L_F'},$ then for the iterates $\left(\mu^n\right)_{n \geq 1}$ of \eqref{eq:semi-implicit-JKO}, we have
$F^{\sigma}(\mu^{n+1}) \leq F^{\sigma}(\mu^{n}), \text{ for all } n \in \mathbb N.$
\item If $F \in \mathcal{C}^1,$ Assumptions \ref{ass:abs-cty-pi}, \ref{assumption:L-derivative} and \ref{ass:lip-intrinsic} hold, and if $\tau < \frac{1}{L_F'},$ then for the iterates $\left(\mu^n\right)_{n \geq 1}$ of \eqref{eq:proximal-JKO}, we have
$F^{\sigma}(\mu^{n+1}) \leq F^{\sigma}(\mu^{n}), \text{ for all } n \in \mathbb N.$
\end{enumerate}
\end{theorem}
We emphasize that, for the scheme \eqref{eq:implicit-JKO}, the monotonic decay of the energy function is a direct consequence of the scheme’s definition, once well-posedness has been proved. In particular, unlike for the other two schemes, this monotonicity does not rely on the optimality condition provided in Theorem \ref{thm:optim}.
Because of the linearization of $F$ in \eqref{eq:semi-implicit-JKO} and the gradient descent step in \eqref{eq:proximal-JKO}, the analysis in parts 2. and 3. requires $F$ to be smooth with respect to the Wasserstein distance (cf. Lemma \ref{ass:relative-smoothness}), and hence in both parts 2. and 3. we impose Assumption \ref{ass:lip-intrinsic}. The proof of Theorem \ref{thm:monot-decrease} can be found in Section \ref{sec:monot-proof}. Finally, we present our main convergence result.

\begin{theorem}[Linear convergence of the schemes \eqref{eq:implicit-JKO}, \eqref{eq:semi-implicit-JKO}, \eqref{eq:proximal-JKO}]
\label{thm:exp-conv}
Let $F \in \mathcal{C}^1$, $\rho \in \mathcal{P}_2(\mathbb{R}^d)$, $\mu^0 \in \mathcal{P}_2^{\rho}(\mathbb R^d)$ and let Assumptions \ref{ass:convexity}, \ref{ass:abs-cty-pi}, \ref{assumption:L-derivative} and \ref{assumption:unif-lsi} hold.
\begin{enumerate}
\item For the iterates $\left(\mu^n\right)_{n \geq 1}$ of \eqref{eq:implicit-JKO}, we have
\begin{equation*}
F^{\sigma}(\mu^{n}) - F^{\sigma}(\mu_{\sigma}^*) \leq \left(1 + \tau \sigma C_{\sigma,F,\rho}\right)^{-n}\left(F^{\sigma}(\mu^{0}) - F^{\sigma}(\mu_{\sigma}^*)\right), \text{ for all } n \in \mathbb N.
\end{equation*}
\item If Assumption \ref{ass:lip-intrinsic} holds and if $\tau < \frac{1}{2L_F'},$ then for the iterates $\left(\mu^n\right)_{n \geq 1}$ of \eqref{eq:semi-implicit-JKO}, we have
\begin{equation*}
F^{\sigma}(\mu^{n}) - F^{\sigma}(\mu_{\sigma}^*) \leq \left(1 + \frac{\tau \sigma \left(1 - 2\tau L_F'\right)}{1 + (\tau L_F')^2}C_{\sigma,F,\rho}\right)^{-n}\left(F^{\sigma}(\mu^{0}) - F^{\sigma}(\mu_{\sigma}^*)\right), \text{ for all } n \in \mathbb N.
\end{equation*}
\item If Assumption \ref{ass:lip-intrinsic} holds and if $\tau < \frac{1}{L_F'},$ then for the iterates $\left(\mu^n\right)_{n \geq 1}$ of \eqref{eq:proximal-JKO}, we have
\begin{equation*}
F^{\sigma}(\mu^{n}) - F^{\sigma}(\mu_{\sigma}^*) \leq \left(1 + \frac{\tau \sigma\left(1 - \tau L_F'\right)}{1 + 4(\tau L_F')^2}C_{\sigma,F,\rho}\right)^{-n}\left(F^{\sigma}(\mu^{0}) - F^{\sigma}(\mu_{\sigma}^*)\right), \text{ for all } n \in \mathbb N.
\end{equation*}
\end{enumerate}
\end{theorem}
Assumption \ref{ass:convexity} is needed in all parts of Theorem \ref{thm:exp-conv}, since the proofs rely on Lemma \ref{lemma:sandwich}, where flat differentiability and convexity of $F$ are necessary. 
Moreover, the proofs of all parts crucially rely on the proximal measure $\Phi_{\sigma}[\mu]$ satisfying the uniform log-Sobolev inequality (Assumption \ref{assumption:unif-lsi}).

The proof of Theorem \ref{thm:exp-conv} can be found in Section \ref{sec:convergence-proof}. As mentioned in Section \ref{sec:prelims}, once we have Theorem \ref{thm:exp-conv}, the convergence with respect to $\mu \mapsto \operatorname{KL}(\mu|\mu_{\sigma}^*)$ and $\mu \mapsto \mathcal{W}_2^2(\mu, \mu_{\sigma}^*)$ follows, with the same rate, from Lemma \ref{lemma:sandwich} and the Talagrand inequality \eqref{eq:Talagrand-ineq}. 
\begin{corollary}[Linear convergence in $\operatorname{KL}$ and $\mathcal{W}_2^2$]
\label{corollary:conv-schemes-kl-wass}
For $\kappa > 1$ corresponding to each rate in 1., 2. and 3. in Theorem \ref{thm:exp-conv}, we obtain 
\begin{equation*}
\begin{split}
\operatorname{KL}\left(\mu^n|\mu_{\sigma}^*\right) & \leq \frac{1}{\sigma}\kappa^{-n}\left(F^{\sigma}(\mu^{0}) - F^{\sigma}(\mu_{\sigma}^*)\right)\text{ and }\\
\mathcal{W}_2^2\left(\mu^n, \mu_{\sigma}^*\right) & \leq \frac{2}{\sigma C_{\sigma,F,\rho}}\kappa^{-n}\left(F^{\sigma}(\mu^{0}) - F^{\sigma}(\mu_{\sigma}^*)\right), \text{ for all } n \in \mathbb N.
\end{split}
\end{equation*}
\end{corollary}
Note that a detailed comparison between our results and the corresponding results available in related papers \cite{korbaproximal, luu2024nongeodesicallyconvex, zhu2025convergenceanalysiswassersteinproximal} has been presented in Section \ref{sec:RelatedWorks}.

\section{Two-layer mean-field neural networks}
\label{appendix:verify-assumptions}
In this section, we consider the example of a loss function for a two-layer mean-field neural network and showing that, under regularity conditions on the activation function, the loss function satisfies Assumptions \ref{ass:convexity}, \ref{assumption:L-derivative}, \ref{ass:lip-intrinsic}, \ref{assumption:unif-lsi}.

The pioneering works \cite{Chizat2018OnTG,Mei2018AMF,Rotskoff2018NeuralNA} showed that training a two-layer neural network in the mean-field regime can be viewed as minimizing a convex function over the space of probability distributions of the parameters of the network. This perspective has provided a framework for studying the convergence of various training dynamics for two-layer mean-field neural networks (see, e.g., \cite{10.1214/20-AIHP1140, chizat2022meanfield, Nitanda2022ConvexAO, suzuki2023meanfield,zhenjiefict}).

We now briefly describe the framework following \cite{10.1214/20-AIHP1140}. Let $C^2(\mathbb R)$ be the space of twice continuously differentiable functions on $\mathbb R$. Let $\nu$ be a finite measure with compact support represent the training data $(y,z) \in \mathbb R \times \mathbb R^{d-1},$ let $(w,b) \in \mathbb R^{d-1} \times \mathbb R$ be the parameters of the neural network and let $\varphi:\mathbb R \to \mathbb R$ be an activation function in $C^2(\mathbb R)$.

For $x \coloneqq (w,b) \in \mathbb R^d$ and $z \in \mathbb R^{d-1},$ define the function $\hat{\varphi}(x,z) \coloneqq \varphi(\langle w, z\rangle + b)$. The training of the two-layer neural network aims to find the optimal set of parameters $\{x_i\}_{i=1}^N$ which minimize the non-convex loss function
\begin{equation}
\label{eq:non-mean-field}
F_N(x_1,...,x_N) \coloneqq \frac{1}{2}\int_{\mathbb R^d} \left|y-\frac{1}{N}\sum_{i=1}^N \hat{\varphi}(x_i,z)\right|^2 \nu(\mathrm{d}y, \mathrm{d}z).
\end{equation}
Instead of the non-convex minimization problem \eqref{eq:non-mean-field}, we consider the mean-field optimization problem
\begin{equation*}
\min_{\mu \in \mathcal{P}_2(\mathbb R^d)} F(\mu), \quad \text{ with } F(\mu) \coloneqq \frac{1}{2}\int_{\mathbb R^d} \left|y-\mathbb E^{X \sim \mu}[\hat{\varphi}(X, z)]\right|^2 \nu(\mathrm{d}y, \mathrm{d}z).
\end{equation*}
Observe that by linearity of the expectation in $\mu$ and convexity of $|\cdot|^2,$ the function $F$ satisfies the flat-convexity condition \eqref{eq:convexity_classical}.

We stress that $F$ is not geodesically convex in the Wasserstein space. If it were, then this would require $F_N$ to be convex, which is clearly not the case (see also \cite[Remark 3.4]{chen2023uniformintimepropagationchaosmean} for more details). Now we show that $F$ satisfies Assumptions \ref{ass:convexity}, \ref{assumption:L-derivative}, \ref{ass:lip-intrinsic} and \ref{assumption:unif-lsi}.
\begin{proposition}
    Let $\nu$ be a finite measure with compact support. Let the activation function $\varphi \in C^2(\mathbb R)$ have bounded derivatives of first and second order. Then $F$ satisfies Assumptions \ref{ass:convexity}, \ref{assumption:L-derivative} and \ref{ass:lip-intrinsic}. Moreover, if $\rho$ satisfies LSI, then Assumption \ref{assumption:unif-lsi} also holds.
\end{proposition}
\begin{proof}
    Clearly, $F \geq 0.$ To show that $F \in \mathcal{C}^1,$ let $\mu^{\varepsilon} \coloneqq \mu + \varepsilon(\mu'-\mu),$ for $\varepsilon \in [0,1]$ and $\mu,\mu' \in \mathcal{P}_2(\mathbb R^d).$ Then, a straightforward calculation gives
\begin{align*}
    F(\mu^{\varepsilon}) &= F(\mu) - \varepsilon\int_{\mathbb R^d}\int_{\mathbb R^d} \left(y - \mathbb E^{X \sim \mu}[\hat{\varphi}(X, z)]\right)\hat{\varphi}(x, z)(\mu'-\mu)(\mathrm{d}x)\nu(\mathrm{d}y, \mathrm{d}z)\\
    &+\frac{\varepsilon^2}{2}\int_{\mathbb R^d}\left|\int_{\mathbb R^d}\hat{\varphi}(x, z)(\mu'-\mu)(\mathrm{d}x)\right|^2\nu(\mathrm{d}y, \mathrm{d}z).
\end{align*}
Rearranging and dividing by $\varepsilon$ yields
\begin{multline}
\label{eq:NN-is-C^1}
\begin{aligned}
    \frac{F(\mu^{\varepsilon})-F(\mu)}{\varepsilon} &= -\int_{\mathbb R^d}\int_{\mathbb R^d} \left(y - \mathbb E^{X \sim \mu}[\hat{\varphi}(X, z)]\right)\hat{\varphi}(x, z)(\mu'-\mu)(\mathrm{d}x)\nu(\mathrm{d}y, \mathrm{d}z)\\
    &+\frac{\varepsilon}{2}\int_{\mathbb R^d}\left|\int_{\mathbb R^d}\hat{\varphi}(x, z)(\mu'-\mu)(\mathrm{d}x)\right|^2\nu(\mathrm{d}y, \mathrm{d}z).
\end{aligned}
\end{multline}
Observe that
\begin{equation*}
    \nabla_x \hat{\varphi}(x,z) = \left(\varphi'(\langle w,z \rangle + b), \varphi'(\langle w,z\rangle + b) z\right).
\end{equation*}
Because $|x'-x| = \sqrt{|w'-w|^2 +|b'-b|^2},$ we have $|x'-x| \geq |w'-w|$ and $|x'-x| \geq |b'-b|.$ Since $\varphi'$ is bounded, $\hat{\varphi}(\cdot,z)$ is Lipschitz for every $z.$ More precisely, for all $z \in \mathbb R^{d-1},$
\begin{equation*}
    \left|\hat{\varphi}(x',z) - \hat{\varphi}(x,z)\right| \leq \|\varphi'\|_\infty(1+|z|)|x'-x|.
\end{equation*}
Also, $\varphi(0,z) = \varphi(0),$ for $z \in \mathbb R^{d-1},$ hence
\begin{equation*}
    \left|\hat{\varphi}(x,z)\right| \leq |\varphi(0)| + \|\varphi'\|_\infty(1+|z|)|x|.
\end{equation*}
Because $\mu \in \mathcal{P}_2(\mathbb R^d),$ the first moment $M_1 \coloneqq \int_{\mathbb R^d}|x|\mu(\mathrm{d}x)$ is finite. Using the bound on $\left|\hat{\varphi}(x,z)\right|,$ we obtain
\begin{equation*}
    \mathbb E^{X \sim \mu}[|\hat{\varphi}(X, z)|] \leq C_\varphi(1+|z|),
\end{equation*}
where $C_\varphi \coloneqq |\varphi(0)| + M_1\|\varphi'\|_\infty > 0,$ and therefore
\begin{align*}
    &\left|\int_{\mathbb R^d}\int_{\mathbb R^d} \left(y - \mathbb E^{X \sim \mu}[\hat{\varphi}(X, z)]\right)\hat{\varphi}(x, z)\mu(\mathrm{d}x)\nu(\mathrm{d}y, \mathrm{d}z)\right|\\
    &\leq \int_{\mathbb R^d}\int_{\mathbb R^d} \left(|y| + \mathbb E^{X \sim \mu}[|\hat{\varphi}(X, z)|]\right)|\hat{\varphi}(x, z)|\mu(\mathrm{d}x)\nu(\mathrm{d}y, \mathrm{d}z)\\
    &=\int_{\mathbb R^d}|y|\mathbb E^{X \sim \mu}[|\hat{\varphi}(X, z)|]\nu(\mathrm{d}y, \mathrm{d}z)+\int_{\mathbb R^d}\left(\mathbb E^{X \sim \mu}[|\hat{\varphi}(X, z)|]\right)^2\nu(\mathrm{d}y, \mathrm{d}z)\\
    &\leq C_\varphi\int_{\mathbb R^d}|y|(1+|z|)\nu(\mathrm{d}y, \mathrm{d}z) +C_\varphi^2\int_{\mathbb R^d} (1+|z|)^2\nu(\mathrm{d}y, \mathrm{d}z)\\
    &\leq C_\varphi \nu(\mathbb R^d) \sup_{(y,z) \in \text{supp}(\nu)}|y|(1+|z|) +C_\varphi^2 \nu(\mathbb R^d) \sup_{(y,z) \in \text{supp}(\nu)}(1+|z|)^2 < \infty.
\end{align*}
From this finiteness one also gets the analogous bound when $\mu$ is replaced by $\mu'.$ Therefore,
\begin{align*}
    &\int_{\mathbb R^d}\int_{\mathbb R^d} \left|\left(y - \mathbb E^{X \sim \mu}[\hat{\varphi}(X, z)]\right)\hat{\varphi}(x, z)\right||\mu'-\mu|(\mathrm{d}x)\nu(\mathrm{d}y, \mathrm{d}z)\\
    &\leq \int_{\mathbb R^d}\int_{\mathbb R^d} \left|\left(y - \mathbb E^{X \sim \mu}[\hat{\varphi}(X, z)]\right)\hat{\varphi}(x, z)\right|\mu'(\mathrm{d}x)\nu(\mathrm{d}y, \mathrm{d}z)\\
    &+\int_{\mathbb R^d}\int_{\mathbb R^d} \left|\left(y - \mathbb E^{X \sim \mu}[\hat{\varphi}(X, z)]\right)\hat{\varphi}(x, z)\right|\mu(\mathrm{d}x)\nu(\mathrm{d}y, \mathrm{d}z) < \infty.
\end{align*}
Thus, by Fubini’s theorem, we may interchange the order of integration and write
\begin{align*}
    &-\int_{\mathbb R^d}\int_{\mathbb R^d} \left(y - \mathbb E^{X \sim \mu}[\hat{\varphi}(X, z)]\right)\hat{\varphi}(x, z)(\mu'-\mu)(\mathrm{d}x)\nu(\mathrm{d}y, \mathrm{d}z)\\
    &=-\int_{\mathbb R^d}\int_{\mathbb R^d} \left(y - \mathbb E^{X \sim \mu}[\hat{\varphi}(X, z)]\right)\hat{\varphi}(x, z)\nu(\mathrm{d}y, \mathrm{d}z)(\mu'-\mu)(\mathrm{d}x).
\end{align*}
Passing to the limit in \eqref{eq:NN-is-C^1}, as $\varepsilon \searrow 0,$ therefore yields
\begin{equation*}
    \lim_{\varepsilon \searrow 0 }\frac{F(\mu^\varepsilon)- F(\mu)}{\varepsilon} = -\int_{\mathbb R^d}\int_{\mathbb R^d} \left(y - \mathbb E^{X \sim \mu}[\hat{\varphi}(X, z)]\right)\hat{\varphi}(x, z)\nu(\mathrm{d}y, \mathrm{d}z)(\mu'-\mu)(\mathrm{d}x),
\end{equation*}
so the flat derivative is
\begin{equation*}
    \frac{\delta F}{\delta \mu}(\mu, x) = -\int_{\mathbb R^d} \left(y - \mathbb E^{X \sim \mu}[\hat{\varphi}(X, z)]\right)\hat{\varphi}(x, z)\nu(\mathrm{d}y, \mathrm{d}z).
\end{equation*}
Since $F \in \mathcal{C}^1$ and it satisfies the flat-convexity condition \eqref{eq:convexity_classical}, it satisfies Assumption \ref{ass:convexity}. 

For Assumption \ref{assumption:L-derivative}, using the Euclidean norm $|(u,v)| = \sqrt{|u|^2 + |v|^2},$ we also get
\begin{equation*}
    |\nabla_x\hat{\varphi}(x, z)| \leq \sqrt{1+|z|^2}\|\varphi'\|_\infty \leq (1+|z|)\|\varphi'\|_\infty,
\end{equation*}
where we used the standard inequality $\sqrt{1+a^2} \leq 1+a$ for all $a \geq 0.$ From this bound one sees that $F$ is Wasserstein differentiable with
\begin{equation*}
\nabla_\mu F(\mu)(x) = -\int_{\mathbb R^d} \left(y - \mathbb E^{X \sim \mu}[\hat{\varphi}(X, z)]\right)\nabla_x\hat{\varphi}(x, z)\nu(\mathrm{d}y, \mathrm{d}z).
\end{equation*}
To verify Assumption \ref{ass:lip-intrinsic}, note that $\varphi''$ bounded implies
\begin{align*}
    \left|\nabla_x\hat{\varphi}(x',z) - \nabla_x\hat{\varphi}(x,z)\right| &\leq \sqrt{1+|z|^2}\left|\varphi'(\langle w',z \rangle + b')-\varphi'(\langle w,z \rangle + b)\right|\\
    &\leq \|\varphi''\|_{\infty}\sqrt{1+|z|^2}\left(|w'-w||z|+|b'-b|\right)\\
    &\leq \|\varphi''\|_{\infty}\sqrt{1+|z|^2}\left(1+|z|\right)|x'-x|\\
    &\leq \|\varphi''\|_{\infty}\left(1+|z|\right)^2|x'-x|.
\end{align*}
Now write
\begin{equation}
\begin{aligned}
\label{eq:verify-lip}
    &\left|\nabla_\mu F(\mu')(x')-\nabla_\mu F(\mu)(x)\right|\\ 
    &\leq \left|\nabla_\mu F(\mu')(x')- \nabla_\mu F(\mu)(x')\right| + \left|\nabla_\mu F(\mu)(x') + \nabla_\mu F(\mu)(x)\right|\\
    &\leq \int_{\mathbb R^d}\left|\int_{\mathbb R^d}\hat{\varphi}(x,z)(\mu'-\mu)(\mathrm{d}x)\right||\nabla_x \hat{\varphi}(x',z)|\nu(\mathrm{d}y, \mathrm{d}z)\\
    &+\int_{\mathbb R^d}\left|y - \mathbb E^{X \sim \mu}[\hat{\varphi}(X, z)]\right|\left|\nabla_x\hat{\varphi}(x',z) - \nabla_x\hat{\varphi}(x,z)\right| \nu(\mathrm{d}y, \mathrm{d}z)\\
    &\leq \|\varphi'\|_\infty \int_{\mathbb R^d}\left|\int_{\mathbb R^d}\hat{\varphi}(x,z)(\mu'-\mu)(\mathrm{d}x)\right|(1+|z|)\nu(\mathrm{d}y, \mathrm{d}z)\\
    &+\|\varphi''\|_{\infty}|x'-x|\int_{\mathbb R^d}\left|y - \mathbb E^{X \sim \mu}[\hat{\varphi}(X, z)]\right|\left(1+|z|\right)^2 \nu(\mathrm{d}y, \mathrm{d}z).
\end{aligned}
\end{equation}
On the one hand, we have
\begin{equation}
\begin{aligned}
\label{eq:rubinstein}
    \left|\int_{\mathbb R^d}\hat{\varphi}(x,z)(\mu'-\mu)(\mathrm{d}x)\right| &\leq \|\varphi'\|_\infty (1+|z|) \mathcal{W}_1(\mu',\mu)\\ 
    &\leq \|\varphi'\|_\infty (1+|z|) \mathcal{W}_2(\mu',\mu),
\end{aligned}
\end{equation}
by the Kantorovich–Rubinstein duality for the $L^1$-Wasserstein distance (cf. \cite[Remark 6.5]{villani2008optimal}) since $\hat{\varphi}(\cdot,z)$ is Lipschitz continuous for all $z.$

On the other hand, we have
\begin{equation}
\begin{aligned}
\label{eq:marginal}
    &\int_{\mathbb R^d}\left|y - \mathbb E^{X \sim \mu}[\hat{\varphi}(X, z)]\right|\left(1+|z|\right)^2\nu(\mathrm{d}y, \mathrm{d}z)\\ 
    &\leq \int_{\mathbb R^d}\left(|y| + C_\varphi(1+|z|)\right)\left(1+|z|\right)^2\nu(\mathrm{d}y, \mathrm{d}z)\\
    &\leq \nu(\mathbb R^d)\sup_{(y,z) \in \text{supp}(\nu)}|y|\left(1+|z|\right)^2 + C_\varphi \nu(\mathbb R^d) \sup_{(y,z) \in \text{supp}(\nu)}\left(1+|z|\right)^3.
\end{aligned}
\end{equation}
Combining estimates \eqref{eq:rubinstein} and \eqref{eq:marginal} with \eqref{eq:verify-lip} gives
\begin{align*}
    &\left|\nabla_\mu F(\mu')(x')-\nabla_\mu F(\mu)(x)\right|\\ 
    &\leq \|\varphi'\|_\infty^2 \nu(\mathbb R^d)\mathcal{W}_2(\mu',\mu) \sup_{(y,z) \in \text{supp}(\nu)}(1+|z|)^2\\
    &+\|\varphi''\|_{\infty}|x'-x|\Bigg(\sup_{(y,z) \in \text{supp}(\nu)}|y|\left(1+|z|\right)^2 + C_\varphi\sup_{(y,z) \in \text{supp}(\nu)}\left(1+|z|\right)^3\Bigg)\nu(\mathbb R^d)\\
    &=L_F'\left(|x'-x| + \mathcal{W}_2(\mu',\mu)\right),
\end{align*}
where one can take $L_F'$ to be $\nu(\mathbb R^d)$ multiplied by the maximum of the two explicit constants 
\begin{align*}
   &\|\varphi'\|_\infty^2 \sup_{(y,z) \in \text{supp}(\nu)}(1+|z|)^2 \quad \text{and}\\ 
   &\|\varphi''\|_{\infty}\Bigg(\sup_{(y,z) \in \text{supp}(\nu)}|y|\left(1+|z|\right)^2 + C_\varphi\sup_{(y,z) \in \text{supp}(\nu)}\left(1+|z|\right)^3\Bigg).
\end{align*}
If $\rho$ satisfies LSI, then Assumption \ref{assumption:unif-lsi} holds by virtue of the second criterion stated in Remark \ref{remark:sufficient-criteria-lsi}. Indeed, by an argument analogous to the one used in verifying Assumption \ref{ass:lip-intrinsic}, we obtain
\begin{align*}
    &\left|\frac{\delta F}{\delta \mu}(\mu, x') - \frac{\delta F}{\delta \mu}(\mu, x)\right|\\ 
    &\leq \int_{\mathbb R^d}\left|y - \mathbb E^{X \sim \mu}[\hat{\varphi}(X, z)]\right|\left|\hat{\varphi}(x',z) - \hat{\varphi}(x,z)\right| \nu(\mathrm{d}y, \mathrm{d}z)\\
    &\leq \|\varphi'\|_\infty|x'-x|\int_{\mathbb R^d}\left(|y| + \mathbb E^{X \sim \mu}[|\hat{\varphi}(X, z)|]\right) (1+|z|)\nu(\mathrm{d}y, \mathrm{d}z)\\
    &\leq \|\varphi'\|_\infty|x'-x|\int_{\mathbb R^d}\left(|y| + C_\varphi (1+|z|)\right) (1+|z|)\nu(\mathrm{d}y, \mathrm{d}z)\\
    &\leq  \|\varphi'\|_\infty\nu(\mathbb R^d)\Bigg(\sup_{(y,z) \in \text{supp}(\nu)}|y|\left(1+|z|\right) + C_\varphi \sup_{(y,z) \in \text{supp}(\nu)}\left(1+|z|\right)^2\Bigg)|x'-x|,
\end{align*}
which shows that $\frac{\delta F}{\delta \mu}(\mu, \cdot)$ is Lipschitz for all $\mu$ with a Lipschitz constant uniform in $\mu$.
\end{proof}
Although the previous proposition, together with our main theorems, provides new convergence guarantees for schemes involving two-layer mean-field neural networks, it is important to emphasize the limitations of these results. 

In \cite[Section 3.1]{10.1214/20-AIHP1140}, the authors consider the neural network $\hat{\varphi}(x,z) \coloneqq \ell(b)\varphi(\langle w, z\rangle),$ where $\varphi \in C^2(\mathbb R)$ is bounded with bounded first and second derivatives and $\ell: \mathbb R \to [-K,K]$ is a clipping function with threshold $K > 0.$ In this case,
\begin{equation*}
\nabla_x \hat{\varphi}(x,z) = \left(\ell'(b)\varphi(\langle w,z \rangle), \varphi'(\langle w,z\rangle)z\right).
\end{equation*}
Verifying Assumption \ref{ass:lip-intrinsic} therefore reduces to showing that $\nabla_x \hat{\varphi}(\cdot,z)$ is Lipschitz, which in turn requires $\varphi$ to be bounded, thereby excluding unbounded activation functions. An analogous restriction appears in \cite[Proposition 5.1]{chizat2022meanfield}, where $\hat \varphi$ is assumed to be bounded, as well as smooth in $x,$ uniformly in $z$. Hence the setting of \cite{chizat2022meanfield} also excludes unbounded activations $\varphi$.

In contrast, we remove the boundedness requirement on $\varphi$ by adopting the network architecture $\hat{\varphi}(x,z) \coloneqq \varphi(\langle w,z \rangle + b)$, which is commonly used in practice; see, for instance, \cite[Chapter 6]{Goodfellow-et-al-2016}, \cite[Section 6.2]{Bishop:DeepLearning24}. We therefore assume that $\varphi \in C^2(\mathbb R)$ with bounded first and second derivatives, but without requiring $\varphi$ itself to be bounded.

As a consequence, in addition to bounded activations such as sigmoid and hyperbolic tangent (tanh), which are already covered by \cite[Section 3.1]{10.1214/20-AIHP1140} and \cite[Proposition 5.1]{chizat2022meanfield}, our setting also includes unbounded activations including GELU, SiLU and Softplus. The Softplus function, in particular, provides a smooth approximation of the ReLU activation. While ReLU is not differentiable at $0$, Softplus belongs to $C^\infty(\mathbb R)$. Extending the analysis to nonsmooth activations such as ReLU remains an open problem, as also noted in \cite{chizat2022meanfield}.

\section{Proof of Theorem \ref{thm:monot-decrease}}\label{sec:monot-proof}
Before we present the proof of Theorem \ref{thm:monot-decrease}, for convenience, we give a short calculation that we will repeatedly use in this proof and in the proof of Theorem \ref{thm:exp-conv}. For $F$ satisfying Assumption \ref{assumption:L-derivative}, $\rho \propto e^{-U},$ any $\mu \in \mathcal{P}_2(\mathbb R^d)$ and any $\mu' \in \mathcal{P}_2^{\lambda}(\mathbb R^d),$ it holds
\begin{equation}
\begin{aligned}
\label{eq:from-F-to-Gibbs}
\nabla_\mu F(\mu) + \sigma \nabla \log \frac{\mathrm{d}\mu'}{\mathrm{d}\rho} &= \nabla \frac{\delta F}{\delta \mu}(\mu, \cdot) + \sigma \nabla \log \frac{\mathrm{d}\mu'}{\mathrm{d}\rho}\\ 
& = -\sigma \nabla \log e^{-\frac{1}{\sigma}\frac{\delta F}{\delta \mu}(\mu,\cdot)} + \sigma \nabla \log \frac{\mathrm{d}\mu'}{\mathrm{d}\rho} 
= \sigma \nabla \log \frac{\mathrm{d}\mu'}{\mathrm{d}\Phi_{\sigma}[\mu]}\,,
\end{aligned}
\end{equation}
where the last equality follows from \eqref{eq:proximal-gibbs}.
\begin{proof}[Proof of Theorem \ref{thm:monot-decrease} 1.]
First, note that by Theorem \ref{thm:well-posedness} 1., the iterates $\left(\mu^n\right)_{n \in \mathbb N} \subset \mathcal{P}^\rho_2(\mathbb R^d).$ Since $\mu^{n+1} \in \mathcal{P}^\rho_2(\mathbb R^d)$ is a minimizer of \eqref{eq:implicit-JKO}, it follows that
\begin{equation}
\label{eq:mun+1}
F^{\sigma}(\mu^{n+1}) + \frac{1}{2\tau}\mathcal{W}_2^2(\mu^{n+1}, \mu^n) \leq F^{\sigma}(\mu^n) + \frac{1}{2\tau}\mathcal{W}_2^2(\mu^n, \mu^n) = F^{\sigma}(\mu^n).
\end{equation}
Because $\frac{1}{2\tau}\mathcal{W}_2^2(\mu^{n+1}, \mu^n) \geq 0,$ we obtain the conclusion.
\end{proof}
\begin{proof}[Proof of Theorem \ref{thm:monot-decrease} 2.]
    First, note that by Theorem \ref{thm:optim} 2., the iterates $\left(\mu^n\right)_{n \in \mathbb N} \subset \mathfrak C.$ Combining Lemma \ref{ass:relative-smoothness} with Lemma \ref{lemma:control-geod-KL} gives
\begin{multline}
\label{eq:smoothness-F-sigma}
\begin{aligned}
F^{\sigma}(\mu^{n+1}) - F^{\sigma}(\mu^{n}) - \left\langle \nabla_\mu F(\mu^n) + \sigma \nabla \log \frac{\mathrm{d}\mu^{n+1}}{\mathrm{d}\rho}\left(T_{\mu^n}^{\mu^{n+1}}\right), T_{\mu^n}^{\mu^{n+1}} - I_d \right\rangle_{L_{\mu^n}^2(\mathbb R^d)}\\ \leq L_F'\mathcal{W}_2^2(\mu^{n+1}, \mu^n).
\end{aligned}
\end{multline}
Using \eqref{eq:from-F-to-Gibbs} with $\mu = \mu^n$ and $\mu' = \mu^{n+1}$, observe that for $\mu^{n}\text{-a.e. } x,$ we have
\begin{multline}
\label{eq:identity-gradient}
\begin{aligned}
\nabla_\mu F(\mu^n)(x) &+ \sigma \nabla \log \frac{\mathrm{d}\mu^{n+1}}{\mathrm{d}\rho}\left(T_{\mu^n}^{\mu^{n+1}}(x)\right)\\ &= \nabla_\mu F(\mu^n)(x) - \nabla_\mu F(\mu^n)\left(T_{\mu^n}^{\mu^{n+1}}(x)\right)\\ 
& \qquad + \nabla_\mu F(\mu^n)\left(T_{\mu^n}^{\mu^{n+1}}(x)\right) + \sigma \nabla \log \frac{\mathrm{d}\mu^{n+1}}{\mathrm{d}\rho}\left(T_{\mu^n}^{\mu^{n+1}}(x)\right)\\
&= \nabla_\mu F(\mu^n)(x) - \nabla_\mu F(\mu^n)\left(T_{\mu^n}^{\mu^{n+1}}(x)\right) + \sigma \nabla \log \frac{\mathrm{d}\mu^{n+1}}{\mathrm{d}\Phi_{\sigma}[\mu^n]}\left(T_{\mu^n}^{\mu^{n+1}}(x)\right).
\end{aligned}
\end{multline}
Hence, using \eqref{eq:identity-gradient} in \eqref{eq:smoothness-F-sigma} gives
\begin{multline}
\label{eq:F-sigma-estimate}
\begin{aligned}
&F^{\sigma}(\mu^{n+1}) - F^{\sigma}(\mu^{n}) - \left\langle \sigma \nabla \log \frac{\mathrm{d}\mu^{n+1}}{\mathrm{d}\Phi_{\sigma}[\mu^n]}\left(T_{\mu^n}^{\mu^{n+1}}\right), T_{\mu^n}^{\mu^{n+1}} - I_d \right\rangle_{L_{\mu^n}^2(\mathbb R^d)}\\ &\leq L_F'\mathcal{W}_2^2(\mu^{n+1}, \mu^n) + \left\langle \nabla_\mu F(\mu^n) - \nabla_\mu F(\mu^n)\left(T_{\mu^n}^{\mu^{n+1}}\right), T_{\mu^n}^{\mu^{n+1}} - I_d \right\rangle_{L_{\mu^n}^2(\mathbb R^d)}\\
&\leq L_F'\mathcal{W}_2^2(\mu^{n+1}, \mu^n) + \left\|\nabla_\mu F(\mu^n) - \nabla_\mu F(\mu^n)\left(T_{\mu^n}^{\mu^{n+1}}\right)\right\|_{L_{\mu^n}^2(\mathbb R^d)}\left\|T_{\mu^n}^{\mu^{n+1}} - I_d\right\|_{L_{\mu^n}^2(\mathbb R^d)}\\
&\leq L_F'\mathcal{W}_2^2(\mu^{n+1}, \mu^n) + L_F'\left\|T_{\mu^n}^{\mu^{n+1}} - I_d\right\|^2_{L_{\mu^n}^2(\mathbb R^d)} =  2L_F'\mathcal{W}_2^2(\mu^{n+1}, \mu^n),
\end{aligned}
\end{multline}
where the second, third and last inequality follows from the Cauchy-Schwarz inequality, Assumption \ref{ass:lip-intrinsic} and Corollary \ref{corollary:wass-transport}, respectively.

By Theorem \ref{thm:optim} 2., the optimality condition for \eqref{eq:semi-implicit-JKO} reads
\begin{equation*}
\nabla_\mu F(\mu^n)(x) + \sigma \nabla \log \frac{\mathrm{d}\mu^{n+1}}{\mathrm{d}\rho}(x) = \frac{1}{\tau}\left(T_{\mu^{n+1}}^{\mu^n}(x) - x\right), \quad \text{for } \mu^{n+1}\text{-a.e. } x.
\end{equation*}
By \eqref{eq:from-F-to-Gibbs} with $\mu = \mu^n$ and $\mu' = \mu^{n+1}$, we obtain
\begin{equation}
\label{eq:foc-prox-gibbs}
\nabla \log \frac{\mathrm{d}\mu^{n+1}}{\mathrm{d}\Phi_{\sigma}[\mu^n]}(x) = \frac{1}{\tau\sigma}\left(T_{\mu^{n+1}}^{\mu^n}(x) - x\right), \quad \text{for } \mu^{n+1}\text{-a.e. } x.
\end{equation}
Hence, using the fact that $T_{\mu^n}^{\mu^{n+1}} \circ T_{\mu^{n+1}}^{\mu^n} = I_d,$ $\mu^{n+1}$-a.e., it follows that
\begin{equation}
\label{eq:foc-pushforward}
\nabla \log \frac{\mathrm{d}\mu^{n+1}}{\mathrm{d}\Phi_{\sigma}[\mu^n]}\left(T_{\mu^n}^{\mu^{n+1}}(x)\right) = \frac{1}{\tau\sigma}\left(x - T_{\mu^n}^{\mu^{n+1}}(x)\right), \quad \text{for } \mu^{n}\text{-a.e. } x.
\end{equation}
Multiplying \eqref{eq:foc-pushforward} by $T_{\mu^n}^{\mu^{n+1}}-I_d$ and integrating over $\mu^n$ gives
\begin{multline}
\label{eq:fisher-estimate}
\begin{aligned}
\left\langle \sigma \nabla \log \frac{\mathrm{d}\mu^{n+1}}{\mathrm{d}\Phi_{\sigma}[\mu^n]}\left(T_{\mu^n}^{\mu^{n+1}}\right), T_{\mu^n}^{\mu^{n+1}} - I_d \right\rangle_{L_{\mu^n}^2(\mathbb R^d)} &= -\tau \sigma^2 \left\|\nabla \log \frac{\mathrm{d}\mu^{n+1}}{\mathrm{d}\Phi_{\sigma}[\mu^n]}\left(T_{\mu^n}^{\mu^{n+1}}\right)\right\|^2_{L_{\mu^n}^2(\mathbb R^d)}\\ &= -\tau \sigma^2 I\left(\mu^{n+1}\big|\Phi_{\sigma}[\mu^n]\right).
\end{aligned}
\end{multline}
Also, squaring both sides of \eqref{eq:foc-prox-gibbs} and integrating with respect to $\mu^{n+1}$ gives
\begin{equation}
\label{eq:fisher=wass}
I\left(\mu^{n+1}\big|\Phi_{\sigma}[\mu^n]\right) = \frac{1}{\tau^2 \sigma^2}\mathcal{W}_2^2(\mu^{n+1}, \mu^n).
\end{equation}
Using \eqref{eq:fisher-estimate} and \eqref{eq:fisher=wass} in \eqref{eq:F-sigma-estimate} gives
\begin{multline}
\label{eq:dissipation-proximal}
\begin{aligned}
F^{\sigma}(\mu^{n+1}) &\leq F^{\sigma}(\mu^{n}) -\tau \sigma^2 I\left(\mu^{n+1}\big|\Phi_{\sigma}[\mu^n]\right) + 2 L_F'\mathcal{W}_2^2(\mu^{n+1}, \mu^n)\\ &= F^{\sigma}(\mu^{n}) -\tau \sigma^2 I\left(\mu^{n+1}\big|\Phi_{\sigma}[\mu^n]\right) + 2\tau^2 \sigma^2L_F' I\left(\mu^{n+1}\big|\Phi_{\sigma}[\mu^n]\right)\\
&= F^{\sigma}(\mu^{n}) - \tau \sigma^2 \left(1 - 2\tau L_F'\right)I\left(\mu^{n+1}\big|\Phi_{\sigma}[\mu^n]\right)\\
&\leq F^{\sigma}(\mu^{n}),
\end{aligned}
\end{multline}
where the last inequality follows from $\tau < \frac{1}{2L_F'}.$
\end{proof}
\begin{proof}[Proof of Theorem \ref{thm:monot-decrease} 3.]
First, note that by Theorem \ref{thm:optim} 3., the iterates $\left(\mu^n\right)_{n \in \mathbb N} \subset \mathfrak C,$ and by Lemma \ref{lemma:pushforward-is-optimal} that $(\nu^n)_{n \in \mathbb N} \subset \mathcal{P}_2^{\rho}(\mathbb R^d).$ By Theorem \ref{thm:optim} 3., we have
\begin{equation*}
\sigma \nabla \log \frac{\mathrm{d}\mu^{n+1}}{\mathrm{d}\rho}(x) = \frac{1}{\tau}\left(T_{\mu^{n+1}}^{\nu^{n+1}}(x) - x\right), \quad \text{for } \mu^{n+1}\text{-a.e. } x.
\end{equation*}
By Corollary \ref{corollary:existence-optimal-transport-proximal}, exists a unique $\nu^{n+1}$-a.e. optimal transport map $T_{\nu^{n+1}}^{\mu^{n+1}}:\mathbb R^d \to \mathbb R^d$ such that $T_{\mu^{n+1}}^{\nu^{n+1}} \circ T_{\nu^{n+1}}^{\mu^{n+1}} = I_d,$ $\nu^{n+1}$-a.e.. We thus have
\begin{equation}
\label{eq:opt-H}
T_{\nu^{n+1}}^{\mu^{n+1}} = I_d - \tau\sigma\nabla \log \frac{\mathrm{d}\mu^{n+1}}{\mathrm{d}\rho}\left(T_{\nu^{n+1}}^{\mu^{n+1}}\right), \quad \nu^{n+1}\text{-a.e.}.
\end{equation}
By Assumption \ref{ass:abs-cty-pi}, it follows that $\operatorname{KL}(\cdot|\rho)$ is convex along generalized geodesics in the Wasserstein space \cite[Theorem 9.4.11]{ambrosio2008gradient}, thus taking $\mu=\mu^{n+1}, \rho=\mu^n$ and $\nu=\nu^{n+1}$ in \cite[Lemma 4]{korbaproximal} gives
\begin{equation*}
%\label{eq:H-cvx-descent}
\operatorname{KL}(\mu^{n+1}|\rho) \leq \operatorname{KL}(\mu^{n}|\rho) - \left\langle\nabla \log \frac{\mathrm{d}\mu^{n+1}}{\mathrm{d}\rho}\left(T_{\nu^{n+1}}^{\mu^{n+1}}\right), T_{\nu_{n+1}}^{\mu_{n}} - T_{\nu_{n+1}}^{\mu_{n+1}}\right\rangle_{L_{\nu^{n+1}}^2(\mathbb R^d)}.
\end{equation*}
By Corollary \ref{corollary:existence-optimal-transport-proximal}, we have $T_{\nu^{n+1}}^{\mu^n} = \left(I_d-\tau\nabla_\mu F(\mu^n)\right)^{-1},$ $\nu^{n+1}\text{-a.e.}.$ Therefore, 
\begin{equation*}
\operatorname{KL}(\mu^{n+1}|\rho) \leq \operatorname{KL}(\mu^{n}|\rho) - \left\langle\nabla \log \frac{\mathrm{d}\mu^{n+1}}{\mathrm{d}\rho}\left(T_{\nu^{n+1}}^{\mu^{n+1}}\right), \left(I_d-\tau\nabla_\mu F(\mu^n)\right)^{-1} -T_{\nu_{n+1}}^{\mu_{n+1}}\right\rangle_{L_{\nu^{n+1}}^2(\mathbb R^d)}.
\end{equation*}
Let us define the map $P_{\mu^n}^{\mu^{n+1}} \coloneqq T_{\nu^{n+1}}^{\mu^{n+1}} \circ T_{\mu^n}^{\nu^{n+1}}.$ Note that ${P_{\mu^n}^{\mu^{n+1}}}_{\#} \mu^n = \mu^{n+1}$, so then $T_{\nu^{n+1}}^{\mu^{n+1}} = P_{\mu^n}^{\mu^{n+1}} \circ \left(I_d-\tau\nabla_\mu F(\mu^n)\right)^{-1},$ and hence by \eqref{eq:pushforward}, the last inequality is equivalent to
\begin{equation}
\label{eq:descent-H}
\operatorname{KL}(\mu^{n+1}|\rho) \leq \operatorname{KL}(\mu^{n}|\rho) - \left\langle\nabla \log \frac{\mathrm{d}\mu^{n+1}}{\mathrm{d}\rho}\left(P_{\mu^n}^{\mu^{n+1}}\right), I_d - P_{\mu^n}^{\mu^{n+1}}\right\rangle_{L_{\mu^{n}}^2(\mathbb R^d)}.
\end{equation}
Using \eqref{eq:opt-H} with $T_{\nu^{n+1}}^{\mu^{n+1}} = P_{\mu^n}^{\mu^{n+1}} \circ  \left(I_d-\tau\nabla_\mu F(\mu^n)\right)^{-1}$ and then composing with $I_d-\tau\nabla_\mu F(\mu^n)$ yields
\begin{equation}
\label{eq:explicitimplicit}
P_{\mu^n}^{\mu^{n+1}} = I_d-\tau\nabla_\mu F(\mu^n) - \tau\sigma\nabla \log \frac{\mathrm{d}\mu^{n+1}}{\mathrm{d}\rho}\left(P_{\mu^n}^{\mu^{n+1}}\right), \quad \mu^{n}\text{-a.e.}
\end{equation}
Multiplying \eqref{eq:descent-H} by $\sigma$ and using Lemma \ref{ass:relative-smoothness} gives
\begin{multline}
\label{eq:descent}
\begin{aligned}
F^{\sigma}(\mu^{n+1}) &\leq F^{\sigma}(\mu^n) + \left\langle \nabla_\mu F(\mu^n) + \sigma\nabla \log \frac{\mathrm{d}\mu^{n+1}}{\mathrm{d}\rho}\left(P_{\mu^n}^{\mu^{n+1}}\right), P_{\mu^n}^{\mu^{n+1}} - I_d\right\rangle_{L_{\mu^n}^2(\mathbb R^d)}\\ 
&+ L_F'\left\|I_d-P_{\mu^n}^{\mu^{n+1}}\right\|^2_{L_{\mu^n}^2(\mathbb R^d)}\\
&= F^{\sigma}(\mu^n) -\frac{1}{\tau}\left\|I_d-P_{\mu^n}^{\mu^{n+1}}\right\|^2_{L_{\mu^n}^2(\mathbb R^d)} +L_F'\left\|I_d-P_{\mu^n}^{\mu^{n+1}}\right\|^2_{L_{\mu^n}^2(\mathbb R^d)}\\
&=F^{\sigma}(\mu^n) -\left(\frac{1}{\tau} - L_F'\right)\left\|I_d-P_{\mu^n}^{\mu^{n+1}}\right\|^2_{L_{\mu^n}^2(\mathbb R^d)}\\
&\leq F^{\sigma}(\mu^n),
\end{aligned}
\end{multline}
where the first equality follows from \eqref{eq:explicitimplicit} and the last inequality follows from $\tau < \frac{1}{L_F'}.$ 
\end{proof}

\section{Proof of Theorem \ref{thm:exp-conv}}\label{sec:convergence-proof}
\begin{proof}[Proof of Theorem \ref{thm:exp-conv} 1.]
First, note that by Theorem \ref{thm:optim} 1., the iterates $\left(\mu^n\right)_{n \in \mathbb N} \subset \mathfrak C.$ Since $\mu^{n+1} \in \mathfrak C.$ Then, by Theorem \ref{thm:optim} 1., the optimality condition for \eqref{eq:implicit-JKO} reads
\begin{equation*}
\nabla_\mu F(\mu^{n+1})(x) + \sigma \nabla \log \frac{\mathrm{d}\mu^{n+1}}{\mathrm{d}\rho}(x) = \frac{1}{\tau}\left(T_{\mu^{n+1}}^{\mu^n}(x) - x\right), \quad \text{for } \mu^{n+1}\text{-a.e. } x.
\end{equation*}
By \eqref{eq:from-F-to-Gibbs} with $\mu' = \mu = \mu^{n+1}$, we obtain
\begin{equation}
\label{eq:foc-prox-gibbs-implicit}
\nabla \log \frac{\mathrm{d}\mu^{n+1}}{\mathrm{d}\Phi_{\sigma}[\mu^{n+1}]}(x) = \frac{1}{\tau\sigma}\left(T_{\mu^{n+1}}^{\mu^n}(x) - x\right), \quad \text{for } \mu^{n+1}\text{-a.e. } x.
\end{equation}
Squaring both sides of \eqref{eq:foc-prox-gibbs-implicit} and integrating with respect to $\mu^{n+1}$ gives
\begin{equation}
\label{eq:fisher=wass-implicit}
I\left(\mu^{n+1}\big|\Phi_{\sigma}[\mu^{n+1}]\right) = \frac{1}{\tau^2 \sigma^2}\mathcal{W}_2^2(\mu^{n+1}, \mu^n).
\end{equation}
Recalling that \eqref{eq:mun+1} is
\begin{equation*}
F^{\sigma}(\mu^{n+1}) \leq F^{\sigma}(\mu^n) - \frac{1}{2\tau}\mathcal{W}_2^2(\mu^{n+1}, \mu^n),
\end{equation*}
and using \eqref{eq:fisher=wass-implicit} gives
\begin{equation*}
\begin{aligned}
& F^{\sigma}(\mu^{n+1}) 
 \leq F^{\sigma}(\mu^n) - \frac{1}{2\tau}\mathcal{W}_2^2(\mu^{n+1}, \mu^n) 
= F^{\sigma}(\mu^n) - \frac{\tau \sigma^2}{2} I\left(\mu^{n+1}\big|\Phi_{\sigma}[\mu^{n+1}]\right)\\ &\leq F^{\sigma}(\mu^n) -  \tau \sigma^2 C_{\sigma,F,\rho}\operatorname{KL}\left(\mu^{n+1}\big|\Phi_{\sigma}[\mu^{n+1}]\right) \leq F^{\sigma}(\mu^n) - \tau \sigma C_{\sigma,F,\rho}\left(F^{\sigma}(\mu^{n+1}) - F^{\sigma}(\mu_{\sigma}^*)\right),
\end{aligned}
\end{equation*}
where the second and third inequalities follow from \eqref{eq:sobolev-ineq} and Lemma \ref{lemma:sandwich}, respectively. Let $\kappa \coloneqq 1 + \tau \sigma C_{\sigma,F,\rho} > 1.$ Rearranging the inequality above gives
\begin{equation*}
F^{\sigma}(\mu^{n+1}) - F^{\sigma}(\mu_{\sigma}^*) \leq \kappa^{-1}\left(F^{\sigma}(\mu^{n}) - F^{\sigma}(\mu_{\sigma}^*)\right).
\end{equation*}
The convergence estimate follows by iterating over $n \in \mathbb N.$ 
\end{proof}
\begin{proof}[Proof of Theorem \ref{thm:exp-conv} 2.]
First, note that by Theorem \ref{thm:optim} 2., the iterates $\left(\mu^n\right)_{n \in \mathbb N} \subset \mathfrak C.$ From \eqref{eq:optim-semi-implicit} with $T_{\mu^n}^{\mu^{n+1}}(x)$ in place of $x,$ we obtain
\begin{multline*}
\begin{aligned}
&\frac{1}{\tau}\left(x-T_{\mu^n}^{\mu^{n+1}}(x)\right)=\nabla_\mu F(\mu^n)\left( T_{\mu^n}^{\mu^{n+1}}(x)\right) + \sigma \nabla\log \frac{\mathrm{d}\mu^{n+1}}{\mathrm{d}\rho}\left(T_{\mu^n}^{\mu^{n+1}}(x)\right)\\ &=  \nabla_\mu F(\mu^n)\left( T_{\mu^n}^{\mu^{n+1}}(x)\right) - \nabla_\mu F(\mu^{n+1})\left( T_{\mu^n}^{\mu^{n+1}}(x)\right) + \nabla_\mu F(\mu^{n+1})\left( T_{\mu^n}^{\mu^{n+1}}(x)\right)\\
& \qquad  + \sigma \nabla\log \frac{\mathrm{d}\mu^{n+1}}{\mathrm{d}\rho}\left(T_{\mu^n}^{\mu^{n+1}}(x)\right)\\
&= \nabla_\mu F(\mu^n)\left( T_{\mu^n}^{\mu^{n+1}}(x)\right) - \nabla_\mu F(\mu^{n+1})\left( T_{\mu^n}^{\mu^{n+1}}(x)\right) + \sigma \nabla \log \frac{\mathrm{d}\mu^{n+1}}{\mathrm{d}\Phi_{\sigma}[\mu^{n+1}]}\left(T_{\mu^n}^{\mu^{n+1}}(x)\right).
\end{aligned}
\end{multline*}
Recalling that \eqref{eq:fisher=wass} is
\begin{equation*}
    I\left(\mu^{n+1}\big|\Phi_{\sigma}[\mu^n]\right) = \frac{1}{\tau^2 \sigma^2}\mathcal{W}_2^2(\mu^{n+1}, \mu^n),
\end{equation*}
by Minkowski's inequality, we obtain
\begin{multline*}
\begin{aligned}
&\sigma^2 I\left(\mu^{n+1}\big|\Phi_{\sigma}[\mu^{n+1}]\right) = \sigma^2\left\|\nabla \log \frac{\mathrm{d}\mu^{n+1}}{\mathrm{d}\Phi_{\sigma}[\mu^{n+1}]}\left(T_{\mu^n}^{\mu^{n+1}}\right)\right\|^2_{L_{\mu^n}^2(\mathbb R^d)}\\ &\leq 2\left(\left\|\nabla_\mu F(\mu^{n+1})\left( T_{\mu^n}^{\mu^{n+1}}\right) - \nabla_\mu F(\mu^n)\left( T_{\mu^n}^{\mu^{n+1}}\right)\right\|^2_{L_{\mu^n}^2(\mathbb R^d)} + \frac{1}{\tau^2}\left\|I_d-T_{\mu^n}^{\mu^{n+1}}\right\|^2_{L_{\mu^n}^2(\mathbb R^d)}\right)\\
&\leq 2(L_F')^2 \mathcal{W}_2^2(\mu^{n+1}, \mu^n) + \frac{2}{\tau^2} \mathcal{W}_2^2(\mu^{n+1}, \mu^n)\\
&= 2\left((L_F')^2 + \frac{1}{\tau^2}\right)\tau^2\sigma^2I\left(\mu^{n+1}\big|\Phi_{\sigma}[\mu^n]\right),
\end{aligned}
\end{multline*}
where the second inequality follows from Assumption \ref{ass:lip-intrinsic} and Corollary \ref{corollary:wass-transport}. 
Hence,
\begin{equation*}
I\left(\mu^{n+1}\big|\Phi_{\sigma}[\mu^{n+1}]\right) \leq 2\left(1 + \tau^2 (L_F')^2\right)I\left(\mu^{n+1}\big|\Phi_{\sigma}[\mu^n]\right).
\end{equation*}
Since $\tau < \frac{1}{2L_F'},$ using this inequality in \eqref{eq:dissipation-proximal} gives
\begin{multline*}
\begin{aligned}
F^{\sigma}(\mu^{n+1}) &\leq F^{\sigma}(\mu^{n}) - \frac{\tau \sigma^2 \left(1 - 2\tau L_F'\right)}{2\left(1 + (\tau L_F')^2\right)}I\left(\mu^{n+1}\big|\Phi_{\sigma}[\mu^{n+1}]\right)\\ &\leq F^{\sigma}(\mu^{n}) - \frac{\tau \sigma^2 \left(1 - 2\tau L_F'\right)}{1 +  (\tau L_F')^2} C_{\sigma,F,\rho}\operatorname{KL}\left(\mu^{n+1}\big|\Phi_{\sigma}[\mu^{n+1}]\right)\\
&\leq  F^{\sigma}(\mu^{n}) - \frac{\tau \sigma \left(1 - 2\tau L_F'\right)}{1 +  (\tau L_F')^2} C_{\sigma,F,\rho}\left(F^{\sigma}(\mu^{n+1}) - F^{\sigma}(\mu_\sigma^*)\right),
\end{aligned}
\end{multline*}
where the second and third inequalities follow from \eqref{eq:sobolev-ineq} and Lemma \ref{lemma:sandwich}, respectively. Let $\kappa \coloneqq 1 + \frac{\tau \sigma \left(1 - 2\tau L_F'\right)}{1 + (\tau L_F')^2}C_{\sigma,F,\rho} > 1.$ Then rearranging the inequality above gives
\begin{equation*}
F^{\sigma}(\mu^{n+1}) - F^{\sigma}(\mu_{\sigma}^*) \leq \kappa^{-1}\left(F^{\sigma}(\mu^{n}) - F^{\sigma}(\mu_{\sigma}^*)\right).
\end{equation*}
The convergence estimate follows by iterating over $n \in \mathbb N.$ 
\end{proof}
\begin{proof}[Proof of Theorem \ref{thm:exp-conv} 3.]
Recall that the map $P_{\mu^n}^{\mu^{n+1}}$ is defined by $P_{\mu^n}^{\mu^{n+1}} \coloneqq T_{\nu^{n+1}}^{\mu^{n+1}} \circ T_{\mu^n}^{\nu^{n+1}}$. Note that, for each $n \in \mathbb N,$ $\gamma^n \coloneqq \left(I_d, P_{\mu^n}^{\mu^{n+1}}\right)_{\#} \mu^n$ is a coupling between $\mu^n$ and $\mu^{n+1}.$ Then, since ${P_{\mu^n}^{\mu^{n+1}}}_{\#} \mu^n = \mu^{n+1},$ we obtain
\begin{multline}
\label{eq:fisher-suboptimal-wass}
\begin{aligned}
\sigma^2 &I\left(\mu^{n+1}\big|\Phi_{\sigma}[\mu^{n+1}]\right) = \sigma^2\left\|\nabla \log \frac{\mathrm{d}\mu^{n+1}}{\mathrm{d}\Phi_{\sigma}[\mu^{n+1}]}\left(P_{\mu^n}^{\mu^{n+1}}\right)\right\|^2_{L_{\mu^n}^2(\mathbb R^d)}\\ &= \left\|\nabla_\mu F(\mu^{n+1})\left(P_{\mu^n}^{\mu^{n+1}}\right) + \sigma\nabla\log \frac{\mathrm{d}\mu^{n+1}}{\mathrm{d}\rho}\left(P_{\mu^n}^{\mu^{n+1}}\right)\right\|^2_{L_{\mu^n}^2(\mathbb R^d)} \\
&\leq 2\left(\left\|\nabla_\mu F(\mu^{n+1})\left(P_{\mu^n}^{\mu^{n+1}}\right) - \nabla_\mu F(\mu^n)\right\|^2_{L_{\mu^n}^2(\mathbb R^d)} + \frac{1}{\tau^2}\left\|I_d-P_{\mu^n}^{\mu^{n+1}}\right\|^2_{L_{\mu^n}^2(\mathbb R^d)}\right)\\
&\leq 2\left(2(L_F')^2\mathcal{W}_2^2(\mu^{n+1}, \mu^n) + 2(L_F')^2\left\|I_d-P_{\mu^n}^{\mu^{n+1}}\right\|^2_{L_{\mu^n}^2(\mathbb R^d)} + \frac{1}{\tau^2}\left\|I_d-P_{\mu^n}^{\mu^{n+1}}\right\|^2_{L_{\mu^n}^2(\mathbb R^d)}\right)\\
&\leq 2\left(\frac{1}{\tau^2} + 4(L_F')^2\right)\left\|I_d-P_{\mu^n}^{\mu^{n+1}}\right\|^2_{L_{\mu^n}^2(\mathbb R^d)},
\end{aligned}
\end{multline}
where the second equality follows from \eqref{eq:from-F-to-Gibbs} with $\mu'=\mu=\mu^{n+1},$ the first inequality follows from \eqref{eq:explicitimplicit} and Minkowski's inequality, and the last two inequalities follow from Assumption \ref{ass:lip-intrinsic} and the definition of the Wasserstein distance (cf. \eqref{eq:wass}), respectively.

Hence, by \eqref{eq:descent},
\begin{multline*}
\begin{aligned}
F^{\sigma}(\mu^{n+1}) & \leq F^{\sigma}(\mu^n) -\left(\frac{1}{\tau} - L_F'\right)\left\|I_d-P_{\mu^n}^{\mu^{n+1}}\right\|^2_{L_{\mu^n}^2(\mathbb R^d)}\\
&\leq F^{\sigma}(\mu^n) -\left(\frac{1}{\tau} - L_F'\right)\frac{\sigma^2}{2\left(\frac{1}{\tau^2} + 4(L_F')^2\right)}I\left(\mu^{n+1}\big|\Phi_{\sigma}[\mu^{n+1}]\right)\\
&= F^{\sigma}(\mu^n) -\left(1 - \tau L_F'\right)\frac{\tau\sigma^2}{2\left(1 + 4(\tau L_F')^2\right)}I\left(\mu^{n+1}\big|\Phi_{\sigma}[\mu^{n+1}]\right)\\
&\leq F^{\sigma}(\mu^n) -\left(1 - \tau L_F'\right)\frac{\tau \sigma^2}{1 + 4(\tau L_F')^2}C_{\sigma,F,\rho}\operatorname{KL}\left(\mu^{n+1}\big|\Phi_{\sigma}[\mu^{n+1}]\right),
\end{aligned}
\end{multline*}
where the second inequality follows from \eqref{eq:fisher-suboptimal-wass} and the fact that $\tau < \frac{1}{L_F'},$ whereas the last inequality follows from \eqref{eq:sobolev-ineq}. Let 
\[
\kappa \coloneqq 1 + \left(1 - \tau L_F'\right)\frac{\tau\sigma}{1 + 4(\tau L_F')^2}C_{\sigma,F,\rho} >1\,.
\]
Then using Lemma \ref{lemma:sandwich} in the previous inequality gives
\begin{equation*}
F^{\sigma}(\mu^{n+1}) - F^{\sigma}(\mu_\sigma^*) \leq \kappa^{-1}\left(F^{\sigma}(\mu^{n}) - F^{\sigma}(\mu_\sigma^*)\right).
\end{equation*}
The convergence estimate follows by iterating over $n \in \mathbb N.$ 
\end{proof}
\section{Proximal point scheme}
\label{app:implicit-scheme}
In this section, we present the proofs of the parts of Theorems \ref{thm:well-posedness} and \ref{thm:optim} related to the proximal point scheme \eqref{eq:implicit-JKO}. We start by proving that \eqref{eq:implicit-JKO} admits a unique minimizer.

The proof is a modification of the argument in the proof of \cite[Proposition 4.1]{jko}, and before we present it we give an outline of the main steps:

\emph{Step 1.} Firstly, given $\mu^0 \in \mathcal{P}_2^{\rho}(\mathbb R^d),$ we show that
\begin{equation}
\label{eq:map-to-minimize-implicit}
\mathcal{P}_2^{\rho}(\mathbb R^d) \ni \mu \mapsto \mathcal{F}(\mu) \coloneqq F^{\sigma}(\mu) + \frac{1}{2\tau}\mathcal{W}_2^2(\mu, \mu^0).
\end{equation}
is bounded below on $\mathcal{P}_2^{\rho}(\mathbb R^d).$

\emph{Step 2.} Secondly, we show that any minimizing sequence  $\left(\mu_k\right)_{k \in \mathbb N} \subset \mathcal{P}_2^{\rho}(\mathbb R^d)$ of $\mathcal{F}$ contains a weakly convergent subsequence in $L_{\rho}^1(\mathbb R^d).$

\emph{Step 3.} Thirdly, we show that the weak limit of the converging subsequence is indeed a minimizer $\mu^1 \in \mathcal{P}_2^{\rho}(\mathbb R^d)$ of $\mathcal{F}.$

\emph{Step 4.} Finally, we deduce the uniqueness of the minimizer from the strict convexity of $\operatorname{KL}(\cdot|\rho).$

\begin{proof}[Proof of Theorem \ref{thm:well-posedness} 1.]
\emph{Step 1.} 
By Jensen's inequality, since the map $z \mapsto z \log z$ is convex on $[0, \infty),$ it follows that
\begin{equation}
\label{eq:KL-nonneg}
\operatorname{KL}(\mu|\rho) \geq 0, \quad \text{ for all } \mu \in \mathcal{P}_2^{\rho}(\mathbb R^d).
\end{equation} 
Since $\frac{1}{2\tau}\mathcal{W}_2^2(\cdot, \mu^0) \geq 0$ and since $F$ is bounded below on $\mathcal{P}_2^{\rho}(\mathbb R^d),$ it follows that $\mathcal{F}$ is bounded below on $\mathcal{P}_2^{\rho}(\mathbb R^d),$ and thus $\inf_{\mu \in \mathcal{P}_2^{\rho}(\mathbb R^d)} \mathcal{F}(\mu) > - \infty.$ 

\emph{Step 2.} Let $\left(\mu_k\right)_{k \in \mathbb N} \subset \mathcal{P}_2^{\rho}(\mathbb R^d)$ be a minimizing sequence for $\mathcal{F},$ i.e., $\lim_{k \to \infty} \mathcal{F}(\mu_k) = \inf_{\mu \in \mathcal{P}_2^{\rho}(\mathbb R^d)} \mathcal{F}(\mu).$ Then the sequence $\left(\mathcal{F}(\mu_k)\right)_k$ is bounded on $\mathcal{P}_2^{\rho}(\mathbb R^d),$ i.e., there exists $M_\mathcal{F} > 0$ such that $\left|\mathcal{F}(\mu_k)\right| \leq M_\mathcal{F},$ for all $k \in \mathbb N.$ 

Thus, since $\frac{1}{2\tau}\mathcal{W}_2^2(\cdot, \mu^0) \geq 0,$ it follows that
\begin{equation*}
\operatorname{KL}(\mu_k|\rho) \leq \frac{1}{\sigma}\left(M_\mathcal{F} - F(\mu_k)\right) < \frac{1}{\sigma}\left(M_\mathcal{F} - \inf_{\mu \in \mathcal{P}_2^{\rho}(\mathbb R^d)} F(\mu)\right) < \infty,
\end{equation*}
and together with \eqref{eq:KL-nonneg},
\begin{equation}
\label{eq:KL-bdd}
\left(\operatorname{KL}(\mu_k|\rho)\right)_k \text{ is bounded.}
\end{equation}
From the inequality $|y|^2 \leq 2|x|^2 + 2|x-y|^2,$ which holds for all $x,y \in \mathbb R^d,$ and \eqref{eq:wass}, it follows that
\begin{equation}
\label{eq:inequality-wasserstein}
\int_{\mathbb R^d} |y|^2 \mu''(\mathrm{d}y) \leq \int_{\mathbb R^d} |x|^2 \mu'(\mathrm{d}x) + 2\mathcal{W}_2^2(\mu', \mu''), \quad \text{ for all } \mu', \mu'' \in \mathcal{P}_2^{\rho}(\mathbb R^d),
\end{equation}
Again using \eqref{eq:KL-nonneg} we obtain
\begin{multline*}
M_\mathcal{F} \geq \mathcal{F}(\mu_k) \geq F(\mu_k) + \frac{1}{2\tau}\mathcal{W}_2^2(\mu_k, \mu^0) \geq F(\mu_k) + \frac{1}{4\tau}\int_{\mathbb R^d} |x|^2 \mu_k(\mathrm{d}x) - \frac{1}{2\tau}\int_{\mathbb R^d} |x|^2 \mu^0(\mathrm{d}x).
\end{multline*}
Therefore,
\begin{multline*}
\begin{aligned}
\int_{\mathbb R^d} |x|^2 \mu_k(\mathrm{d}x) &\leq 4\tau\left(M_\mathcal{F} - F(\mu_k)\right) + 2\int_{\mathbb R^d} |x|^2 \mu^0(\mathrm{d}x)\\ &\leq 4\tau\left(M_\mathcal{F} - \inf_{\mu \in \mathcal{P}_2^{\rho}(\mathbb R^d)} F(\mu)\right) + 2\int_{\mathbb R^d} |x|^2 \mu^0(\mathrm{d}x) < \infty,
\end{aligned}
\end{multline*}
hence
\begin{equation}
\label{eq:second-mom-bdd}
\left(\int_{\mathbb R^d} |x|^2 \mu_k(\mathrm{d}x)\right)_k \text{ is bounded.}
\end{equation}
Note that there exists $C > 0$ such that
\begin{equation*}
\left|\min\left\{z \log z, 0\right\}\right| \leq C, \quad \text{ for all } z \geq 0.
\end{equation*}
Hence, we obtain
\begin{equation}
\label{eq:min-bdd}
\int_{\mathbb R^d} \left|\min\left\{\frac{\mathrm{d}\mu}{\mathrm{d}\rho}(x)\log\frac{\mathrm{d}\mu}{\mathrm{d}\rho}(x), 0\right\}\right| \rho(\mathrm{d}x) \leq C\int_{\mathbb R^d} \rho(\mathrm{d}x) = C.
\end{equation}
Furthermore, from \eqref{eq:min-bdd}, we obtain
\begin{equation}
\label{eq:min-seq-bdd}
\left(\int_{\mathbb R^d} \left|\min\left\{\frac{\mathrm{d}\mu_k}{\mathrm{d}\rho}(x)\log\frac{\mathrm{d}\mu_k}{\mathrm{d}\rho}(x), 0\right\}\right| \rho(\mathrm{d}x) \right)_k \text{ is bounded.}
\end{equation}
Since $\max\{z \log z, 0\} = z \log z + |\min\{z \log z,0\}|,$ for all $z \geq 0,$ it follows from \eqref{eq:KL-bdd} and \eqref{eq:min-seq-bdd} that
\begin{equation*}
\left(\int_{\mathbb R^d} \max\left\{\frac{\mathrm{d}\mu_k}{\mathrm{d}\rho}(x)\log\frac{\mathrm{d}\mu_k}{\mathrm{d}\rho}(x), 0\right\} \rho(\mathrm{d}x) \right)_k \text{ is bounded.}
\end{equation*}
Since $\|\frac{\mathrm{d}\mu_k}{\mathrm{d}\rho}\|_{L_{\rho}^1(\mathbb R^d)} = 1,$ for all $k \in \mathbb N,$ we obtain that $\left(\frac{\mathrm{d}\mu_k}{\mathrm{d}\rho}\right)_k$ is uniformly bounded in $L_{\rho}^1(\mathbb R^d).$ As $[0, \infty) \ni z \mapsto \max\{z \log z, 0\}$ is non-negative, increasing and has superlinear growth together with \eqref{eq:KL-bdd} implies via \cite[Theorem 4.5.9]{bogachev2007measure} that $\left(\frac{\mathrm{d}\mu_k}{\mathrm{d}\rho}\right)_k$ is uniformly integrable. Consequently, according to the Dunford-Pettis theorem (see \cite[Corollary 4.7.19]{bogachev2007measure}), there exists $\mu^* \in \mathcal{P}^{\rho}(\mathbb R^d)$ such that (at least for a subsequence)
\begin{equation}
\label{eq:weak-conv}
\frac{\mathrm{d}\mu_k}{\mathrm{d}\rho} \to \frac{\mathrm{d}\mu^*}{\mathrm{d}\rho} \text{ weakly in } L_{\rho}^1(\mathbb R^d) \text{ as } k \to \infty, 
\end{equation}
i.e.,
\begin{equation*}
\int_{\mathbb R^d} h(x) \frac{\mathrm{d}\mu_k}{\mathrm{d}\rho}(x) \rho(\mathrm{d}x) \to \int_{\mathbb R^d} h(x) \frac{\mathrm{d}\mu^*}{\mathrm{d}\rho}(x) \rho(\mathrm{d}x) \text{ as } k \to \infty,
\end{equation*}
for all $h \in L_{\rho}^{\infty}(\mathbb R^d).$

\emph{Step 3.} Observe that any continuous bounded function $g:\mathbb R^d \to \mathbb R$ is in $L_{\rho}^{\infty}(\mathbb R^d),$ and hence, by \eqref{eq:weak-conv}, we obtain as $k \to \infty,$
\begin{multline*}
\int_{\mathbb R^d} g(x) \mu_k(\mathrm{d}x) = \int_{\mathbb R^d} g(x) \frac{\mathrm{d}\mu_k}{\mathrm{d}\rho}(x)\rho(\mathrm{d}x) \to \int_{\mathbb R^d} g(x) \frac{\mathrm{d}\mu^*}{\mathrm{d}\rho}(x)\rho(\mathrm{d}x) = \int_{\mathbb R^d} g(x) \mu^*(\mathrm{d}x),
\end{multline*} 
i.e., $\mu_k \to \mu^*$ weakly (with respect to the topology of probability measures convergence) as $k \to \infty.$ Hence, by Fatou's lemma for weak convergence of measure (see e.g. \cite[Lemma 13]{zhenjiefict}) and \eqref{eq:second-mom-bdd},
\begin{equation*}
    \int_{\mathbb{R}^d}|x|^2\mu^*(\mathrm{d}x)\leq \liminf_{k \to \infty} \int_{\mathbb{R}^d}|x|^2\mu_k(\mathrm{d}x) < \infty,
\end{equation*}
which shows that $\mu^*$ has finite second moment and hence we can conclude that $\mu^* \in \mathcal{P}_2^{\rho}(\mathbb R^d)$.

Since both $F$ and $\operatorname{KL}(\cdot|\rho)$ are lower semi-continuous with respect to weak convergence of probability measures, it follows that
\begin{equation*}
  F(\mu^*) \leq  \liminf_{k \to \infty }F(\mu_k) \quad \text{and} \quad 
\operatorname{KL}(\mu^*|\rho) \leq \liminf_{k \to \infty} \operatorname{KL}(\mu_k|\rho).
\end{equation*}
By \eqref{eq:second-mom-bdd} and continuity of $\mathcal{W}_2$ (see \cite[Corollary 6.11]{villani2008optimal}), we have
\begin{equation*}
\mathcal{W}_2^2(\mu^0, \mu^*) = \lim_{k \to \infty} \mathcal{W}_2^2(\mu^0, \mu_k).
\end{equation*}
Putting everything together, we obtain
\begin{equation*}
\mathcal{F}(\mu^*) \leq \liminf_{k \to \infty} \mathcal{F}(\mu_k) = \inf_{\mu \in \mathcal{P}_2^{\rho}(\mathbb R^d)} \mathcal{F}(\mu).
\end{equation*}
On the other hand, from the definition of infimum, we have
\begin{equation*}
\mathcal{F}(\mu^*) \geq \inf_{\mu \in \mathcal{P}_2^{\rho}(\mathbb R^d)} \mathcal{F}(\mu).
\end{equation*}
Hence, $\mathcal{F}(\mu^*) = \inf_{\mu \in \mathcal{P}_2^{\rho}(\mathbb R^d)} \mathcal{F}(\mu),$ and therefore $\mu^* \in \mathcal{P}_2^{\rho}(\mathbb R^d)$ is a minimizer of $\mathcal{F},$ which we shall denote by $\mu^{1}.$

\emph{Step 4.} The uniqueness of the minimizer of $\mathcal{F}$ follows from Assumption \ref{ass:convexity}, convexity of $\mathcal{P}_2^{\rho}(\mathbb R^d) \ni \mu \mapsto \mathcal{W}_2^2(\mu, \mu^0),$ and strict convexity of $\mathcal{P}_2^{\rho}(\mathbb R^d) \ni \mu \mapsto \operatorname{KL}(\mu|\rho).$

Therefore, it follows inductively that, for each $n \in \mathbb N,$ given $\mu^n \in \mathcal{P}_2^{\rho}(\mathbb R^d),$ the scheme \eqref{eq:implicit-JKO} admits a unique minimizer $\mu^{n+1} \in \mathcal{P}_2^{\rho}(\mathbb R^d).$ Hence, if $\mu^0 \in \mathcal{P}_2^{\rho}(\mathbb R^d),$ then $\left(\mu^n\right)_{n \in \mathbb N} \subset \mathcal{P}_2^{\rho}(\mathbb R^d)$ along the scheme \eqref{eq:implicit-JKO}. 
\end{proof}
Therefore, via Theorem \ref{thm:existence-optimal-coupling}, we obtain 
\begin{corollary}[Existence of optimal transport maps along \eqref{eq:implicit-JKO}]
\label{corollary:existence-optimal-transport-implicit}
Let $\nu \in \mathcal{P}_2(\mathbb R^d)$ and $\rho \in \mathcal{P}_2^{\lambda}(\mathbb R^d).$ If $F$ is weakly lower semi-continuous and convex in the sense of \eqref{eq:convexity_classical}, then given $\mu^0 \in \mathcal{P}_2^{\rho}(\mathbb R^d),$ there exists a unique $\mu^n$-a.e. optimal transport map $T_{\mu^n}^{\nu}:\mathbb R^d \to \mathbb R^d$ from $\mu^n$ to $\nu.$ In particular, if $\nu = \mu^{n+1},$ there also exists a unique $\mu^{n+1}$-a.e. optimal transport map $T_{\mu^{n+1}}^{\mu^n}:\mathbb R^d \to \mathbb R^d$ such that $T_{\mu^{n+1}}^{\mu^n} \circ T_{\mu^n}^{\mu^{n+1}} = I_d,$ $\mu^n$-a.e and $T_{\mu^n}^{\mu^{n+1}} \circ T_{\mu^{n+1}}^{\mu^n} = I_d,$ $\mu^{n+1}$-a.e.
\end{corollary}

We make the following remark on the assumptions in Theorem \ref{thm:well-posedness} 1.

\begin{remark}\label{remark:weak_lower_semi_continuity}
Note that Assumption \ref{ass:convexity} implies weak lower semi-continuity of $F$. 
Indeed, by Definition \ref{def:flat-derivative}, the flat derivative $\frac{\delta F}{\delta \mu}$ is continuous and has at most quadratic growth uniformly in $\mu.$ Hence, by \cite[Definition 6.8 (iv)]{villani2008optimal}, the weak convergence $\mu_k \to \mu^*$ as $k \to \infty$ implies that
\begin{equation*}
   \lim_{k \to \infty} \int_{\mathbb R^d} \frac{\delta F}{\delta \mu} (\mu^*, x)\mu_k(\mathrm{d}x) = \int_{\mathbb R^d} \frac{\delta F}{\delta \mu} (\mu^*, x)\mu^*(\mathrm{d}x).
\end{equation*}
Using \eqref{eq:convexity_via_flat_derivative}, which holds due to \cite[Lemma 4.1]{10.1214/20-AIHP1140}, we get
\begin{equation*}
    F(\mu_k)-F(\mu^*) \geq \int_{\mathbb R^d} \frac{\delta F}{\delta \mu} (\mu^*, x)(\mu_k-\mu^*)(\mathrm{d}x)
\end{equation*}
and passing to the limit $k \to \infty$, we obtain
\begin{equation*}
    \liminf_{k \to \infty }F(\mu_k) \geq F(\mu^*),
\end{equation*}
which shows the weak lower semi-continuity of $F$.
\end{remark}

Next, we present the proof of Theorem \ref{thm:optim} 1., which is done in two steps:

\emph{Step 1.} We show by induction that, for each $n \in \mathbb N,$ given $\mu^n \in \mathcal{P}_2^{\rho}(\mathbb R^d),$ the unique minimizer $\mu^{n+1}$ of \eqref{eq:implicit-JKO} belongs to $\mathfrak C.$

\emph{Step 2.} We show that $\mu^{n+1} \in \mathfrak C$ satisfies a first-order optimality condition.
\begin{proof}[Proof of Theorem \ref{thm:optim} 1.]
\emph{Step 1.} Given $\mu^0 \in \mathcal{P}_2^{\rho}(\mathbb R^d),$ Theorem \ref{thm:well-posedness} 1. guarantees the existence and uniqueness of a minimizer $\mu^1 \in \mathcal{P}_2^{\rho}(\mathbb R^d)$ for \eqref{eq:map-to-minimize-implicit}. We will now prove that $\mu^1 \in \mathfrak C.$ By Definition \ref{def:slope}, the metric slope of the relative entropy $\operatorname{KL}(\cdot|\rho)$ at $\mu \in \mathcal{P}_2^{\rho}(\mathbb R^d)$ is defined by
\begin{equation*}
|\mathfrak d \operatorname{KL}(\cdot|\rho)|(\mu) = \limsup_{\nu \to \mu} \frac{\left(\operatorname{KL}(\mu|\rho)-\operatorname{KL}(\nu|\rho)\right)_{+}}{\mathcal{W}_2(\nu, \mu)}.
\end{equation*}
Since $\mu^1 \in \mathcal{P}_2^{\rho}(\mathbb R^d)$ is a minimizer for \eqref{eq:map-to-minimize-implicit}, it follows that for any $\mu \in \mathcal{P}_2^{\rho}(\mathbb R^d),$
\begin{multline}
\label{eq:slope-estimate-implicit}
\begin{aligned}
&\operatorname{KL}(\mu^1|\rho) - \operatorname{KL}(\mu|\rho)
\leq \frac{1}{\sigma}\left(F(\mu)-F(\mu^1)\right) + \frac{1}{2\tau\sigma}\left(\mathcal{W}_2^2(\mu, \mu^0) - \mathcal{W}_2^2(\mu^1, \mu^0)\right)\\
&= \frac{1}{\sigma}\left(F(\mu)-F(\mu^1)\right) + \frac{1}{2\tau\sigma}\left(\mathcal{W}_2(\mu, \mu^0) - \mathcal{W}_2(\mu^1, \mu^0)\right)\left(\mathcal{W}_2(\mu, \mu^0) + \mathcal{W}_2(\mu^1, \mu^0)\right)\\
&\leq \frac{1}{\sigma}\left(F(\mu)-F(\mu^1)\right) + \frac{1}{2\tau\sigma}\mathcal{W}_2(\mu, \mu^1)\left(\mathcal{W}_2(\mu, \mu^0) + \mathcal{W}_2(\mu^1, \mu^0)\right),
\end{aligned}
\end{multline}
where the last inequality follows from the triangle inequality applied to $\mathcal{W}_2.$ 

By Assumption \ref{assumption:L-derivative}, $F$ is Wasserstein differentiable at any $\mu \in \mathcal{P}_2(\mathbb R^d).$ Since $\mu^1,\mu \in \mathcal{P}_2^{\lambda}(\mathbb R^d)$ due to $\rho \ll \lambda,$ it follows that there exists a unique optimal coupling $\gamma = \left(I_d,T_{\mu^1}^\mu\right)_\#\mu^1 \in \Gamma_o(\mu^1, \mu).$ Then, using Definition \ref{def:wass-differentiability}, we obtain
\begin{align*}
    F(\mu)-F(\mu^1) &= \int_{\mathbb R^d} \langle \nabla_\mu F(\mu^1)(x), T_{\mu^1}^\mu(x)-x\rangle\ \mu^1(\mathrm{d}x) + o\left(\mathcal{W}_2(\mu^1,\mu)\right)\\
    &\leq \left\|\nabla_\mu F(\mu^1)\right\|_{L_{\mu^1}^2(\mathbb R^d)}\left\|T_{\mu^1}^\mu-I_d\right\|_{L_{\mu^1}^2(\mathbb R^d)} + o\left(\mathcal{W}_2(\mu^1,\mu)\right)\\
    &=\left\|\nabla_\mu F(\mu^1)\right\|_{L_{\mu^1}^2(\mathbb R^d)}\mathcal{W}_2(\mu^1,\mu) + o\left(\mathcal{W}_2(\mu^1,\mu)\right),
\end{align*}
where the last equality follows from optimality of $\gamma.$

Therefore, for any $\mu \neq \mu^1,$ dividing \eqref{eq:slope-estimate-implicit} by $\mathcal{W}_2(\mu,\mu^1)$ gives
\begin{align*}
\frac{\left(\operatorname{KL}(\mu^1|\rho) - \operatorname{KL}(\mu|\rho)\right)_+}{\mathcal{W}_2(\mu,\mu^1)} &\leq \frac{1}{\sigma}\left(\left\|\nabla_\mu F(\mu^1)\right\|_{L_{\mu^1}^2(\mathbb R^d)} + \frac{o\left(\mathcal{W}_2(\mu,\mu^1)\right)}{\mathcal{W}_2(\mu,\mu^1)}\right)\\ 
&+ \frac{1}{2\tau\sigma}\left(\mathcal{W}_2(\mu, \mu^0) + \mathcal{W}_2(\mu^1, \mu^0)\right).
\end{align*}
Note that $\left\|\nabla_\mu F(\mu^1)\right\|_{L_{\mu^1}^2(\mathbb R^d)} < \infty$ since $\nabla_\mu F(\mu^1)$ grows at most linearly. Taking the limsup as $\mu \to \mu^1$ in the sense of weak convergence, applying the continuity of $\mathcal{W}_2$ via \cite[Corollary 6.11]{villani2008optimal}, and using $\lim_{h \to 0}\frac{o(h)}{h} = 0,$ we obtain
\begin{equation*}
|\mathfrak d \operatorname{KL}(\cdot|\rho)|(\mu^1) \leq \frac{1}{\sigma}\left\|\nabla_\mu F(\mu^1)\right\|_{L_{\mu^1}^2(\mathbb R^d)} + \frac{1}{\tau\sigma}\mathcal{W}_2(\mu^1, \mu^0) < \infty.
\end{equation*}
Hence, using Theorem \ref{thm:subdiff-entropy}, it follows that $\frac{\mathrm{d}\mu^1}{\mathrm{d}\rho} \in W_{\lambda,\text{loc}}^{1,1}(\mathbb R^d),$ $\frac{\left|\nabla\frac{\mathrm{d}\mu^1}{\mathrm{d}\rho}\right|^2}{\frac{\mathrm{d}\mu^1}{\mathrm{d}\rho}} \in L_{\rho}^1(\mathbb R^d),$ and $|\mathfrak d \operatorname{KL}(\cdot|\rho)|^2(\mu^1) = I(\mu_1|\rho).$ Since $\frac{\left|\nabla\frac{\mathrm{d}\mu^1}{\mathrm{d}\rho}\right|^2}{\frac{\mathrm{d}\mu^1}{\mathrm{d}\rho}} \in L_{\rho}^1(\mathbb R^d)$ we have $\nabla \sqrt \frac{\mathrm{d}\mu^1}{\mathrm{d}\rho} \in L_{\rho}^{2}(\mathbb R^d).$ In addition, $\mu^1 \in \mathcal{P}_2^{\rho}(\mathbb R^d)$ implies $\sqrt \frac{\mathrm{d}\mu^1}{\mathrm{d}\rho} \in L_{\rho}^{2}(\mathbb R^d),$ and therefore we obtain that $\mu^1 \in \mathfrak C.$

Therefore, it follows inductively that, for each $n \in \mathbb N,$ given $\mu^n \in \mathcal{P}_2^{\rho}(\mathbb R^d),$ the unique minimizer $\mu^{n+1}$ of \eqref{eq:implicit-JKO} belongs to $\mathfrak C.$ Hence, if $\mu^0 \in \mathcal{P}_2^{\rho}(\mathbb R^d),$ then $\left(\mu^n\right)_{n \in \mathbb N} \subset \mathfrak C$ along the scheme \eqref{eq:implicit-JKO}.

\emph{Step 2.} Since $\operatorname{KL}(\cdot|\rho)$ admits a unique Wasserstein sub-differential and $F$ is Wasserstein differentiable by Assumption \ref{assumption:L-derivative}, \cite[Lemma 10.1.2]{ambrosio2008gradient} allows us to write the first-order optimality condition for \eqref{eq:implicit-JKO}.

Indeed, since $F$ and $\operatorname{KL}(\cdot|\rho)$ are weakly lower semi-continuous, it follows that $F^{\sigma}(\mu) = F(\mu) + \sigma \operatorname{KL}(\mu|\rho)$ is also weakly lower semi-continuous, with $F^{\sigma}(\mu) < +\infty,$ for all $\mu \in \mathcal{P}_2^{\rho}(\mathbb R^d).$ 

By Theorem \ref{thm:well-posedness} 1., \eqref{eq:implicit-JKO} admits a unique minimizer $\mu^{n+1} \in \mathcal{P}_2^{\rho}(\mathbb R^d).$ Moreover, since $\rho \ll \lambda,$ we have $\mu^{n+1} \in \mathcal{P}_2^{\lambda}(\mathbb R^d).$ By Corollary \ref{corollary:existence-optimal-transport-implicit}, there exists a unique $\mu^{n+1}$-a.e. optimal transport map $T_{\mu^{n+1}}^{\mu^n}:\mathbb R^d \to \mathbb R^d$ from $\mu^{n+1}$ to $\mu^n.$ Therefore, by \cite[Lemma 10.1.2]{ambrosio2008gradient}, $\left|\mathfrak d F^{\sigma}\right|(\mu^{n+1}) <  +\infty$ and 
\begin{equation*}
\frac{1}{\tau}\left(T_{\mu^{n+1}}^{\mu^n} - I\right) \in \partial^- F^{\sigma}(\mu^{n+1}),
\end{equation*}
where $\partial^-$ denotes the Wasserstein sub-differential (cf. Definition \ref{def:wass-sub-and-sup}). By Assumption \ref{assumption:L-derivative} and using the fact that $\mu^{n+1} \in \mathfrak C$ together with Theorem \ref{thm:subdiff-entropy}, we obtain 
\begin{equation*}
\partial^- F^{\sigma}(\mu^{n+1}) = \left\{\nabla_\mu F(\mu^{n+1}) + \sigma \nabla \log \frac{\mathrm{d}\mu^{n+1}}{\mathrm{d}\rho}\right\},
\end{equation*}
and hence the conclusion follows.
\end{proof}
\section{Prox-linear scheme}
\label{app:semi-implicit-scheme}
In this section, we present the proofs of the parts of Theorems \ref{thm:well-posedness} and \ref{thm:optim} related to the prox-linear scheme \eqref{eq:semi-implicit-JKO}. We start by proving that \eqref{eq:semi-implicit-JKO} admits a unique minimizer. The proof follows the same steps as the proof of Theorem \ref{thm:well-posedness} 1..
\begin{proof}[Proof of Theorem \ref{thm:well-posedness} 2.]
Given $\mu^0 \in \mathcal{P}_2^{\rho}(\mathbb R^d),$ consider the map
\begin{equation}
\label{eq:map-to-minimize-semi-implicit}
\mathcal{P}_2^{\rho}(\mathbb R^d) \ni \mu \mapsto \mathcal{G}(\mu) \coloneqq \int_{\mathbb R^d} \frac{\delta F}{\delta \mu}(\mu^0, x)(\mu-\mu^0)(\mathrm{d}x) + \sigma \operatorname{KL}(\mu|\rho) + \frac{1}{2\tau}\mathcal{W}_2^2(\mu, \mu^0).
\end{equation}
Observe that
\begin{equation*}
\argmin_{\mu \in \mathcal{P}_2^{\rho}(\mathbb R^d)} \mathcal{G}(\mu) = \argmin_{\mu \in \mathcal{P}_2^{\rho}(\mathbb R^d)} \left\{\int_{\mathbb R^d} \frac{\delta F}{\delta \mu}(\mu^0, x)\mu(\mathrm{d}x) + \sigma \operatorname{KL}(\mu|\rho) + \frac{1}{2\tau}\mathcal{W}_2^2(\mu, \mu^0)\right\},
\end{equation*}
and therefore it suffices to show that $\mu^{1} \in \mathcal{P}_2^{\rho}(\mathbb R^d)$ is the unique minimizer of $\mathcal{G}.$ The existence proof follows the same steps as the proof of Theorem \ref{thm:well-posedness} 1., with two key differences. First, to show that $\mathcal{G}$ is bounded below, note that by Definition \ref{def:flat-derivative}, there exists $C_F>0$ such that $\left|\frac{\delta F}{\delta \mu}(\mu, x)\right|\leq C_F(1+|x|^2),$ for all $\mu \in \mathcal{P}_2^{\rho}(\mathbb R^d)$ and $x \in \mathbb R^d.$ Hence
\begin{align*}
    \int_{\mathbb R^d} \frac{\delta F}{\delta \mu}(\mu^0, x)\mu(\mathrm{d}x) \geq -C_F \left(1+ \int_{\mathbb R^d}|x|^2\mu(\mathrm{d}x)\right),
\end{align*}
which is finite for all $\mu \in \mathcal{P}_2^{\rho}(\mathbb R^d).$ Using \eqref{eq:KL-nonneg} and the non-negativity of $\frac{1}{2\tau}\mathcal{W}_2^2(\cdot,\mu^0),$ it follows that $\inf_{\mu \in \mathcal{P}_2^{\rho}(\mathbb R^d)} \mathcal{G}(\mu) > -\infty.$ 

Second, instead of relying on the weak lower semi-continuity of $F,$ observe that since $x \mapsto \frac{\delta F}{\delta \mu}(\mu^0, x)$ is continuous with at most quadratic growth, it follows that
\begin{equation*}
    \lim_{k \to \infty}\int_{\mathbb R^d} \frac{\delta F}{\delta \mu}(\mu^0, x) \mu_k(\mathrm{d}x) = \int_{\mathbb R^d} \frac{\delta F}{\delta \mu}(\mu^0, x) \mu^*(\mathrm{d}x)
\end{equation*}
whenever $\mu_k \to \mu^*$ weakly as $k \to \infty$ (see \cite[Definition 6.8 (iv)]{villani2008optimal}.

The uniqueness of the minimizer of $\mathcal{G}$ follows from the linearity of $\mathcal{P}_2^{\rho}(\mathbb R^d) \ni \mu \mapsto \int_{\mathbb R^d}\frac{\delta F}{\delta \mu}(\mu^0, x)\mu(\mathrm{d}x)$, the convexity of $\mathcal{P}_2^{\rho}(\mathbb R^d) \ni \mu \mapsto \mathcal{W}_2^2(\mu, \mu^0),$ and the strict convexity of $\mathcal{P}_2^{\rho}(\mathbb R^d) \ni \mu \mapsto \operatorname{KL}(\mu|\rho).$

Therefore, it follows inductively that, for each $n \in \mathbb N,$ given $\mu^n \in \mathcal{P}_2^{\rho}(\mathbb R^d),$ the scheme \eqref{eq:semi-implicit-JKO} admits a unique minimizer $\mu^{n+1} \in \mathcal{P}_2^{\rho}(\mathbb R^d).$ Hence, if $\mu^0 \in \mathcal{P}_2^{\rho}(\mathbb R^d),$ then $\left(\mu^n\right)_{n \in \mathbb N} \subset \mathcal{P}_2^{\rho}(\mathbb R^d)$ along the scheme \eqref{eq:semi-implicit-JKO}.
\end{proof}
Therefore, via Theorem \ref{thm:existence-optimal-coupling}, we obtain 
\begin{corollary}[Existence of optimal transport maps along \eqref{eq:semi-implicit-JKO}]
\label{corollary:existence-optimal-transport-semi-implicit}
Let $\nu \in \mathcal{P}_2(\mathbb R^d),$ $\rho \in \mathcal{P}_2^{\lambda}(\mathbb R^d).$ If $F \in \mathcal{C}^1,$ then given $\mu^0 \in \mathcal{P}_2^{\rho}(\mathbb R^d),$ there exists a unique $\mu^n$-a.e. optimal transport map $T_{\mu^n}^{\nu}:\mathbb R^d \to \mathbb R^d$ from $\mu^n$ to $\nu.$ In particular, if $\nu = \mu^{n+1},$ there also exists a unique $\mu^{n+1}$-a.e. optimal transport map $T_{\mu^{n+1}}^{\mu^n}:\mathbb R^d \to \mathbb R^d$ such that $T_{\mu^{n+1}}^{\mu^n} \circ T_{\mu^n}^{\mu^{n+1}} = I_d,$ $\mu^n$-a.e and $T_{\mu^n}^{\mu^{n+1}} \circ T_{\mu^{n+1}}^{\mu^n} = I_d,$ $\mu^{n+1}$-a.e..
\end{corollary}
Next, we present the proof of Theorem \ref{thm:optim} 2. The proof follows the same steps as the proof of Theorem \ref{thm:optim} 1.
\begin{proof}[Proof of Theorem \ref{thm:optim} 2.]
\emph{Step 1.} Since $\mu^0 \in \mathcal{P}_2^{\rho}(\mathbb R^d),$ $F \in \mathcal{C}^1,$ Theorem \ref{thm:well-posedness} 2. guarantees the existence and uniqueness of a minimizer $\mu^1 \in \mathcal{P}_2^{\rho}(\mathbb R^d)$ for \eqref{eq:map-to-minimize-semi-implicit}. We will now prove that $\mu^1 \in \mathfrak C.$

Since $\mu^1 \in \mathcal{P}_2^{\rho}(\mathbb R^d)$ is a minimizer for \eqref{eq:map-to-minimize-semi-implicit}, it follows that for any $\mu \in \mathcal{P}_2^{\rho}(\mathbb R^d),$
\begin{multline}
\label{eq:slope-KL-estimate}
\begin{aligned}
\operatorname{KL}(\mu^1|\rho) - \operatorname{KL}(\mu|\rho)
&\leq \frac{1}{\sigma}\int_{\mathbb R^d} \frac{\delta F}{\delta \mu}(\mu^0, x)(\mu-\mu^1)(\mathrm{d}x)\\ &+ \frac{1}{2\tau\sigma}\mathcal{W}_2(\mu, \mu^1)\left(\mathcal{W}_2(\mu, \mu^0) + \mathcal{W}_2(\mu^1, \mu^0)\right).
\end{aligned}
\end{multline}
Observe that the map $\hat{\mathcal{F}}(\mu)\coloneqq\int_{\mathbb R^d} \frac{\delta F}{\delta \mu}(\mu^0, x)\mu(\mathrm{d}x)$ is flat differentiable by Definition \ref{def:flat-derivative}, with $\frac{\delta \hat{\mathcal{F}}}{\delta \mu}(\mu,\cdot) = \frac{\delta F}{\delta \mu}(\mu^0, \cdot),$ and by Assumption \ref{assumption:L-derivative}, it is also Wasserstein differentiable, with $\nabla_\mu \hat{\mathcal{F}}(\mu) = \nabla_\mu F(\mu^0).$ 

Since $\mu^1,\mu \in \mathcal{P}_2^{\lambda}(\mathbb R^d)$ due to $\rho \ll \lambda,$ there exists a unique optimal coupling $\gamma = \left(I_d,T_{\mu^1}^\mu\right)_\#\mu^1 \in \Gamma_o(\mu^1, \mu).$ By Definition \ref{def:wass-differentiability}, this gives
\begin{align*}
    \hat{\mathcal{F}}(\mu)-\hat{\mathcal{F}}(\mu^1) &= \int_{\mathbb R^d} \langle \nabla_\mu \hat{\mathcal{F}}(\mu^1)(x), T_{\mu^1}^\mu(x)-x\rangle\ \mu^1(\mathrm{d}x) + o\left(\mathcal{W}_2(\mu^1,\mu)\right)\\
    &\leq \left\|\nabla_\mu \hat{\mathcal{F}}(\mu^1)\right\|_{L_{\mu^1}^2(\mathbb R^d)}\left\|T_{\mu^1}^\mu-I_d\right\|_{L_{\mu^1}^2(\mathbb R^d)} + o\left(\mathcal{W}_2(\mu^1,\mu)\right)\\
    &=\left\|\nabla_\mu \hat{\mathcal{F}}(\mu^1)\right\|_{L_{\mu^1}^2(\mathbb R^d)}\mathcal{W}_2(\mu^1,\mu) + o\left(\mathcal{W}_2(\mu^1,\mu)\right)\\
    &= \left\|\nabla_\mu F(\mu^0)\right\|_{L_{\mu^1}^2(\mathbb R^d)}\mathcal{W}_2(\mu^1,\mu) + o\left(\mathcal{W}_2(\mu^1,\mu)\right),
\end{align*}
where the penultimate equality follows from optimality of $\gamma.$ Note that $\left\|\nabla_\mu F(\mu^0)\right\|_{L_{\mu^1}^2(\mathbb R^d)} < \infty$ since $\nabla_\mu F(\mu^0)$ grows at most linearly. The remainder of the proof proceeds identically to the proof of Theorem \ref{thm:optim} 1.

Therefore, it follows inductively that, for each $n \in \mathbb N,$ given $\mu^n \in \mathcal{P}_2^{\rho}(\mathbb R^d),$ the unique minimizer $\mu^{n+1}$ of \eqref{eq:semi-implicit-JKO} belongs to $\mathfrak C.$ Hence, if $\mu^0 \in \mathcal{P}_2^{\rho}(\mathbb R^d),$ then $\left(\mu^n\right)_{n \in \mathbb N} \subset \mathfrak C$ along the scheme \eqref{eq:semi-implicit-JKO}.

\emph{Step 2.} Almost identical to the proof of \emph{Step 2} of Theorem \ref{thm:optim} 1. once we replace $F$ by $\int_{\mathbb R^d} \frac{\delta F}{\delta \mu}(\mu^n, x)(\mu-\mu^n)(\mathrm{d}x),$ and, for each $n \in \mathbb N,$ define
\begin{equation*}
    J_n(\mu) \coloneqq \int_{\mathbb R^d} \frac{\delta F}{\delta \mu}(\mu^n, x)(\mu-\mu^n)(\mathrm{d}x) + \sigma\operatorname{KL}(\mu|\rho),
\end{equation*}
for all $\mu \in \mathcal{P}_2^{\rho}(\mathbb R^d).$ The main distinction is that, instead of relying on the weak lower semi-continuity of $F,$ we observe that since $\frac{\delta F}{\delta \mu}(\mu, \cdot)$ is continuous and grows at most quadratically uniformly in $\mu$, it follows that
\begin{equation*}
    \lim_{k \to \infty}\int_{\mathbb R^d} \frac{\delta F}{\delta \mu}(\mu^n, x) \mu_k(\mathrm{d}x) = \int_{\mathbb R^d} \frac{\delta F}{\delta \mu}(\mu^n, x) \mu^*(\mathrm{d}x)
\end{equation*}
whenever $\mu_k \to \mu^*$ weakly as $k \to \infty$ (see \cite[Definition 6.8 (iv)]{villani2008optimal}. Combining this with the weak lower semi-continuity of $\operatorname{KL}(\cdot|\rho),$ we conclude that $J_n$ is weakly lower semi-continuous.
\end{proof}
\section{Proximal gradient scheme}
\label{app:proximal-scheme}
In this section, we present the proofs of the parts of Theorems \ref{thm:well-posedness} and \ref{thm:optim} related to the proximal gradient scheme \eqref{eq:proximal-JKO}.

Before applying the same argument as in the proof of Theorem \ref{thm:well-posedness} 1.-2. to show the existence and uniqueness of a minimizer of the JKO step \eqref{eq:proximal-JKO}, we prove that the map $I_d-\tau\nabla_\mu F(\mu^n)(\cdot)$ is an optimal transport map from $\mu^n$ to $\nu^{n+1} = \left(I_d-\tau\nabla_\mu F(\mu^n)(\cdot)\right)_{\#}\mu^n,$ and moreover that $\nu^{n+1} \in \mathcal{P}_2^{\rho}(\mathbb R^d).$ As we saw in the proofs of Theorem \ref{thm:well-posedness} 1.-2., the previous step in the JKO update, in this case $\nu^{n+1},$ needs to be absolutely continuous with respect to $\rho.$ This is proved in the following lemma, which is a generalization of \cite[Lemma 2]{korbaproximal}.
\begin{lemma}
\label{lemma:pushforward-is-optimal}
Let Assumption \ref{assumption:L-derivative}, \ref{ass:lip-intrinsic} hold. Let $\mu \in \mathcal{P}_2(\mathbb R^d),$ $\sigma > 0$ and $\nu = \left(I_d-\tau\nabla_\mu F(\mu)(\cdot)\right)_{\#}\mu.$ If $\tau < \frac{1}{L_F'},$ then the optimal transport map from $\mu$ to $\nu$ is given by
\begin{equation*}
T_{\mu}^{\nu} = I_d-\tau\nabla_\mu F(\mu)(\cdot).
\end{equation*}
Moreover, if $\mu \in \mathcal{P}_2^{\rho}(\mathbb R^d)$ and $\rho \sim \lambda,$ then $\nu \in \mathcal{P}_2^{\rho}(\mathbb R^d).$
\end{lemma}
\begin{proof}
Note that Assumption \ref{assumption:L-derivative} together with that fact that $\mu \in \mathcal{P}_2(\mathbb R^d)$ imply that
\begin{equation*}
\int_{\mathbb R^d} \left|x- \tau\nabla_\mu F(\mu)(x)\right|^2 \mu(\mathrm{d}x) < \infty,
\end{equation*}
and hence $\nu \in \mathcal{P}_2(\mathbb R^d).$

Since $\left(I_d-\tau\nabla_\mu F(\mu)(\cdot)\right)_{\#}\mu$ is the pushforward measure from $\mu$ to $\nu,$ by Theorem \ref{thm:sufficient-cond-optimal-convexity}, it suffices to show that $I_d-\tau\nabla_\mu F(\mu)(\cdot)$ can be written as the gradient of a convex function. Let $u(x) \coloneqq \frac12|x|^2 - \tau \frac{\delta F}{\delta \mu}(\mu, x).$ Then, for any $x \in \mathbb R^d,$ $\nabla u(x) = x - \tau \nabla_\mu F(\mu)(x).$ Moreover,
\begin{multline}
\label{eq:u-str-conv}
\begin{aligned}
(x-y)\cdot \left(\nabla u(x) - \nabla u(y)\right) &=  (x-y)\cdot \left(x-y - \tau \left(\nabla_\mu F(\mu)(x) - \nabla_\mu F(\mu)(y)\right)\right)\\
&= |x-y|^2 - \tau (x-y)\cdot\left(\nabla_\mu F(\mu)(x) - \nabla_\mu F(\mu)(y)\right)\\
&\geq |x-y|^2 - \tau|x-y| \left|\nabla_\mu F(\mu)(x) - \nabla_\mu F(\mu)(y)\right|\\
&\geq \left(1- \tau L_F'\right)|x-y|^2.
\end{aligned}
\end{multline}
Since by assumption $\tau < \frac{1}{L_F'},$ it follows that $u$ is $\left(1- \tau L_F'\right)$-strongly convex and moreover $\nabla u$ is injective. By strong convexity of $u$ and Theorem \ref{thm:sufficient-cond-optimal-convexity}, we obtain that $\nabla u$ is an optimal transport map, and we denote it by
\begin{equation*}
T_{\mu}^{\nu} = I_d-\tau\nabla_\mu F(\mu)(\cdot).
\end{equation*}
Since $\mu \in \mathcal{P}_2^{\rho}(\mathbb R^d)$ and $\rho \ll \lambda$, we have $\mu \in \mathcal{P}_2^{\lambda}(\mathbb R^d)$. An argument analogous to that of \eqref{eq:u-str-conv} yields
\begin{equation*}
\left|\nabla u(x) - \nabla u(y)\right| \leq \left(1 + \tau L_F'\right) |x - y|,
\end{equation*}
showing that $\nabla u$ is Lipschitz. Hence, by Rademacher’s theorem, $\nabla u$ is differentiable almost everywhere. Moreover, by the strong convexity of $u,$ the determinant of the Hessian $\nabla^2 u$ is positive almost everywhere. Combined with the injectivity of $\nabla u,$ this implies, by \cite[Lemma 5.5.3]{ambrosio2008gradient}, that $\nu \in \mathcal{P}_2^{\lambda}(\mathbb R^d)$. Since $\lambda \ll \rho$, we obtain $\nu \in \mathcal{P}_2^{\rho}(\mathbb R^d).$
\end{proof}
Now, since $\nu^{n+1} \in \mathcal{P}_2^{\rho}(\mathbb R^d)$ for a sufficiently small step-size $\tau,$ the existence and uniqueness of a minimizer for \eqref{eq:proximal-JKO} is a consequence of Theorem \ref{thm:well-posedness} 1. or 2.
\begin{proof}[Proof of Theorem \ref{thm:well-posedness} 3.]
Given $\mu^0 \in \mathcal{P}_2^{\rho}(\mathbb R^d),$ consider the map
\begin{equation}
\label{eq:map-to-minimize-proximal}
\mathcal{P}_2^{\rho}(\mathbb R^d) \ni \mu \mapsto \mathcal{H}(\mu) \coloneqq \sigma \operatorname{KL}(\mu|\rho) + \frac{1}{2\tau}\mathcal{W}_2^2(\mu, \nu^1).
\end{equation}
From Lemma \ref{lemma:pushforward-is-optimal} it follows that given $\mu^0 \in \mathcal{P}_2^{\rho}(\mathbb R^d),$ we obtain $\nu^{1} \in \mathcal{P}_2^{\rho}(\mathbb R^d).$ Hence, we can invoke Theorem \ref{thm:well-posedness} (either 1. or 2.) with $F=0$ to obtain the existence of a unique minimizer $\mu^1 \in \mathcal{P}_2^{\rho}(\mathbb R^d)$ for $\mathcal{H}.$

Therefore, using Lemma \ref{lemma:pushforward-is-optimal}, it follows inductively that, for each $n \in \mathbb N,$ given $\mu^n \in \mathcal{P}_2^{\rho}(\mathbb R^d),$ we obtain $\nu^{n+1} \in \mathcal{P}_2^{\rho}(\mathbb R^d),$ and hence $\mu^{n+1} \in \mathcal{P}_2^{\rho}(\mathbb R^d).$ Therefore, if $\mu^0 \in \mathcal{P}_2^{\rho}(\mathbb R^d),$ then $\left(\nu^n, \mu^n\right)_{n \in \mathbb N} \subset \mathcal{P}_2^{\rho}(\mathbb R^d) \times \mathcal{P}_2^{\rho}(\mathbb R^d)$ along the scheme \eqref{eq:proximal-JKO}.
\end{proof}
Therefore, via Theorem \ref{thm:existence-optimal-coupling}, we obtain
\begin{corollary}[Existence of optimal transport maps along \eqref{eq:proximal-JKO}]
\label{corollary:existence-optimal-transport-proximal}
Let $F \in \mathcal{C}^1,$ Assumptions \ref{assumption:L-derivative}, \ref{ass:lip-intrinsic} hold and $\rho \in \mathcal{P}_2(\mathbb R^d)$ such that $\rho \sim \lambda.$ If $\tau < \frac{1}{L_F'},$ then given $\mu^0 \in \mathcal{P}_2^{\rho}(\mathbb R^d),$ there exist unique $\mu^n$-a.e. and $\nu^{n+1}$-a.e. optimal transport maps $T_{\mu^n}^{\nu^{n+1}}:\mathbb R^d \to \mathbb R^d$ and $T_{\nu^{n+1}}^{\mu^n}:\mathbb R^d \to \mathbb R^d$ from $\mu^n$ to $\nu^{n+1}$ and from $\nu^{n+1}$ to $\mu^n,$ respectively, given by $T_{\mu^n}^{\nu^{n+1}} = I_d-\tau \nabla_\mu F(\mu^n)(\cdot)$ and $T_{\nu^{n+1}}^{\mu^n} = \left(I_d-\tau\nabla_\mu F(\mu^n)(\cdot) \right)^{-1}.$ Moreover, there also exist unique $\nu^{n+1}$-a.e. and $\mu^{n+1}$-a.e. optimal transport maps $T_{\nu^{n+1}}^{\mu^{n+1}}:\mathbb R^d \to \mathbb R^d$ and $T_{\mu^{n+1}}^{\nu^{n+1}}:\mathbb R^d \to \mathbb R^d$ such that $T_{\nu^{n+1}}^{\mu^{n+1}} \circ T_{\mu^{n+1}}^{\nu^{n+1}} = I_d,$ $\mu^{n+1}$-a.e. and $T_{\mu^{n+1}}^{\nu^{n+1}} \circ T_{\nu^{n+1}}^{\mu^{n+1}} = I_d,$ $\nu^{n+1}$-a.e..
\end{corollary}
Next, we present the proof of Theorem \ref{thm:optim} 3., which is a particular case of Theorem \ref{thm:optim} 1. and 2.
\begin{proof}[Proof of Theorem \ref{thm:optim} 3.]
\emph{Step 1.} Since $\mu^0 \in \mathcal{P}_2^{\rho}(\mathbb R^d),$ Theorem \ref{thm:well-posedness} 3. guarantees the existence and uniqueness of a minimizer $\mu^1 \in \mathcal{P}_2^{\rho}(\mathbb R^d)$ for \eqref{eq:map-to-minimize-proximal}. Following either the proof of Theorem \ref{thm:optim} 1. or 2. will show that $\mu^1 \in \mathfrak C.$ 

Therefore, it follows inductively that, for each $n \in \mathbb N,$ given $\mu^n \in \mathcal{P}_2^{\rho}(\mathbb R^d),$ the unique minimizer $\mu^{n+1}$ of \eqref{eq:proximal-JKO} belongs to $\mathfrak C.$ Hence, if $\mu^0 \in \mathcal{P}_2^{\rho}(\mathbb R^d),$ then $\left(\mu^n\right)_{n \in \mathbb N} \subset \mathfrak C$ along the scheme \eqref{eq:proximal-JKO}.

\emph{Step 2.} The result follows from either Theorem \ref{thm:optim} 1. or 2. with $F=0.$
\end{proof}

\section{Additional lemmas}
In this section, we prove two results that are key to proving Theorem \ref{thm:exp-conv}. As in Euclidean geometry, where Lipschitz continuity of the gradients implies smoothness with respect to the squared Euclidean distance, an analogue implication holds in the Wasserstein space. In particular, as the following result shows, Assumption \ref{ass:lip-intrinsic} implies that $F$ is smooth with respect to $\mathcal{W}_2^2.$ 
\begin{lemma}[$L_F$-smoothness of $F$ relative to $\mathcal{W}_2^2$]
\label{ass:relative-smoothness}
Assume $F \in \mathcal{C}^1$ and Assumptions \ref{assumption:L-derivative} and \ref{ass:lip-intrinsic} hold. Then, for any $\mu', \mu \in \mathcal{P}_2(\mathbb R^d),$ it holds
\begin{equation*}
F(\mu') - F(\mu) - \left\langle \nabla_\mu F(\mu)(\cdot), P_{\mu}^{\mu'} - I_d\right\rangle_{L_{\mu}^2(\mathbb R^d)}\leq L_F'\left\|I_d - P_{\mu}^{\mu'}\right\|^2_{L_{\mu}^2(\mathbb R^d)},
\end{equation*}
where $P_{\mu}^{\mu'}:\mathbb R^d \to \mathbb R^d$ is a measurable map such that ${P_{\mu}^{\mu'}}_\#\mu$ is the pushforward measure of $\mu$ onto $\mu'.$ If $P_{\mu}^{\mu'}$ is in fact an optimal transport map from $\mu$ to $\mu',$ then $\left\|I_d - P_{\mu}^{\mu'}\right\|^2_{L_{\mu}^2(\mathbb R^d)} = \mathcal{W}_2^2(\mu', \mu).$
\end{lemma}
\begin{proof}
Let $\mu', \mu \in \mathcal{P}_2(\mathbb R^d).$ Then $\gamma = \left(I_d, P_{\mu}^{\mu'}\right)_{\#} \mu$ is a coupling between $\mu$ and $\mu'.$ For any $\varepsilon \in [0,1],$ set $\mu^{\varepsilon}= \mu + \varepsilon(\mu'- \mu).$ Then since $F \in \mathcal{C}^1,$ it follows by Remark \ref{rmk:FTC} that
\begin{multline*}
\begin{aligned}
&F(\mu') - F(\mu) - \left\langle \nabla_\mu F(\mu)(\cdot), P_{\mu}^{\mu'} - I_d\right\rangle_{L_{\mu}^2(\mathbb R^d)}\\ 
&=\int_0^1 \int_{\mathbb R^d} \frac{\delta F}{\delta \mu} (\mu^{\varepsilon}, x) (\mu'-\mu) (\mathrm{d}x)\mathrm{d}\varepsilon - \int_{\mathbb R^d} \nabla_\mu F(\mu)(x) \cdot \left(P_{\mu}^{\mu'}(x) - x \right) \mu(\mathrm{d}x)\\
& = \int_0^1 \int_{\mathbb R^d \times \mathbb R^d} \left(\frac{\delta F}{\delta \mu} (\mu^{\varepsilon}, y) - \frac{\delta F}{\delta \mu} (\mu^{\varepsilon}, x)\right)\gamma (\mathrm{d}x, \mathrm{d}y)\mathrm{d}\varepsilon\\ 
&- \int_{\mathbb R^d \times \mathbb R^d} \nabla_\mu F(\mu)(x) \cdot \left(y - x \right) \gamma (\mathrm{d}x, \mathrm{d}y).
\end{aligned}
\end{multline*}
By Assumption \ref{assumption:L-derivative}, we have
\begin{multline*}
\begin{aligned}
&\int_0^1 \int_{\mathbb R^d \times \mathbb R^d} \left(\frac{\delta F}{\delta \mu} (\mu^{\varepsilon}, y) - \frac{\delta F}{\delta \mu} (\mu^{\varepsilon}, x)\right)\gamma (\mathrm{d}x, \mathrm{d}y)\mathrm{d}\varepsilon\\ &= \int_0^1 \int_{\mathbb R^d \times \mathbb R^d} \int_0^1 \nabla_\mu F(\mu^{\varepsilon})(x+\eta(y-x)) \cdot \left(y-x\right)\mathrm{d}\eta\gamma (\mathrm{d}x, \mathrm{d}y)\mathrm{d}\varepsilon.
\end{aligned}
\end{multline*}
By Assumption \ref{ass:lip-intrinsic}, the Cauchy-Schwarz inequality and convexity of $\mathcal{W}_2$, we obtain
\begin{multline*}
\begin{aligned}
&F(\mu') - F(\mu) - \left\langle \nabla_\mu F(\mu)(\cdot), P_{\mu}^{\mu'} - I_d\right\rangle_{L_{\mu}^2(\mathbb R^d)}\\
&= \int_0^1 \int_{\mathbb R^d \times \mathbb R^d} \int_0^1 \left( \nabla_\mu F(\mu^{\varepsilon})(x+\eta(y-x)) - \nabla_\mu F(\mu)(x)\right)\cdot \left(y-x\right)\mathrm{d}\eta\gamma (\mathrm{d}x, \mathrm{d}y)\mathrm{d}\varepsilon\\
&\leq L_F' \int_0^1 \int_{\mathbb R^d \times \mathbb R^d} \int_0^1 \left(\eta|y-x|+\mathcal{W}_2(\mu^{\varepsilon}, \mu)\right)\left|y-x\right|\mathrm{d}\eta\gamma (\mathrm{d}x, \mathrm{d}y)\mathrm{d}\varepsilon\\
&\leq L_F' \int_0^1 \int_{\mathbb R^d \times \mathbb R^d} \int_0^1 \left(\eta|y-x|+\varepsilon\mathcal{W}_2(\mu', \mu)\right)\left|y-x\right|\mathrm{d}\eta\gamma (\mathrm{d}x, \mathrm{d}y)\mathrm{d}\varepsilon\\
&= \frac{L_F'}{2} \int_{\mathbb R^d \times \mathbb R^d} |y-x|^2\gamma (\mathrm{d}x, \mathrm{d}y) + \frac{L_F'}{2}\mathcal{W}_2(\mu', \mu)\int_{\mathbb R^d \times \mathbb R^d}\left|y-x\right|\gamma (\mathrm{d}x, \mathrm{d}y)\\
&\leq L_F' \left\|I_d - P_{\mu}^{\mu'}\right\|^2_{L_{\mu}^2(\mathbb R^d)}.
\end{aligned}
\end{multline*}
\end{proof}
The following result is a consequence of the geodesic convexity of $\operatorname{KL}(\cdot|\rho)$ in the Wasserstein space and combined with Lemma \ref{ass:relative-smoothness}, it allows us to prove in Theorem \ref{thm:exp-conv} 2. that $\left(F^{\sigma}(\mu^n)\right)_n$ decreases along \eqref{eq:semi-implicit-JKO} as long as the step-size $\tau$ is small enough. It can also be viewed as a particular case of \cite[Lemma 4]{korbaproximal}. However, our proof is slightly different.
\begin{lemma}
\label{lemma:control-geod-KL}
Let Assumption \ref{ass:abs-cty-pi} hold. For any $\mu', \mu \in \mathfrak C,$ it holds 
\begin{equation*}
\operatorname{KL}(\mu'|\rho) - \operatorname{KL}(\mu|\rho) - \left\langle \nabla \log \frac{\mathrm{d}\mu'}{\mathrm{d}\rho}\left(T_{\mu}^{\mu'}\right), T_{\mu}^{\mu'} - I_d\right\rangle_{L_{\mu}^2(\mathbb R^d)} \leq 0.
\end{equation*}
\end{lemma}
\begin{proof}
Let $\mu', \mu \in \mathfrak C.$ Then since $\mathfrak C \subset \mathcal{P}_2^{\rho}(\mathbb R^d)$ and $\rho \in \mathcal{P}_2^{\lambda}(\mathbb R^d),$ it follows by Theorem \ref{thm:existence-optimal-coupling} that there exist unique $\mu$-a.e. and $\mu'$-a.e. optimal transport maps $T_{\mu}^{\mu'}:\mathbb R^d \to \mathbb R^d$ from $\mu$ to $\mu'$ and $T_{\mu'}^{\mu}:\mathbb R^d \to \mathbb R^d$ from $\mu'$ to $\mu$ such that $T_{\mu'}^{\mu} \circ T_{\mu}^{\mu'} = I_d,$ $\mu$-a.e and $T_{\mu}^{\mu'} \circ T_{\mu'}^{\mu} = I_d,$ $\mu'$-a.e..

Now we show that, for any $\varepsilon \in (0,1),$ $I_d+\varepsilon\left(T_{\mu'}^{\mu}-I_d\right)$ is the unique optimal transport map from $\mu'$ to $\left(I_d+\varepsilon\left(T_{\mu'}^{\mu}-I_d\right)\right)_{\#}\mu'.$ First, we observe that
\begin{equation*}
\int_{\mathbb R^d} \left|x+\varepsilon\left(T_{\mu'}^{\mu}(x)-x\right)\right|^2\mu'(\mathrm{d}x) \leq 2\int_{\mathbb R^d} |x|^2 \mu'(\mathrm{d}x) + 2\varepsilon^2\mathcal{W}_2^2(\mu', \mu) < \infty.
\end{equation*}
Hence, $\left(I_d+\varepsilon\left(T_{\mu'}^{\mu}-I_d\right)\right)_{\#}\mu' \in \mathcal{P}_2(\mathbb R^d).$ Therefore, by Theorem \ref{thm:sufficient-cond-optimal-convexity}, it suffices to show that $I_d+\varepsilon\left(T_{\mu'}^{\mu}-I_d\right)$ is the gradient of a convex differentiable $\mu'$-a.e. function. By Theorem \ref{thm:existence-optimal-coupling}, we have that
$T_{\mu'}^{\mu}(x) = \nabla \varphi(x)$ $\mu'\text{-a.e.}$ for a convex function $\varphi: \mathbb R^d \to \mathbb R.$ Hence, for $\mu'\text{-a.e.}$ $x,$
\begin{equation*}
x+\varepsilon\left(T_{\mu'}^{\mu}(x)-x\right) = (1-\varepsilon)x+\varepsilon T_{\mu'}^{\mu}(x) = \nabla\left((1-\varepsilon)\frac{|x|^2}{2}+\varepsilon\varphi(x)\right),
\end{equation*}
where the map $\mathbb R^d \ni x \mapsto (1-\varepsilon)\frac{|x|^2}{2}+\varepsilon\varphi(x) \in \mathbb R$ is convex and $\mu'\text{-a.e.}$ differentiable for all $\varepsilon \in (0,1).$ The uniqueness of $I_d+\varepsilon\left(T_{\mu'}^{\mu}-I_d\right)$ follows from Theorem \ref{thm:existence-optimal-coupling} since $\mu' \in \mathcal{P}_2^{\lambda}(\mathbb R^d)$ and the fact that $\left(I_d+\varepsilon\left(T_{\mu'}^{\mu}-I_d\right)\right)_{\#}\mu' \in \mathcal{P}_2(\mathbb R^d).$

By Theorem \ref{thm:subdiff-entropy}, since $\mu' \in \mathfrak C,$
\begin{equation*}
\partial^- \operatorname{KL}(\mu'|\rho) = \left\{\nabla \log \frac{\mathrm{d}\mu'}{\mathrm{d}\rho}\right\}.
\end{equation*}
Hence, by \eqref{eq:subdifferential}, we have
\begin{multline}
\label{eq:subdiff-new-transport}
\begin{aligned}
\operatorname{KL}\left(\left(I_d+\varepsilon\left(T_{\mu'}^{\mu}-I_d\right)\right)_{\#}\mu'\Big| \rho\right) &\geq \operatorname{KL}\left(\mu'|\rho\right) + \varepsilon\int_{\mathbb R^d} \nabla \log \frac{\mathrm{d}\mu'}{\mathrm{d}\rho}(x)\cdot\left(  T_{\mu'}^{\mu}(x)-x\right) \mu'(\mathrm{d}x)\\ &+ o\left(\mathcal{W}_2\left(\mu', \left(I_d+\varepsilon\left(T_{\mu'}^{\mu}-I_d\right)\right)_{\#}\mu'\right)\right).
\end{aligned}
\end{multline}
Now, by Corollary \ref{corollary:wass-transport}, we have
\begin{equation}
\label{eq:small-oh}
\mathcal{W}_2\left(\mu', \left(I_d+\varepsilon\left(T_{\mu'}^{\mu}-I_d\right)\right)_{\#}\mu'\right) = \varepsilon\mathcal{W}_2(\mu', \mu).
\end{equation}
By Assumption \ref{ass:abs-cty-pi}, $\rho$ is log-concave, thus by \cite[Theorem 9.4.10]{ambrosio2008gradient} the relative entropy $\operatorname{KL}(\cdot|\rho)$ is geodesically convex, and hence
\begin{equation}
\label{eq:geod-new-transport}
\operatorname{KL}\left(\left(I_d+\varepsilon\left(T_{\mu'}^{\mu}-I_d\right)\right)_{\#}\mu'\Big| \rho\right) \leq (1-\varepsilon)\operatorname{KL}(\mu'|\rho) + \varepsilon\operatorname{KL}(\mu|\rho).
\end{equation}
Combining \eqref{eq:subdiff-new-transport}, \eqref{eq:small-oh} and \eqref{eq:geod-new-transport} and using \eqref{eq:pushforward} gives
\begin{multline*}
\begin{aligned}
\operatorname{KL}(\mu|\rho) - \operatorname{KL}(\mu'|\rho) &\geq \int_{\mathbb R^d}\nabla \log \frac{\mathrm{d}\mu'}{\mathrm{d}\rho}(x)\cdot  \left(T_{\mu'}^{\mu}(x)-x\right) \mu'(\mathrm{d}x) + \frac{o(\varepsilon)}{\varepsilon}\\
&= \int_{\mathbb R^d}\nabla \log \frac{\mathrm{d}\mu'}{\mathrm{d}\rho}(x)\cdot\left(  T_{\mu'}^{\mu}(x)-x\right) \left({T_{\mu}^{\mu'}}_{\#}\mu\right)(\mathrm{d}x) + \frac{o(\varepsilon)}{\varepsilon}\\
&= \int_{\mathbb R^d}\nabla \log \frac{\mathrm{d}\mu'}{\mathrm{d}\rho}\left(T_{\mu}^{\mu'}(x)\right)\cdot  \left(x-T_{\mu}^{\mu'}(x)\right) \mu(\mathrm{d}x) + \frac{o(\varepsilon)}{\varepsilon}.
\end{aligned}
\end{multline*}
Sending $\varepsilon \to 0$ and rearranging gives the conclusion.
\end{proof}

\appendix
\section{Auxiliary results}
\begin{remark}[Conditions under which $\rho$ satisfies the LSI]\label{remark:LSI_for_rho}
We recall below a few criteria under which $\rho(\mathrm{d}x) \propto e^{-U(x)}\mathrm{d}x \in \mathcal{P}_2(\mathbb R^d)$ satisfies LSI:
    \begin{itemize}
        \item Assume $U \in C^2(\mathbb R^d)$ with $\nabla^2 U \succeq -C_U$ for some $C_U \geq 0$ and there exists $\delta > 0$ such that $\int_{\mathbb R^d} e^{\left(\frac{C_U}{2}+\delta\right)|x|^2}\rho(\mathrm{d}x) < \infty.$ Hence, by \cite[Theorem 1.1]{wang2001}, $\rho$ satisfies LSI.
        \item Assume $U \in C^2(\mathbb R^d)$ with $\nabla^2 U \succeq -C_U$ for some $C_U \geq 0$ and there exist $c_U', R' > 0$ such that $x \cdot \nabla U(x) \geq c_U'|x|^2$ whenever $|x| > R'.$ Hence, by \cite[Corollary 2.1]{CattiauxGuillinWu2009}, $\rho$ satisfies LSI.
    \end{itemize}
    Note that if $U \in C^2(\mathbb R^d),$ then Assumption \ref{ass:abs-cty-pi} implies $\nabla^2 U \succeq 0,$ so the first part in both criteria is satisfied. Assume that $U$ grows at least quadratically at infinity, i.e., there exist $k_U, R > 0$ such that $U(x) \geq k_U |x|^2$ for all $|x| > R$. This condition ensures that $\rho \in \mathcal{P}_2(\mathbb R^d)$, the first criterion is satisfied with $0 < \delta < k_U$ and the second criterion is also met since the growth condition implies that there exist $c_U', R' > 0$ such that $x \cdot \nabla U(x) \geq c_U' |x|^2$ whenever $|x| > R'.$
\end{remark}
\begin{proposition}[Existence and uniqueness of the minimizer of \eqref{eq:mean-field-min-problem}]
\label{prop:minimizer}
Let $F$ be weakly lower semi-continuous, convex in the sense of \eqref{eq:convexity_classical} and let $\rho \in \mathcal{P}_2(\mathbb R^d).$ Then there exists a unique minimizer of $F^{\sigma}$ and it belongs to $\mathcal{P}_2^{\rho}(\mathbb R^d).$
\end{proposition}
\begin{proof}
    We follow the argument presented in \cite[Proposition 2.5]{10.1214/20-AIHP1140}. Note that there exists $ \bar{\mu}\in\mathcal{P}_2(\mathbb{R}^d) $ such that $F^\sigma(\bar{\mu}) <+\infty.$ Denote 
	\begin{equation*}
	    \mathcal{S}:=\left\{\nu\in\mathcal{P}_2(\mathbb{R}^d):\operatorname{KL}(\nu|\rho) \leq \sigma^{-1} F^\sigma(\bar{\mu}) - \sigma^{-1}\inf_{\nu'\in\mathcal{P}_2(\mathbb{R}^d)} F(\nu') \right\}.
	\end{equation*}
    Observe that $F^{\sigma}(\mu) = F(\mu) + \sigma \operatorname{KL}(\mu|\rho) \geq \inf_{\mu'\in\mathcal{P}_2(\mathbb{R}^d)} F(\mu') > -\infty,$ for all $\mu \in \mathcal{P}_2(\mathbb R^d).$ Let $(\mu_k)_{k \geq 0}$ be a minimizing sequence of $F^{\sigma}$ on $\mathcal{S},$ i.e., $$\lim_{k \to \infty} F^{\sigma}(\mu_k) = \inf_{\mu'\in\mathcal{S}} F^{\sigma}(\mu').$$
    Since $\sigma^{-1} F^\sigma(\bar{\mu}) - \sigma^{-1}\inf_{\nu'\in\mathcal{P}_2(\mathbb{R}^d)} F(\nu') < \infty,$ for all $\sigma > 0,$ the set $\mathcal{S}$ is weakly compact (see, e.g., \cite[Lemma 1.4.3 (c)]{pDUP97a}), being a sub-level set of the relative entropy $\operatorname{KL}(\cdot|\rho).$ Hence, there exists $\mu_{\sigma}^* \in \mathcal{S}$ such that, at least along a subsequence of $(\mu_k)_{k \geq 0},$ we have $\mu_k \to \mu_{\sigma}^*$ weakly. 
    
    Since both $F$ and $\operatorname{KL}(\cdot|\rho)$ are weakly lower semi-continuous, it follows that
    \begin{equation*}
        F^{\sigma}(\mu_{\sigma}^*) \leq \liminf_{k \to \infty} F^{\sigma}(\mu_k) = \inf_{\mu'\in\mathcal{S}} F^{\sigma}(\mu'). 
    \end{equation*}
    Thus, $\mu_{\sigma}^*$ is a minimizer of $F^{\sigma}$ on $\mathcal{S}.$ By convexity of $F$ and strict convexity of $\operatorname{KL}(\cdot|\rho)$ on $\mathcal{S}$ (see, e.g., \cite[Lemma 1.4.3 (b)]{pDUP97a}), the function $F^{\sigma}$ is strictly convex on $\mathcal{S}$, and hence $\mu_{\sigma}^*$ is unique.

    Furthermore, for all $\mu_0 \notin \mathcal{S}$, we have 
    \begin{equation*}
        \sigma\operatorname{KL}(\mu_0|\rho) +\inf_{\nu'\in\mathcal{P}_2(\mathbb{R}^d)} F(\nu') > F^{\sigma}(\bar{\mu}), 
    \end{equation*}
    and therefore,
    \begin{equation*}
        F^{\sigma}(\mu_0) \geq \sigma\operatorname{KL}(\mu_0|\rho) +\inf_{\nu'\in\mathcal{P}_2(\mathbb{R}^d)} F(\nu') > F^{\sigma}(\bar{\mu}) \geq \inf_{\nu'\in\mathcal{P}_2(\mathbb{R}^d)} F^{\sigma}(\nu'),
    \end{equation*}
    for all $\mu_0 \notin \mathcal{S}.$ 
    
    Suppose there exists $\bar{\mu}_{\sigma} \in \mathcal{P}_2(\mathbb R^d)$ such that $F^{\sigma}(\bar{\mu}_{\sigma}) = \inf_{\nu'\in\mathcal{P}_2(\mathbb{R}^d)} F^{\sigma}(\nu').$ Then necessarily $\bar{\mu}_{\sigma} \in \mathcal{S},$ because $F^{\sigma}(\mu_0) > \inf_{\nu'\in\mathcal{P}_2(\mathbb{R}^d)} F^{\sigma}(\nu'),$ for all $\mu_0 \notin \mathcal{S},$ and since $\mu_{\sigma}^*$ is the unique minimizer of $F^\sigma$ on $\mathcal{S},$ we have $\bar{\mu}_{\sigma} = \mu_{\sigma}^*.$ Hence, $\mu_{\sigma}^*$ is the unique minimizer of $F^{\sigma}$ on $\mathcal{P}_2(\mathbb{R}^d).$
    
    Finally, since $\mu_{\sigma}^*$ minimizes $F^\sigma,$ we have $\operatorname{KL}(\mu_{\sigma}^*|\rho)<+\infty,$ which implies $\mu_{\sigma}^* \in \mathcal{P}_2^{\rho}(\mathbb R^d).$ 
\end{proof}
The following proposition shows that any global minimizer of $F^{\sigma}$ satisfies a first-order condition. Conversely, if $F$ is convex (cf. Assumption \ref{ass:convexity}) and there exists a measure in $\mathcal{P}_2(\mathbb R^d)$ satisfying this condition, then that measure is the unique minimizer of $F^\sigma.$
\begin{proposition}[First-order condition]
\label{prop:first-order-cond}
Assume $F \in \mathcal{C}^1$ and $\rho \in \mathcal{P}_2(\mathbb{R}^d).$ If $\mu_{\sigma}^* \in \argmin_{\mu \in \mathcal{P}_2(\mathbb{R}^d)}F^{\sigma}(\mu),$ then $\mu_{\sigma}^* \sim \rho$ and 
	\begin{equation}
    \label{InvariantSet}
		\frac{\delta F}{\delta \mu}(\mu_{\sigma}^*,\cdot) + \sigma \log\frac{\mathrm{d}\mu_{\sigma}^*}{\mathrm{d}\rho} \text{ is constant, }\rho\text{-a.e.}
	\end{equation}
Conversely, suppose Assumption \ref{ass:convexity} holds and let $\mu_{\sigma}^* \in \mathcal{P}_2(\mathbb R^d)$ satisfy \eqref{InvariantSet}. Then $\mu_{\sigma}^* = \argmin_{\mu \in \mathcal{P}_2(\mathbb{R}^d)}F^{\sigma}(\mu).$
\end{proposition}
\begin{proof}
We follow the argument presented in \cite[Proposition 2.5]{10.1214/20-AIHP1140}.

	\textit{Step 1. Sufficient condition:} Let $\mu_\sigma^* \in \mathcal{P}_2(\mathbb R^d)$ satisfy \eqref{InvariantSet}. In particular, $\mu_\sigma^*$ is equivalent to $\rho.$ Let $\mu \in \mathcal{P}_2^{\rho}(\mathbb R^d)$ (otherwise $F^\sigma(\mu)=+\infty $). Define the Radon-Nikodym derivative
	\[ 
	f:=\frac{\mathrm{d}\mu}{\mathrm{d}\mu_\sigma^*}. 
	\]
	For $\varepsilon \in [0,1],$ set $\mu^\varepsilon := (1-\varepsilon)\mu_\sigma^* + \varepsilon \mu = (1+\varepsilon (f-1))\mu_\sigma^*.$ Let $h(z) = z \log z$ for all $z > 0$ and $h(0)=0.$ Note that $h(z) \geq z-1,$ for all $z > 0$. 
	By Assumption \ref{ass:convexity}, using \eqref{eq:convexity_via_flat_derivative} we have
	\[
	\frac{F(\mu^\varepsilon) - F(\mu_\sigma^*)}{\varepsilon} \geq \frac1\varepsilon \int_{\mathbb R^d} \frac{\delta F}{\delta \mu}(\mu_\sigma^*,x) (\mu^\varepsilon - \mu_\sigma^*)\,\mathrm{d}x
	= \int_{\mathbb R^d} \frac{\delta F}{\delta \mu}(\mu_\sigma^*,x) (f(x)-1)\mu_\sigma^*(\mathrm{d}x).
	\]
	Moreover,
	\[
	\begin{split}
		&\frac{\sigma}{\varepsilon}\left(\operatorname{KL}(\mu^\varepsilon|\rho) - \operatorname{KL}(\mu_\sigma^*|\rho)\right)\\ 
		& = \frac{\sigma}{\varepsilon}\int_{\mathbb R^d} \left(\frac{\mathrm{d}\mu^\varepsilon}{\mathrm{d}\rho}(x) \log \frac{\mathrm{d}\mu^\varepsilon}{\mathrm{d}\rho}(x) - \frac{\mathrm{d}\mu_\sigma^*}{\mathrm{d}\rho}(x)\log \frac{\mathrm{d}\mu_\sigma^*}{\mathrm{d}\rho}(x) \right)\rho(\mathrm{d}x)\\
		& = \frac{\sigma}{\varepsilon}\int_{\mathbb R^d} \left(\frac{\mathrm{d}\mu^\varepsilon}{\mathrm{d}\rho}(x) - \frac{\mathrm{d}\mu_\sigma^*}{\mathrm{d}\rho}(x)\right)\log\frac{\mathrm{d}\mu_\sigma^*}{\mathrm{d}\rho}(x)\rho(\mathrm{d}x)\\ 
        &+ \frac{\sigma}{\varepsilon}\int_{\mathbb R^d} \frac{\mathrm{d}\mu^\varepsilon}{\mathrm{d}\rho}(x)\left(\log \frac{\mathrm{d}\mu^\varepsilon}{\mathrm{d}\rho}(x) - \log\frac{\mathrm{d}\mu_\sigma^*}{\mathrm{d}\rho}(x) \right)\rho(\mathrm{d}x)\\
		& = \sigma\int_{\mathbb R^d} (f(x)-1) \log\frac{\mathrm{d}\mu_\sigma^*}{\mathrm{d}\rho}(x)\mu_\sigma^*(\mathrm{d}x) + \frac{\sigma}{\varepsilon}\int_{\mathbb R^d} \frac{\mathrm{d}\mu^\varepsilon}{\mathrm{d}\mu_\sigma^*}(x) \log \frac{\mathrm{d}\mu^\varepsilon}{\mathrm{d}\mu_\sigma^*}\mu_\sigma^*(\mathrm{d}x)\\
		& = \sigma\int_{\mathbb R^d} (f(x)-1) \log\frac{\mathrm{d}\mu_\sigma^*}{\mathrm{d}\rho}(x)\mu_\sigma^*(\mathrm{d}x) + \frac{\sigma}{\varepsilon}\int_{\mathbb R^d} h(1+\varepsilon(f(x)-1))\mu_\sigma^*(\mathrm{d}x)\\
		& \geq \sigma\int_{\mathbb R^d} (f(x)-1) \log\frac{\mathrm{d}\mu_\sigma^*}{\mathrm{d}\rho}(x)\mu_\sigma^*(\mathrm{d}x) + \sigma\int_{\mathbb R^d} (f(x)-1)\mu_\sigma^*(\mathrm{d}x)\\
		& = \sigma\int_{\mathbb R^d} (f(x)-1) \log\frac{\mathrm{d}\mu_\sigma^*}{\mathrm{d}\rho}(x)\mu_\sigma^*(\mathrm{d}x),
	\end{split}
	\]
	since $\int_{\mathbb R^d} (f(x)-1)\mu_\sigma^*(\mathrm{d}x) = \int_{\mathbb R^d} (\mu-\mu_\sigma^*)(\mathrm{d}x) = 0.$
	Hence, 
	\[
	\frac{F^\sigma(\mu^\varepsilon)-F^\sigma(\mu_\sigma^*)}{\varepsilon} 
	\geq \int_{\mathbb R^d} \left(\frac{\delta F}{\delta \mu}(\mu_\sigma^*,x) + \sigma \log\frac{\mathrm{d}\mu_\sigma^*}{\mathrm{d}\rho}(x)\right) (f(x)-1)\mu_\sigma^*(\mathrm{d}x) = 0.
	\]
    Finally, by the strict convexity of $F^{\sigma},$ we obtain
    \begin{equation*}
        0 \leq \frac{F^\sigma(\mu^\varepsilon)-F^\sigma(\mu_\sigma^*)}{\varepsilon} < F^\sigma(\mu)-F^\sigma(\mu_\sigma^*),
    \end{equation*}
    for all $\mu \in \mathcal{P}_2(\mathbb R^d),$ which shows that $\mu_\sigma^*$ is the unique global minimizer of $F^\sigma.$

	\textit{Step 2. Necessary condition:} Let $\mu_{\sigma}^* \in \argmin_{\mu \in \mathcal{P}_2(\mathbb{R}^d)}F^{\sigma}(\mu).$ Then $\operatorname{KL}(\mu_{\sigma}^*|\rho)<+\infty,$ which implies that $\mu_{\sigma}^* \in \mathcal{P}_2^{\rho}(\mathbb R^d).$ For any $\mu \in \mathcal{P}_2^{\rho}(\mathbb R^d),$ we have
    \begin{align*}
	&0 \leq \frac{F^\sigma(\mu^\varepsilon)-F^\sigma(\mu_{\sigma}^*)}{\varepsilon}\\
    &= \frac{F(\mu^\varepsilon)-F(\mu_{\sigma}^*)}{\varepsilon} + \int_{\mathbb R^d} \frac{\sigma}{\varepsilon}\left(h\left(\frac{\mathrm{d}\mu^\varepsilon}{\mathrm{d}\rho}(x)\right) - h\left(\frac{\mathrm{d}\mu_{\sigma}^*}{\mathrm{d}\rho}(x)\right)\right)\rho(\mathrm{d}x)\\
    &=\int_{0}^{1}\int_{\mathbb{R}^{d}}\frac{\delta F}{\delta \mu}(\mu^{s\varepsilon},x)(\mu-\mu_\sigma^*)(\mathrm{d}x)\mathrm{d}s + \int_{\mathbb R^d} \frac{\sigma}{\varepsilon}\left(h\left(\frac{\mathrm{d}\mu^\varepsilon}{\mathrm{d}\rho}(x)\right) - h\left(\frac{\mathrm{d}\mu_{\sigma}^*}{\mathrm{d}\rho}(x)\right)\right)\rho(\mathrm{d}x).
    \end{align*}
    Since $\sup_{\beta \in [0,1]} \left|\frac{\delta F}{\delta \mu}(\mu^{\beta},x)\right| \leq C_F(1+|x|^2)$ for some $C_F > 0,$ and by continuity of $\frac{\delta F}{\delta \mu}(\cdot,x),$ the dominated convergence theorem yields
    \begin{equation*}
        \lim_{\varepsilon \searrow 0} \int_{0}^{1}\int_{\mathbb{R}^{d}}\frac{\delta F}{\delta \mu}(\mu^{s\varepsilon},x)(\mu-\mu_\sigma^*)(\mathrm{d}x)\mathrm{d}s = \int_{\mathbb{R}^{d}}\frac{\delta F}{\delta \mu}(\mu_\sigma^*,x)(\mu-\mu_\sigma^*)(\mathrm{d}x).
    \end{equation*}
    Since $h$ is convex, we have
    \begin{equation}
    \label{eq:rev-fatou-dominate}
        \frac{h\left(\frac{\mathrm{d}\mu^\varepsilon}{\mathrm{d}\rho}\right) - h\left(\frac{\mathrm{d}\mu_{\sigma}^*}{\mathrm{d}\rho}\right)}{\varepsilon} \leq h\left(\frac{\mathrm{d}\mu}{\mathrm{d}\rho}\right) - h\left(\frac{\mathrm{d}\mu_{\sigma}^*}{\mathrm{d}\rho}\right) \leq \left|h\left(\frac{\mathrm{d}\mu}{\mathrm{d}\rho}\right)\right| + \left|h\left(\frac{\mathrm{d}\mu_{\sigma}^*}{\mathrm{d}\rho}\right)\right|.
    \end{equation}
    Let $x_+ = \max(x,0)$ and $x_- = \min(x,0),$ so that $x = x_+ + x_-$ and $|x| = x_+ - x_-.$ There exists $c > 0$ such that $h(z)_- \in [-c,0]$ for all $z \geq 0.$ Hence,
    \begin{align*}
        +\infty > \operatorname{KL}(\mu|\rho) &= \int_{\mathbb R^d} h\left(\frac{\mathrm{d}\mu}{\mathrm{d}\rho}(x)\right)\rho(\mathrm{d}x)\\
        &= \int_{\mathbb R^d} h\left(\frac{\mathrm{d}\mu}{\mathrm{d}\rho}(x)\right)_+\rho(\mathrm{d}x) + \int_{\mathbb R^d} h\left(\frac{\mathrm{d}\mu}{\mathrm{d}\rho}(x)\right)_-\rho(\mathrm{d}x)
    \end{align*}
    and therefore
    \begin{equation*}
        \int_{\mathbb R^d} h\left(\frac{\mathrm{d}\mu}{\mathrm{d}\rho}(x)\right)_+\rho(\mathrm{d}x) < +\infty.
    \end{equation*}
    It follows that
    \begin{align*}
        \int_{\mathbb R^d} \left|h\left(\frac{\mathrm{d}\mu}{\mathrm{d}\rho}(x)\right)\right|\rho(\mathrm{d}x)= \int_{\mathbb R^d} h\left(\frac{\mathrm{d}\mu}{\mathrm{d}\rho}(x)\right)_+\rho(\mathrm{d}x) - \int_{\mathbb R^d} h\left(\frac{\mathrm{d}\mu}{\mathrm{d}\rho}(x)\right)_-\rho(\mathrm{d}x) < +\infty.
    \end{align*}
    Similarly,
    \begin{equation*}
        \int_{\mathbb R^d} \left|h\left(\frac{\mathrm{d}\mu_{\sigma}^*}{\mathrm{d}\rho}(x)\right)\right|\rho(\mathrm{d}x) < +\infty.
    \end{equation*}
    Thus, the function $x \mapsto \left|h\left(\frac{\mathrm{d}\mu}{\mathrm{d}\rho}(x)\right)\right| + \left|h\left(\frac{\mathrm{d}\mu_{\sigma}^*}{\mathrm{d}\rho}(x)\right)\right|$ is non-negative and $\rho$-integrable. By combining this with \eqref{eq:rev-fatou-dominate} and applying the reversed Fatou lemma, we obtain
    \begin{multline}
    \begin{aligned}
    \label{derivative h}
        &\limsup_{\varepsilon \searrow 0} \int_{\mathbb R^d} \frac{h\left(\frac{\mathrm{d}\mu^\varepsilon}{\mathrm{d}\rho}(x)\right) - h\left(\frac{\mathrm{d}\mu_{\sigma}^*}{\mathrm{d}\rho}(x)\right)}{\varepsilon}\rho(\mathrm{d}x)\\ 
        &\leq \int_{\mathbb R^d} \limsup_{\varepsilon \searrow 0} \frac{h\left(\frac{\mathrm{d}\mu^\varepsilon}{\mathrm{d}\rho}(x)\right) - h\left(\frac{\mathrm{d}\mu_{\sigma}^*}{\mathrm{d}\rho}(x)\right)}{\varepsilon}\rho(\mathrm{d}x).
    \end{aligned}
    \end{multline}
    Since $h$ is not differentiable at $z = 0,$ we will first show that $\frac{\mathrm{d}\mu_{\sigma}^*}{\mathrm{d}\rho} > 0$ $\rho$-a.e., before passing to the limit in the expression above.

    Let $g \coloneqq \frac{\mathrm{d}\mu_{\sigma}^*}{\mathrm{d}\rho}$ and define \begin{equation*}
        A \coloneqq \{x \in \mathbb R^d: g(x) > 0\}, \quad B \coloneqq \{x \in \mathbb R^d: g(x) = 0\}.
    \end{equation*}
    Consider the convex combination $\mu^\varepsilon = (1-\varepsilon)\mu_\sigma^* + \varepsilon \rho,$ for $\varepsilon \in (0,1),$ so that $g^\varepsilon = (1-\varepsilon)g + \varepsilon.$
    For $x \in A,$ we have $g(x)-g^\varepsilon(x) =\varepsilon(g(x)-1).$ Since $h$ is convex and differentiable on $A,$ it follows that
    \begin{multline}
    \label{A}
    \begin{aligned}
        h(g(x)) - h(g^\varepsilon(x))
        &\geq h'\left(g^\varepsilon(x)\right)\left(g(x)-g^\varepsilon(x)\right)\\
        &= \varepsilon\left(1+\log g^\varepsilon(x)\right)\left(g(x) - 1\right)\\
        &\geq \varepsilon\left(g(x) - 1\right).
    \end{aligned}
    \end{multline}
    The last inequality holds by considering separately the cases $g \geq 1$ and $g \in (0,1).$ If $g\geq1,$ then $$1+\log g^\varepsilon(x) = 1+\log \left((1-\varepsilon)g(x)+\varepsilon\right) \geq 1,$$ hence $$\left(1+\log g^\varepsilon(x)\right)\left(g(x) - 1\right) \geq g(x)-1.$$ If instead $g\in (0,1),$ then $$1+\log g^\varepsilon(x) = 1+\log \left((1-\varepsilon)g(x)+\varepsilon\right) \leq 1,$$ so again $$\left(1+\log g^\varepsilon(x)\right)\left(g(x) - 1\right) \geq g(x) - 1.$$
    For $x \in B,$ since $h(0) = 0$ and $g=0,$ we have $g^\varepsilon(x) = \varepsilon,$ and therefore
    \begin{equation}
    \label{B}
        h(g(x)) - h(g^\varepsilon(x)) = -\varepsilon \log \varepsilon.
    \end{equation}
    Now observe that
    \begin{align*}
        \operatorname{KL}(\mu_\sigma^*|\rho) - \operatorname{KL}(\mu^\varepsilon|\rho) &= \int_{\mathbb R^d}\left(h(g(x)) - h(g^\varepsilon(x))\right)\rho(\mathrm{d}x)\\
        &= \int_{A}\left(h(g(x)) - h(g^\varepsilon(x))\right)\rho(\mathrm{d}x) + \int_{B}\left(h(g(x)) - h(g^\varepsilon(x))\right)\rho(\mathrm{d}x).
    \end{align*}
    We apply \eqref{A} to the first term and \eqref{B} to the second, and after rearranging obtain
    \begin{align*}
        & \sigma\varepsilon\int_A \left(g(x) - 1\right)\rho(\mathrm{d}x) -\sigma \rho(B) \varepsilon \log \varepsilon \\
        &\leq\sigma \operatorname{KL}(\mu_\sigma^*|\rho) - \sigma \operatorname{KL}(\mu^\varepsilon|\rho)\\
        &\leq F(\mu^\varepsilon) - F(\mu^*_\sigma)\\
        &=\varepsilon\int_0^1 \int \frac{\delta F}{\delta \mu}(\mu^*_\sigma + \beta\varepsilon(\rho - \mu^*_\sigma), x)(\rho - \mu^*_\sigma)(\mathrm{d}x)\mathrm{d}\beta\\
        &\leq \varepsilon\int (1+|x|^2)(\rho + \mu^*_\sigma)(\mathrm{d}x) < +\infty,
    \end{align*}
     where the second inequality follows from $\mu^*_\sigma \in \argmin_{\mu} F^\sigma(\mu).$
     
    Dividing both sides by $\varepsilon$ and letting $\varepsilon \searrow 0$ leads to a contradiction unless $\rho(B) = 0.$ Thus $g > 0$ $\rho$-a.e., which implies that for any $E\in \mathcal{B}(\mathbb R^d)$ with $\rho(E)>0,$ we must have $\mu_\sigma^*(E) >0.$ This is equivalent to $\rho \ll \mu_\sigma^*,$ and since we already know that $\mu_\sigma^* \ll \rho,$ it follows that $\mu_\sigma^* \sim \rho.$

    We now proceed to establish the optimality condition for the minimizer. Using $h'(z) = 1 + \log z$ for $z > 0$ in \eqref{derivative h} gives
    \begin{align*}
        &\int_{\mathbb R^d }\limsup_{\varepsilon \searrow 0} \frac{h\left(\frac{\mathrm{d}\mu^\varepsilon}{\mathrm{d}\rho}(x)\right) - h\left(\frac{\mathrm{d}\mu_{\sigma}^*}{\mathrm{d}\rho}(x)\right)}{\varepsilon}\rho(\mathrm{d}x)\\ 
        &= \int_{\mathbb R^d} \left(1+\log \frac{\mathrm{d}\mu_{\sigma}^*}{\mathrm{d}\rho}(x)\right)\left(\frac{\mathrm{d}\mu}{\mathrm{d}\rho}(x) - \frac{\mathrm{d}\mu_{\sigma}^*}{\mathrm{d}\rho}(x)\right)\rho(\mathrm{d}x)\\
        &=\int_{\mathbb R^d} \log \frac{\mathrm{d}\mu_{\sigma}^*}{\mathrm{d}\rho}(x)(\mu-\mu_\sigma^*)(\mathrm{d}x),
    \end{align*}
    since $\int_{\mathbb R^d} (\mu-\mu_\sigma^*)(\mathrm{d}x)=0.$
    
    Therefore, 
    \begin{equation}
    \label{eq:variational-ineq}
        0 \leq \limsup_{\varepsilon \searrow 0} \frac{F^\sigma(\mu^\varepsilon)-F^\sigma(\mu_{\sigma}^*)}{\varepsilon} \leq \int_{\mathbb{R}^{d}}\left(\frac{\delta F}{\delta \mu}(\mu_\sigma^*,x) + \sigma \log \frac{\mathrm{d}\mu_{\sigma}^*}{\mathrm{d}\rho}(x)\right)(\mu-\mu_\sigma^*)(\mathrm{d}x),
    \end{equation}
    for all $\mu \in \mathcal{P}_2^{\rho}(\mathbb R^d).$ Hence, by \cite[Lemma 33]{jabir2021meanfieldneuralodesrelaxed},
    \begin{equation*}
        \frac{\delta F}{\delta \mu}(\mu_{\sigma}^*,\cdot) + \sigma \log\frac{\mathrm{d}\mu_{\sigma}^*}{\mathrm{d}\rho} \text{ is constant, }\mu_{\sigma}^*\text{-a.e.}
    \end{equation*}
    Since $\mu_\sigma^* \sim \rho,$ it follows that 
    \begin{equation*}
        \frac{\delta F}{\delta \mu}(\mu_{\sigma}^*,\cdot) + \sigma \log\frac{\mathrm{d}\mu_{\sigma}^*}{\mathrm{d}\rho} \text{ is constant, }\rho\text{-a.e.}
    \end{equation*}
\end{proof}
We also highlight that, under Assumption \ref{assumption:unif-lsi}, $\mu_\sigma^*$ satisfies both the LSI \eqref{eq:sobolev-ineq} and the Talagrand inequality \eqref{eq:Talagrand-ineq} since $\mu_\sigma^* = \Phi_{\sigma}[\mu_\sigma^*].$
Using Assumption \ref{ass:convexity}, \cite{Nitanda2022ConvexAO,chizat2022meanfield} proved the following entropy ``sandwich'' lemma, which provides bounds for the distance between $F^{\sigma}$ and the minimum value $F^{\sigma}(\mu_{\sigma}^*)$. We include the statement and its proof, as the version we require is formulated slightly differently, using the relative entropy $\operatorname{KL}(\cdot|\rho)$ instead of the entropy $H.$
\begin{lemma}[\protect{\cite[Proposition 1]{Nitanda2022ConvexAO}, \cite[Lemma 3.4]{chizat2022meanfield}}]
\label{lemma:sandwich}
Assume $F \in \mathcal{C}^1$ and convex in the sense of \eqref{eq:convexity_classical} (i.e., Assumption \ref{ass:convexity} holds), and let $\rho \in \mathcal{P}_2(\mathbb R^d).$ Let $\mu_{\sigma}^* \in \mathcal{P}_2^{\rho}(\mathbb R^d)$ be the unique minimizer of $F^{\sigma}.$ Then, for any $\mu \in \mathcal{P}_2(\mathbb R^d),$ we have
\begin{equation*}
\sigma \operatorname{KL}(\mu|\mu_{\sigma}^*) \leq  F^{\sigma}(\mu) - F^{\sigma}(\mu_{\sigma}^*) \leq \sigma \operatorname{KL}(\mu|\Phi_{\sigma}[\mu]).
\end{equation*}
\end{lemma}
\begin{proof}
    Since $F \in \mathcal{C}^1$ and is convex in the sense of \eqref{eq:convexity_classical}, it follows from \cite[Lemma 4.1]{10.1214/20-AIHP1140} that $F$ satisfies \eqref{eq:convexity_via_flat_derivative}. By Remark \ref{remark:weak_lower_semi_continuity}, $F$ is weakly lower semi-continuous. Hence, by Proposition \ref{prop:minimizer}, the measure $\mu_\sigma^* \in \mathcal{P}_2^{\rho}(\mathbb R^d)$ is the unique minimizer of $F^\sigma.$ Furthermore, by Proposition \ref{prop:first-order-cond}, $\mu_\sigma^*$ is equivalent to $\rho$ and satisfies \eqref{InvariantSet}, that is, 
    \begin{equation*}
    \mu_{\sigma}^*(\mathrm{d}x) = \frac{1}{Z( \mu_{\sigma}^*)}e^{-\frac{1}{\sigma}  \frac{\delta F}{\delta \mu}(\mu_{\sigma}^*, x)}\rho(\mathrm{d}x),
    \end{equation*}
    where $Z(\mu_{\sigma}^*)$ is the normalization constant. Let $\mu \in \mathcal{P}_2(\mathbb R^d).$ Then, by \eqref{eq:convexity_via_flat_derivative},
    \begin{align*}
        F^{\sigma}(\mu) - F^\sigma(\mu_{\sigma}^*) &\geq \int_{\mathbb R^d} \frac{\delta F}{\delta \mu}(\mu_{\sigma}^*, x)(\mu-\mu_{\sigma}^*)(\mathrm{d}x) +\sigma \operatorname{KL}(\mu|\rho) - \sigma \operatorname{KL}(\mu_{\sigma}^*|\rho)\\
        &=\sigma \operatorname{KL}(\mu|\mu_{\sigma}^*) - \sigma \operatorname{KL}(\mu_{\sigma}^*|\mu_{\sigma}^*) = \sigma \operatorname{KL}(\mu|\mu_{\sigma}^*).
    \end{align*}
    Similarly, by \eqref{eq:convexity_via_flat_derivative},
    \begin{align*}
    F^{\sigma}(\mu) - F^\sigma(\mu_{\sigma}^*) &\leq \int_{\mathbb R^d} \frac{\delta F}{\delta \mu}(\mu, x)(\mu - \mu_{\sigma}^*)(\mathrm{d}x) +\sigma \operatorname{KL}(\mu|\rho) - \sigma \operatorname{KL}(\mu_{\sigma}^*|\rho)\\
    &=\sigma \operatorname{KL}(\mu|\Phi_{\sigma}[\mu]) - \sigma \operatorname{KL}(\mu_{\sigma}^*|\Phi_{\sigma}[\mu]) \leq \sigma \operatorname{KL}(\mu|\Phi_{\sigma}[\mu]).
    \end{align*}
\end{proof}

\section{Optimal Transport}
\label{app:OT}
In this appendix, we recall the fundamental results from optimal transport that are used throughout the paper.
\begin{definition}[Pushforward of a measure by a map]
Let $T:\mathbb R^d \to \mathbb R^d$ be a $\mathcal{B}(\mathbb R^d)$-measurable map. Then, for every $\mu \in \mathcal{P}_2(\mathbb R^d),$ we denote by $T_{\#}\mu \in \mathcal{P}_2(\mathbb R^d)$ the pushforward measure of $\mu$ by $T,$ characterized by
\begin{equation}
\label{eq:pushforward}
\int_{\mathbb R^d} f(T(x))\mu(\mathrm{d}x) = \int_{\mathbb R^d} f(y)\left(T_{\#}\mu\right)(\mathrm{d}y), \text{ for any measurable bounded function }f.
\end{equation}
\end{definition}
Consider the $2$-Wasserstein distance $\mathcal{W}_2:\mathcal{P}_2(\mathbb R^d) \times \mathcal{P}_2(\mathbb R^d) \to [0, \infty),$ defined by
\begin{equation}
\label{eq:wass}
\mathcal{W}_2(\mu,\nu) \coloneqq \left(\inf_{\gamma \in \Gamma(\mu,\nu)} \int_{\mathbb R^d \times \mathbb R^d} |x-y|^2\gamma(\mathrm{d}x,\mathrm{d}y)\right)^{\frac{1}{2}},
\end{equation}
where $\Gamma(\mu,\nu) = \left\{\gamma \in \mathcal{P}_2(\mathbb R^d \times \mathbb R^d): (P_x)_{\#} \gamma = \mu, (P_y)_{\#} \gamma = \nu\right\}$ is the set of couplings between $\mu$ and $\nu,$ where $P_x : (x,y) \mapsto x$ and $P_y : (x,y) \mapsto y$ are the projections onto the first and second component, respectively. The set of optimal couplings for which the infimum is attained in \eqref{eq:wass} is denoted by $\Gamma_o(\mu, \nu) \coloneqq \left\{\Bar{\gamma} \in \Gamma(\mu, \nu): \mathcal{W}_2^2(\mu,\nu) = \int_{\mathbb R^d \times \mathbb R^d} |x-y|^2\Bar{\gamma}(\mathrm{d}x,\mathrm{d}y)\right\}.$

Now we recall a standard result from optimal transport (see e.g. \cite[Theorem 4.5]{gangbo-mccann} and also \cite{ambrosio2008gradient,santambrogio}), which shall be essential throughout the paper.
\begin{theorem}
\label{thm:existence-optimal-coupling}
Let $\mu \in \mathcal{P}_2^{\lambda}(\mathbb R^d)$ and $\nu \in \mathcal{P}_2(\mathbb R^d).$ Then
\begin{enumerate}
\item there exists a unique optimal coupling $\gamma^* \coloneqq \left(I_d, T_{\mu}^{\nu}\right)_{\#} \mu$ which minimizes \eqref{eq:wass}, where $T_{\mu}^{\nu}:\mathbb R^d \to \mathbb R^d$ is the unique $\mu$-almost everywhere (a.e.) optimal transport map from $\mu$ to $\nu.$
\item Moreover, $T_{\mu}^{\nu}(x) = \nabla \varphi(x)$ $\mu\text{-a.e.}$ for a convex function $\varphi: \mathbb R^d \to \mathbb R$ with $\varphi(x) \coloneqq \frac{1}{2}|x|^2-\psi(x),$ where $\psi$ is a c-concave function.\footnote{A function $\xi:\mathbb R^d \to \mathbb R$ is c-concave if and only if $x \mapsto \frac{1}{2}|x|^2 - \xi(x)$ is convex and lower semi-continuous.}
\item If $\nu \in \mathcal{P}_2^{\lambda}(\mathbb R^d),$ then $\gamma^* = \left(T_{\nu}^{\mu}, I_d\right)_{\#} \nu,$ where $T_{\nu}^{\mu}:\mathbb R^d \to \mathbb R^d$ is the unique $\nu$-a.e. optimal transport map such that 
\begin{equation*}
T_{\nu}^{\mu} \circ T_{\mu}^{\nu} = I_d, \quad \mu\text{-a.e. and } T_{\mu}^{\nu} \circ T_{\nu}^{\mu} = I_d, \quad \nu\text{-a.e.}
\end{equation*}
\end{enumerate}  
\end{theorem}
As a consequence of Theorem \ref{thm:existence-optimal-coupling}, we have the following
\begin{corollary}
\label{corollary:wass-transport}
Let $\mu, \nu \in \mathcal{P}_2^{\lambda}(\mathbb R^d).$ Then
\begin{equation*}
\mathcal{W}_2^2(\mu,\nu) = \int_{\mathbb R^d} \left|x-T_{\mu}^{\nu}(x)\right|^2\mu(\mathrm{d}x) = \int_{\mathbb R^d} \left|x-T_{\nu}^{\mu}(x)\right|^2\nu(\mathrm{d}x).
\end{equation*}
\end{corollary}
\begin{theorem}[\protect{\cite[Theorem 1.48]{santambrogio2015optimal}}]
\label{thm:sufficient-cond-optimal-convexity}
Suppose $\mu \in \mathcal{P}_2(\mathbb R^d)$ and that $u: \mathbb R^d \to \mathbb R$ is a convex and differentiable $\mu$-a.e. Set $T \coloneqq \nabla u$ and suppose $\int_{\mathbb R^d} |T(x)|^2 \mu(\mathrm{d}x) < \infty.$ Then $T$ is an optimal transport map from $\mu$ to $T_{\#}\mu.$
\end{theorem}
\begin{theorem}[Subdifferential of $\operatorname{KL}(\cdot|\rho);$ \protect{\cite[Theorem 10.4.9]{ambrosio2008gradient}}]
\label{thm:subdiff-entropy}
Suppose $\rho \in \mathcal{P}_2(\mathbb{R}^d)$ satisfies Assumption \ref{ass:abs-cty-pi}. The relative entropy $\operatorname{KL}(\cdot|\rho)$ has finite slope at $\mu \in \mathcal{P}^{\rho}_2(\mathbb R^d)$ (cf. Definition \ref{def:slope}), i.e., $|\mathfrak d \operatorname{KL}(\cdot|\rho)|(\mu) < \infty$ if and only if $\frac{\mathrm{d}\mu}{\mathrm{d}\rho} \in W^{1,1}_{\lambda, \text{loc}}(\mathbb R^d)$ and $\nabla \log \frac{\mathrm{d}\mu}{\mathrm{d}\rho} \in L^2_{\mu}(\mathbb R^d).$ In this case, $I(\mu|\rho) = |\mathfrak d \operatorname{KL}(\cdot|\rho)|^2(\mu)$ and $\partial^- \operatorname{KL}(\mu|\rho) = \left\{\nabla \log \frac{\mathrm{d}\mu}{\mathrm{d}\rho}\right\}$ (cf. \eqref{eq:subdifferential} in Definition \ref{def:wass-sub-and-sup}).
\end{theorem}
We first observe that the requirement
\begin{equation*}
    \mu \in \mathfrak C = \left\{m \in \mathcal{P}^{\rho}_2(\mathbb R^d): \frac{\mathrm{d}m}{\mathrm{d}\rho} \in W_{\lambda,\text{loc}}^{1,1}(\mathbb R^d), \sqrt \frac{\mathrm{d}m}{\mathrm{d}\rho} \in W_{\rho}^{1,2}(\mathbb R^d)\right\}
\end{equation*} 
is equivalent to the hypotheses of Theorem \ref{thm:subdiff-entropy}, namely $\frac{\mathrm{d}\mu}{\mathrm{d}\rho} \in W^{1,1}_{\lambda, \text{loc}}(\mathbb R^d)$ and $\nabla \log \frac{\mathrm{d}\mu}{\mathrm{d}\rho} \in L^2_{\mu}(\mathbb R^d).$ Indeed, for any $\mu \in \mathcal{P}^{\rho}_2(\mathbb R^d)$, the condition $\sqrt \frac{\mathrm{d}\mu}{\mathrm{d}\rho} \in W_{\rho}^{1,2}(\mathbb R^d)$ is equivalent to $\sqrt \frac{\mathrm{d}\mu}{\mathrm{d}\rho} \in L_{\rho}^2(\mathbb R^d)$ and $\nabla\sqrt \frac{\mathrm{d}\mu}{\mathrm{d}\rho} \in L_{\rho}^2(\mathbb R^d)$. The first property holds automatically since $\mu \in \mathcal{P}^{\rho}_2(\mathbb R^d)$. For the second, applying the chain rule yields
\begin{align*}
    \int_{\mathbb R^d} \left|\nabla\sqrt \frac{\mathrm{d}\mu}{\mathrm{d}\rho}(x)\right|^2\rho(\mathrm{d}x) &= \frac{1}{4}\int_{\mathbb R^d} \left|\frac{\nabla \frac{\mathrm{d}\mu}{\mathrm{d}\rho}(x)}{\sqrt \frac{\mathrm{d}\mu}{\mathrm{d}\rho}(x)}\right|^2\rho(\mathrm{d}x)\\ 
    &= \frac{1}{4}\int_{\mathbb R^d} \left|\frac{\nabla \frac{\mathrm{d}\mu}{\mathrm{d}\rho}(x)}{ \frac{\mathrm{d}\mu}{\mathrm{d}\rho}(x)}\right|^2\frac{\mathrm{d}\mu}{\mathrm{d}\rho}(x)\rho(\mathrm{d}x)\\ &= \frac{1}{4}\int_{\mathbb R^d} \left|\nabla \log\frac{\mathrm{d}\mu}{\mathrm{d}\rho}(x)\right|^2 \mu(\mathrm{d}x),
\end{align*}
which is precisely the condition $\nabla \log \frac{\mathrm{d}\mu}{\mathrm{d}\rho} \in L^2_{\mu}(\mathbb R^d).$ Note that the motivation behind the definition of the space $\mathfrak{C}$ is to characterize the required integrability conditions for $\mu$ without referring to the space $L^2_{\mu}(\mathbb R^d)$ which depends on $\mu$ itself.

Now let us briefly discuss how \cite[Theorem 10.4.9]{ambrosio2008gradient} is applied in our setting. Note that, using the notation from \cite{ambrosio2008gradient}, their Theorem 10.4.9 is formulated for a functional $\mathcal{F}(\mu|\gamma) := \int_{\mathbb{R}^d} F(\sigma) \mathrm{d}\gamma$ for $\mu=\sigma \cdot \gamma$ (which means that $\sigma = \frac{\mathrm{d}\mu}{\mathrm{d}\gamma}$), where $\gamma$ is log-concave, associated with a sufficiently regular energy density $F: \mathbb{R}_+ \to \mathbb{R}_+$. The result states that the functional $\mathcal{F}(\cdot|\gamma)$ has finite slope at $\mu=\sigma \cdot \gamma$, i.e., $|\mathfrak d\mathcal{F}|(\mu) < \infty,$ if and only if $L_F(\sigma) \in W_{\lambda,\text{loc}}^{1,1}(\mathbb R^d)$ and $\nabla L_F (\sigma) = \sigma \mathbf{w}$ for some function $\mathbf{w} \in L_\mu^q(\mathbb R^d)$, where $L_F(z) = z F'(z) - F(z)$ is defined in \cite[page 214, (9.3.13)]{ambrosio2008gradient}. In this case, 
    \begin{equation*}
        |\mathfrak d \mathcal{F}|(\mu) = \left(\int_{\mathbb R^d} |\mathbf{w}(x)|^q \mu(\mathrm{d}x)\right)^{1/q},
    \end{equation*}
    and $\mathbf{w}$ is the minimal selection in $\partial^- \mathcal{F}(\mu),$ i.e., $\mathbf{w} = \partial^\circ \mathcal{F}(\mu).$

    In our setting, we take $\mathcal{F}(\cdot|\gamma) = \operatorname{KL}(\cdot|\rho),$ with $\gamma = \rho \propto e^{-U}$, with convex $U,$ so that $\sigma  = \frac{\mathrm{d}\mu}{\mathrm{d}\rho}$, and $F(z)=z\log z$, so that $L_F(z) = z$. In Theorem \ref{thm:optim}, we prove that $|\mathfrak d \operatorname{KL}(\cdot|\rho)|(\mu) < \infty,$ which implies $L_F(\sigma) = \frac{\mathrm{d}\mu}{\mathrm{d}\rho} \in W_{\lambda,\text{loc}}^{1,1}(\mathbb R^d)$ and $\mathbf{w} = \nabla \log \frac{\mathrm{d}\mu}{\mathrm{d}\rho} \in L_\mu^2(\mathbb R^d).$ Since we work on $\mathcal{P}_2(\mathbb R^d),$ we have $q=2$ (the general case $q = \frac{p}{p-1}$ is given in \cite[Definition 10.3.1]{ambrosio2008gradient}). By the alternative characterization of the slope, i.e., $$|\mathfrak d \operatorname{KL}(\cdot|\rho)|(\mu) = \min\{\|\xi\|_{L_\mu^2(\mathbb R^d)}: \xi \in \partial^-\operatorname{KL}(\mu|\rho)\},$$ we obtain $|\mathfrak d \operatorname{KL}(\cdot|\rho)|(\mu)^2 = I(\mu|\rho)$ with the minimizer in $\partial^-\operatorname{KL}(\mu|\rho)$ given uniquely by $\nabla \log \frac{\mathrm{d}\mu}{\mathrm{d}\rho}.$ Thus, $\partial^\circ \operatorname{KL}(\mu|\rho)=\partial^-\operatorname{KL}(\mu|\rho) = \left\{\nabla \log \frac{\mathrm{d}\mu}{\mathrm{d}\rho} \right\}.$
\section{Differential calculus on $\mathcal{P}_2(\mathbb R^d)$}
In this appendix, we recall the notions of linear functional (flat) differentiability \cite{cardaliaguet2019master}, Wasserstein differentiability \cite{Carmona2018ProbabilisticTO} and slope of functions of measures \cite{ambrosio2008gradient} used throughout the paper.
\begin{definition}[Flat differentiability on $\mathcal{P}_2(\mathbb R^d)$]
\label{def:flat-derivative}
We say a function $F:\mathcal{P}_2(\mathbb R^d)\to \mathbb R$ is in $\mathcal{C}^1$, if there exists a continuous function $\frac{\delta F}{\delta \mu}:\mathcal{P}_2(\mathbb R^d)\times \mathbb R^d \to \mathbb{R},$ with respect to the product topology on $\mathcal{P}_2(\mathbb R^d) \times \mathbb R^d,$ called the flat derivative of $F,$ for which there exists $\kappa > 0$ such that for all $(\mu,x) \in \mathcal{P}_2(\mathbb R^d) \times \mathbb R^d,$ $\left|\frac{\delta F}{\delta \mu}(\mu,x)\right| \leq \kappa\left(1+|x|^2\right),$ and for all $\mu' \in \mathcal{P}_2(\mathbb R^d),$
\begin{equation}
\label{eq:flat-der}
\lim_{\varepsilon \searrow 0 }\frac{F(\mu^\varepsilon)- F(\mu)}{\varepsilon} =
\int_{\mathbb R^d} \frac{\delta F}{\delta \mu} (\mu, x) (\mu'-\mu) (\mathrm{d}x), \quad \textnormal{with $\mu^\varepsilon =\mu + \varepsilon (\mu' - \mu)$\,,}
\end{equation}
and $\int_{\mathbb R^d} \frac{\delta F}{\delta \mu} (\mu, x) \mu(\mathrm{d}x)=0.$
\end{definition}
\begin{remark}
\label{rmk:FTC}
One can show that if $F:\mathcal{P}_2(\mathbb R^d) \to \mathbb R$ admits a flat derivative $\frac{\delta F }{\delta \mu},$ then for all $\mu,\mu'\in \mathcal{P}_2(\mathbb R^d),$ the function $[0,1]\ni \varepsilon \mapsto F(\mu^\varepsilon) $ is
continuous on $[0,1]$ and differentiable on $(0,1)$ with derivative 
$\frac{\mathrm{d}}{\mathrm{d} \varepsilon}F(\mu^\varepsilon) = \int_{\mathbb R^d} \frac{\delta F}{\delta \mu} (\mu^\varepsilon, x) (\mu'-\mu) (\mathrm{d}x)$ (see \cite[Theorem 2.3]{jourdain}).
Hence, by the fundamental theorem of calculus, $F(\mu')-F(\mu)=\int_0^1 \int_{\mathbb R^d} \frac{\delta F}{\delta \mu} (\mu^\varepsilon, x) (\mu'-\mu) (\mathrm{d}x)\mathrm{d}\varepsilon,$
provided that $\varepsilon \mapsto \int  \frac{\delta F}{\delta \mu} (\mu^\varepsilon, x) (\mu'-\mu)(\mathrm{d}x)$ is integrable.
\end{remark}
Recall that the tangent space of $\mathcal{P}_2(\mathbb R^d)$ at $\mu \in \mathcal{P}_2(\mathbb R^d)$ is defined as 
\begin{equation*}
\mathcal{T}_{\mu}\mathcal{P}_2(\mathbb R^d) = \overline{\left\{\nabla \psi: \psi \in C_c^{\infty}(\mathbb R^d) \right\} } \subset L_\mu^2(\mathbb R^d),
\end{equation*}
where the closure is taken in $L_\mu^2(\mathbb R^d),$ see \cite[Definition 8.4.1]{ambrosio2008gradient}, and $C_c^{\infty}(\mathbb R^d)$ denotes the space of smooth functions with compact support in $\mathbb R^d.$
\begin{definition}[Wasserstein sub- and super-differential on $\mathcal{P}_2(\mathbb R^d)$]
\label{def:wass-sub-and-sup}
Let $F:\mathcal{P}_2(\mathbb R^d) \to \mathbb R$ and let $\mu \in \mathcal{P}_2(\mathbb R^d).$ Then
\begin{enumerate}
\item a map $\xi \in \mathcal{T}_{\mu}\mathcal{P}_2(\mathbb R^d)$ belongs to the sub-differential $\partial^-F(\mu)$ of $F$ at $\mu \in \mathcal{P}_2(\mathbb R^d)$ if for all $\mu' \in \mathcal{P}_2(\mathbb R^d),$
\begin{equation}
\label{eq:subdifferential}
F(\mu') \geq F(\mu) + \sup_{\gamma \in \Gamma_o(\mu,\mu')} \int_{\mathbb R^d \times \mathbb R^d} \langle \xi(x), y-x\rangle\ \mathrm{d}\gamma(x,y) + o\left(\mathcal{W}_2(\mu,\mu')\right).
\end{equation}
If $\partial^-F(\mu) \neq \emptyset,$ we say the function $F$ is Wasserstein sub-differentiable at $\mu.$
\item A map $\xi \in \mathcal{T}_{\mu}\mathcal{P}_2(\mathbb R^d)$ belongs to the super-differential $\partial^+F(\mu)$ of $F$ at $\mu \in \mathcal{P}_2(\mathbb R^d)$ if $-\xi\in\partial^-(-F)(\mu).$ If $\partial^+F(\mu) \neq \emptyset,$ we say the function $F$ is Wasserstein super-differentiable at $\mu.$
\end{enumerate}
\end{definition}
Then, we say that a function is Wasserstein differentiable if it admits sub- and super- differentials which coincide.
\begin{definition}[Wasserstein differentiability on $\mathcal{P}_2(\mathbb R^d)$]
\label{def:wass-differentiability}
We say that a function $F:\mathcal{P}_2(\mathbb R^d)\to\mathbb R$ is Wasserstein differentiable at $\mu \in \mathcal{P}_2(\mathbb R^d)$ if $\partial^-F(\mu) \cap \partial^+F(\mu) \neq \emptyset.$
\end{definition}
If $F:\mathcal{P}_2(\mathbb R^d) \to \mathbb R$ is Wasserstein differentiable at $\mu \in \mathcal{P}_2(\mathbb R^d)$ (cf. Definition \ref{def:wass-differentiability}), then by \cite[Proposition 5.63]{Carmona2018ProbabilisticTO}, there exists a unique map $\nabla_\mu F(\mu) \in \mathcal{T}_{\mu}\mathcal{P}_2(\mathbb R^d)$ such that $\partial^-F(\mu) = \partial^+F(\mu)= \left\{\nabla_\mu F(\mu)\right\},$ called the Wasserstein gradient of $F$ at $\mu \in \mathcal{P}_2(\mathbb R^d),$ satisfying for any $\mu' \in \mathcal{P}_2(\mathbb R^d),$ and $\gamma \in \Gamma_o(\mu,\mu'),$
\begin{equation*}
F(\mu') = F(\mu) + \int_{\mathbb R^d \times \mathbb R^d} \langle \nabla_\mu F(\mu)(x), y-x\rangle\ \mathrm{d}\gamma(x,y) + o\left(\mathcal{W}_2(\mu,\mu')\right).
\end{equation*}
Finally, we recall the definition of the slope of a function $F$ defined on $ \mathcal{P}_2(\mathbb R^d)$ (see also \cite[Definition 1.2.4]{ambrosio2008gradient}).
\begin{definition}
\label{def:slope}
    Let $F:\mathcal{P}_2(\mathbb R^d) \to \mathbb R$ and let $\mu \in \mathcal{P}_2(\mathbb R^d).$ The slope of $F$ at $\mu$ is defined by
    \begin{equation*}
        |\mathfrak d F|(\mu) \coloneqq \limsup_{\nu \to \mu} \frac{\left(F(\mu)-F(\nu)\right)_{+}}{\mathcal{W}_2(\nu, \mu)},
    \end{equation*}
    where $a_+ \coloneqq \max(a,0).$
\end{definition}

\section*{Acknowledgements} 
R-AL was partly supported by the EPSRC Centre for Doctoral Training in Mathematical Modelling, Analysis and Computation (MAC-MIGS), funded by the UK Engineering and Physical Sciences Research Council (grant EP/S023291/1), Heriot-Watt University, and the University of Edinburgh. This work was initiated while R-AL was a PhD student at Heriot-Watt University and completed during a postdoctoral appointment at RIKEN AIP. \L S acknowledges the support of the UKRI Prosperity Partnership Scheme (FAIR) under EPSRC Grant EP/V056883/1 and the Alan Turing Institute.

\bibliographystyle{abbrv}
\bibliography{references}

@misc{jabir2021meanfieldneuralodesrelaxed,
      title={Mean-Field Neural {ODE}s via Relaxed Optimal Control}, 
      author={Jean-François Jabir and David Šiška and Łukasz Szpruch},
      year={2021},
      note={arXiv:1912.05475},
      primaryClass={math.PR}
}

@article{wang2001,
 author = {Feng-Yu Wang},
 journal = {Journal of Operator Theory},
 number = {1},
 pages = {183--197},
 publisher = {Theta Foundation},
 title = {Logarithmic {S}obolev inequalities: Conditions and counterexamples},
 volume = {46},
 year = {2001}
}

@article{santambrogio,
author = {Santambrogio, Filippo},
year = {2016},
month = {09},
pages = {},
title = {{ Euclidean, Metric, and Wasserstein } Gradient Flows: an overview},
volume = {7},
journal = {Bulletin of Mathematical Sciences}
}

@article{gangbo-mccann,
author = {Wilfrid Gangbo and Robert J. McCann},
title = {{The geometry of optimal transportation}},
volume = {177},
journal = {Acta Mathematica},
number = {2},
publisher = {Institut Mittag-Leffler},
pages = {113 -- 161},
year = {1996}
}

@book{villani2008optimal,
  title={Optimal Transport: Old and New},
  author={Villani, C.},
  isbn={9783540710509},
  lccn={2008932183},
  series={Grundlehren der mathematischen Wissenschaften},
  year={2008},
  publisher={Springer Berlin Heidelberg}
}

@book{Carmona2018ProbabilisticTO,
  title={Probabilistic Theory of Mean Field Games with Applications I: Mean Field FBSDEs, Control, and Games},
  author={Ren{\'e} A. Carmona and François Delarue},
  year={2018},
publisher={Springer International Publishing}
}

@article{jko,
author = {Jordan, Richard and Kinderlehrer, David and Otto, Felix},
title = {The Variational Formulation of the {F}okker--{P}lanck Equation},
journal = {SIAM Journal on Mathematical Analysis},
volume = {29},
number = {1},
pages = {1-17},
year = {1998}
}

@book{bogachev2007measure,
  title={Measure Theory: Volume 1},
  author={Bogachev, V.I.},
  isbn={9783540345145},
  lccn={2006933997},
  year={2007},
  publisher={Springer Berlin Heidelberg}
}

@book{Goodfellow-et-al-2016,
    title={Deep Learning},
    author={Ian Goodfellow and Yoshua Bengio and Aaron Courville},
    publisher={MIT Press},
    note={\url{http://www.deeplearningbook.org}},
    year={2016}
}

@book{Bishop:DeepLearning24,
author = {Christopher M. Bishop and Hugh Bishop},
title = {Deep Learning: Foundations and Concepts},
year = {2024},
publisher = {Springer}
}

@InProceedings{yamamoto24a,
  title = 	 {Mean Field {L}angevin Actor-Critic: Faster Convergence and Global Optimality beyond Lazy Learning},
  author =       {Yamamoto, Kakei and Oko, Kazusato and Yang, Zhuoran and Suzuki, Taiji},
  booktitle = 	 {Proceedings of the 41st International Conference on Machine Learning},
  pages = 	 {55706--55738},
  year = 	 {2024},
  volume = 	 {235},
  series = 	 {Proceedings of Machine Learning Research},
  month = 	 {21--27 Jul},
  publisher =    {PMLR}
}

@inproceedings{
mokrov2021largescale,
title={Large-Scale {W}asserstein Gradient Flows},
author={Petr Mokrov and Alexander Korotin and Lingxiao Li and Aude Genevay and Justin Solomon and Evgeny Burnaev},
booktitle={Advances in Neural Information Processing Systems},
year={2021}
}

@ARTICLE{cheng,
  author={Cheng, Xiuyuan and Lu, Jianfeng and Tan, Yixin and Xie, Yao},
  journal={IEEE Transactions on Information Theory}, 
  title={Convergence of Flow-Based Generative Models via Proximal Gradient Descent in {W}asserstein Space}, 
  year={2024},
  volume={70},
  number={11},
  pages={8087-8106}
}

@article{CarrilloCraigPatacchini2019,
  author    = {Carrillo, J. A. and Craig, K. and Patacchini, F. S.},
  title     = {A blob method for diffusion},
  journal   = {Calculus of Variations and Partial Differential Equations},
  year      = {2019},
  volume    = {58},
  number    = {53}
}

@article{Carrillo2022PrimalDual,
  author    = {Carrillo, J. A. and Craig, K. and Wang, L. and Wei, C.},
  title     = {Primal Dual Methods for {W}asserstein Gradient Flows},
  journal   = {Foundations of Computational Mathematics},
  year      = {2022},
  volume    = {22},
  number    = {2},
  pages     = {389--443}
}

@misc{natale2020tpfafinitevolumeapproximation,
      title={{TPFA} Finite Volume Approximation of {W}asserstein Gradient Flows}, 
      author={Andrea Natale and Gabriele Todeschi},
      year={2020},
      note={arXiv:2001.07005},
      archivePrefix={arXiv},
      primaryClass={math.NA}
}

@misc{hraivoronska2025convergencefullydiscretejko,
      title={Convergence of the fully discrete {JKO} scheme}, 
      author={Anastasiia Hraivoronska and Filippo Santambrogio},
      year={2025},
      note={arXiv:2504.13513},
      archivePrefix={arXiv},
      primaryClass={math.AP}
}

@article{CancesGallouetTodeschi2020,
  author    = {Cancès, C. and Gallouët, T. O. and Todeschi, G.},
  title     = {A variational finite volume scheme for {W}asserstein gradient flows},
  journal   = {Numerische Mathematik},
  year      = {2020},
  volume    = {146},
  pages     = {437--480}
}

@article{Yao,
author = {Yao, Rentian and Chen, Xiaohui and Yang, Yun},
title = {Wasserstein proximal coordinate gradient algorithms},
year = {2024},
volume = {25},
number = {1},
issn = {1532-4435},
journal = {J. Mach. Learn. Res.},
articleno = {269},
numpages = {66}
}

@article{10.1214/20-AIHP1140,
author = {Kaitong Hu and Zhenjie Ren and David \v{S}i\v{s}ka and {\L}ukasz Szpruch},
title = {{Mean-field Langevin dynamics and energy landscape of neural networks}},
volume = {57},
journal = {Annales de l'Institut Henri Poincar\'e, Probabilit\'es et Statistiques},
number = {4},
publisher = {Institut Henri Poincar\'e},
pages = {2043 -- 2065},
year = {2021}
}

@article{chizat2022meanfield,
title={Mean-Field {L}angevin Dynamics : Exponential Convergence and Annealing},
author={L{\'e}na{\"\i}c Chizat},
journal={Transactions on Machine Learning Research},
issn={2835-8856},
year={2022}
}

@inproceedings{suzuki2023meanfield,
title={Mean-field Langevin dynamics: Time-space discretization, stochastic gradient, and variance reduction},
author={Taiji Suzuki and Denny Wu and Atsushi Nitanda},
booktitle={Thirty-seventh Conference on Neural Information Processing Systems},
year={2023}
}

@InProceedings{Nitanda2022ConvexAO,
  title = 	 {Convex Analysis of the Mean Field {L}angevin Dynamics},
  author =       {Nitanda, Atsushi and Wu, Denny and Suzuki, Taiji},
  booktitle = 	 {Proceedings of The 25th International Conference on Artificial Intelligence and Statistics},
  year = 	 {2022}
}

@article{OTTO2000361,
title = {Generalization of an Inequality by {T}alagrand and Links with the Logarithmic {S}obolev Inequality},
journal = {Journal of Functional Analysis},
volume = {173},
number = {2},
pages = {361-400},
year = {2000},
issn = {0022-1236},
author = {F. Otto and C. Villani}
}

@article{Rotskoff2018TrainabilityAA,
  title={Trainability and Accuracy of Artificial Neural Networks: An Interacting Particle System Approach},
  author={Grant M. Rotskoff and Eric Vanden-Eijnden},
  journal={Communications on Pure and Applied Mathematics},
  year={2018},
  volume={75}
}

@misc{huang2024generativemodelingminimizingwasserstein2,
      title={Generative Modeling by Minimizing the {W}asserstein-2 Loss}, 
      author={Yu-Jui Huang and Zachariah Malik},
      year={2024},
      note={arXiv:2406.13619},
      primaryClass={stat.ML}
}

@article{Benamou,
  author = "Jean-David Benamou,Guillaume Carlier,Maxime Laborde",
  title = "An augmented {L}agrangian approach to {W}asserstein gradient flows and applications",
  journal = "ESAIM: Proceedings and Surveys",
  year = "2016",
  volume = "54",
  number = "1",
  pages = "1-17"
}

@article{Drusvyatskiy2016EfficiencyOM,
  title={Efficiency of minimizing compositions of convex functions and smooth maps},
  author={Dmitriy Drusvyatskiy and Courtney Paquette},
  journal={Mathematical Programming},
  year={2016},
  pages={1-56}
}

@article{sirignano,
author = {Sirignano, Justin and Spiliopoulos, Konstantinos},
title = {Mean Field Analysis of Neural Networks: A Law of Large Numbers},
journal = {SIAM Journal on Applied Mathematics},
volume = {80},
number = {2},
pages = {725-752},
year = {2020}
}

@misc{wang2024uniformlogsobolevinequalitiesmean,
      title={Uniform log-{S}obolev inequalities for mean field particles with flat-convex energy}, 
      author={Songbo Wang},
      year={2024},
      note={arXiv:2408.03283},
      primaryClass={math.PR}
}

@misc{chewi2024uniforminnlogsobolevinequalitymeanfield,
      title={Uniform-in-$N$ log-{S}obolev inequality for the mean-field {L}angevin dynamics with convex energy}, 
      author={Sinho Chewi and Atsushi Nitanda and Matthew S. Zhang},
      year={2024},
      note={arXiv:2409.10440},
      archivePrefix={arXiv},
      primaryClass={math.PR}
}

@article{pierre,
author = {Pierre Monmarch{\'e} and Zhenjie Ren and Songbo Wang},
title = {{Time-uniform log-{S}obolev inequalities and applications to propagation of chaos}},
volume = {29},
journal = {Electronic Journal of Probability},
number = {none},
publisher = {Institute of Mathematical Statistics and Bernoulli Society},
pages = {1 -- 38},
year = {2024}
}

@article{guillin,
author = {Patrick Cattiaux and Arnaud Guillin},
title = {{Functional inequalities for perturbed measures with applications to log-concave measures and to some Bayesian problems}},
volume = {28},
journal = {Bernoulli},
number = {4},
publisher = {Bernoulli Society for Mathematical Statistics and Probability},
pages = {2294 -- 2321},
year = {2022}
}

@article{parikh,
author = {Parikh, Neal and Boyd, Stephen},
title = {Proximal Algorithms},
year = {2014},
issue_date = {January 2014},
publisher = {Now Publishers Inc.},
address = {Hanover, MA, USA},
volume = {1},
number = {3},
issn = {2167-3888},
journal = {Found. Trends Optim.},
month = {01},
pages = {127–239},
numpages = {113}
}

@article{Benamou2014DiscretizationOF,
  title={Discretization of functionals involving the {M}onge–{A}mp{\`e}re operator},
  author={Jean-David Benamou and Guillaume Carlier and Quentin M{\'e}rigot and {\'E}douard Oudet},
  journal={Numerische Mathematik},
  year={2014},
  volume={134},
  pages={611-636}
}

@InProceedings{leahy,
  title = 	 {Convergence of Policy Gradient for Entropy Regularized {MDP}s with Neural Network Approximation in the Mean-Field Regime},
  author =       {Leahy, James-Michael and Kerimkulov, Bekzhan and \v{S}i\v{s}ka, David and Szpruch, {\L}ukasz},
  booktitle = 	 {Proceedings of the 39th International Conference on Machine Learning},
  pages = 	 {12222--12252},
  year = 	 {2022},
  volume = 	 {162},
  series = 	 {Proceedings of Machine Learning Research},
  month = 	 {17--23 Jul},
  publisher =    {PMLR}
}

@inproceedings{arbel,
 author = {Arbel, Michael and Korba, Anna and Salim, Adil and Gretton, Arthur},
 booktitle = {Advances in Neural Information Processing Systems},
 pages = {},
 title = {Maximum Mean Discrepancy Gradient Flow},
 volume = {32},
 year = {2019}
}

@InProceedings{ruiyi,
  title = 	 {Policy Optimization as {W}asserstein Gradient Flows},
  author =       {Zhang, Ruiyi and Chen, Changyou and Li, Chunyuan and Carin, Lawrence},
  booktitle = 	 {Proceedings of the 35th International Conference on Machine Learning},
  pages = 	 {5737--5746},
  year = 	 {2018},
  volume = 	 {80},
  series = 	 {Proceedings of Machine Learning Research},
  month = 	 {10--15 Jul},
  publisher =    {PMLR}
}

@misc{chen2023uniformintimepropagationchaosmean,
      title={Uniform-in-time propagation of chaos for mean field {L}angevin dynamics}, 
      author={Fan Chen and Zhenjie Ren and Songbo Wang},
      year={2023},
      note={arXiv:2212.03050},
      primaryClass={math.PR},
}

@inproceedings{Wibisono2018SamplingAO,
  title={Sampling as optimization in the space of measures: The {L}angevin dynamics as a composite optimization problem},
  author={Andre Wibisono},
  booktitle={Conference on Learning Theory},
  year={2018}
}

@article{zhenjiefict,
  author  = {Fan Chen and Zhenjie Ren and Songbo Wang},
  title   = {Entropic Fictitious Play for Mean Field Optimization Problem},
  journal = {Journal of Machine Learning Research},
  year    = {2023},
  volume  = {24},
  number  = {211},
  pages   = {1--36}
}

@misc{nitanda2017stochasticparticlegradientdescent,
      title={Stochastic Particle Gradient Descent for Infinite Ensembles}, 
      author={Atsushi Nitanda and Taiji Suzuki},
      year={2017},
      eprint={1712.05438},
      note={arXiv:1712.05438},
      primaryClass={stat.ML}, 
}

@inproceedings{korbaproximal,
 author = {Salim, Adil and Korba, Anna and Luise, Giulia},
 booktitle = {Advances in Neural Information Processing Systems},
 pages = {12356--12366},
 publisher = {Curran Associates, Inc.},
 title = {The {W}asserstein Proximal Gradient Algorithm},
 volume = {33},
 year = {2020}
}

@article{
teter2024proximal,
title={Proximal Mean Field Learning in Shallow Neural Networks},
author={Alexis Teter and Iman Nodozi and Abhishek Halder},
journal={Transactions on Machine Learning Research},
issn={2835-8856},
year={2024}
}

@misc{xu2024forwardeulertimediscretizationwassersteingradient,
      title={Forward-{E}uler time-discretization for {W}asserstein gradient flows can be wrong}, 
      author={Yewei Xu and Qin Li},
      year={2024},
      note={arXiv:2406.08209},
      primaryClass={stat.ML}
}

@inproceedings{
luu2024nongeodesicallyconvex,
title={Non-geodesically-convex optimization in the {W}asserstein space},
author={Hoang Phuc Hau Luu and Hanlin Yu and Bernardo Williams and Petrus Mikkola and Marcelo Hartmann and Kai Puolam{\"a}ki and Arto Klami},
booktitle={The 38th Annual Conference on Neural Information Processing Systems},
year={2024}
}

@article{jourdain,
author = {Benjamin Jourdain and Alvin Tse},
title = {{Central limit theorem over non-linear functionals of empirical measures with applications to the mean-field fluctuation of interacting diffusions}},
volume = {26},
journal = {Electronic Journal of Probability},
publisher = {Institute of Mathematical Statistics and Bernoulli Society},
pages = {1 -- 34},
year = {2021}
}

@article{Holley1987LogarithmicSI,
  title={Logarithmic {S}obolev inequalities and stochastic {I}sing models},
  author={Richard Holley and Daniel W. Stroock},
  journal={Journal of Statistical Physics},
  year={1987},
  volume={46},
  pages={1159-1194}
}

@book{cardaliaguet2019master,
  title={The Master Equation and the Convergence Problem in Mean Field Games},
  author={Cardaliaguet, P. and Delarue, F. and Lasry, J.M. and Lions, P.L.},
  isbn={9780691190716},
  series={Annals of Mathematics Studies},
  year={2019},
  publisher={Princeton University Press}
}

@book{pDUP97a,
        AUTHOR    ="P. Dupuis and R. S. Ellis",
        TITLE     ={A Weak Convergence Approach to the Theory of Large Deviations},
        PUBLISHER ={Wiley},
        ADDRESS   ={New York, NY},
        YEAR      ={1997}  }

@inproceedings{Chizat2018OnTG,
  title={On the Global Convergence of Gradient Descent for Over-parameterized Models using Optimal Transport},
  author={L{\'e}na{\"i}c Chizat and Francis R. Bach},
  booktitle={NeurIPS},
  year={2018}
}

@article{Mei2018AMF,
  title={A mean field view of the landscape of two-layer neural networks},
  author={Song Mei and Andrea Montanari and Phan-Minh Nguyen},
  journal={Proceedings of the National Academy of Sciences of the United States of America},
  year={2018},
  volume={115},
  pages={E7665 - E7671}
}

@misc{Rotskoff2018NeuralNA,
  title={Neural Networks as Interacting Particle Systems: Asymptotic Convexity of the Loss Landscape and Universal Scaling of the Approximation Error},
  author={Grant M. Rotskoff and Eric Vanden-Eijnden},
  year={2018},
  note={arXiv:1805.00915}
}

@book{ambrosio2008gradient,
  title={Gradient Flows: In Metric Spaces and in the Space of Probability Measures},
  author={Ambrosio, L. and Gigli, N. and Savare, G.},
  isbn={9783764387228},
  lccn={2008921489},
  series={Lectures in Mathematics. ETH Z{\"u}rich},
  url={https://books.google.co.uk/books?id=rCDK9JA5BAEC},
  year={2008},
  publisher={Birkh{\"a}user Basel}
}

@book{santambrogio2015optimal,
  title={Optimal Transport for Applied Mathematicians: Calculus of Variations, PDEs, and Modeling},
  author={Santambrogio, F.},
  isbn={9783319208282},
  series={Progress in Nonlinear Differential Equations and Their Applications},
  year={2015},
  publisher={Springer International Publishing}
}

@article{CattiauxGuillinWu2009,
  title        = {A note on {T}alagrand’s transportation inequality and logarithmic {S}obolev inequality},
  author       = {Cattiaux, Patrick and Guillin, Arnaud and Wu, Li-Ming},
  journal      = {Probability Theory and Related Fields},
  volume       = {148},
  number       = {3–4},
  pages        = {285--304},
  year         = {2010},
  month        = {June}
}

@misc{zhu2025convergenceanalysiswassersteinproximal,
      title={Convergence Analysis of the {W}asserstein Proximal Algorithm beyond Geodesic Convexity}, 
      author={Shuailong Zhu and Xiaohui Chen},
      year={2025},
      note={arXiv:2501.14993},
      primaryClass={math.OC},
}
\end{document}